\documentclass[a4paper,11pt,reqno]{article}
\usepackage[utf8]{inputenc}
\usepackage[T1]{fontenc}
\usepackage{lmodern}
\usepackage[english]{babel}
\usepackage{microtype}
\usepackage{pst-node}
\usepackage{dsfont}
\usepackage{aligned-overset}

\usepackage{amsmath,amssymb,amsfonts,amsthm}
\usepackage{mathtools,accents}
\usepackage{mathrsfs}
\usepackage{aliascnt}
\usepackage{braket}
\usepackage{bm}

\usepackage[a4paper,margin=3cm]{geometry}
\usepackage[citecolor=blue,colorlinks]{hyperref}

\usepackage{enumerate}
\usepackage{xcolor}
\usepackage{tikz-cd}

\makeatletter
\g@addto@macro\@floatboxreset\centering
\makeatother

\makeatletter
\def\newaliasedtheorem#1[#2]#3{
  \newaliascnt{#1@alt}{#2}
  \newtheorem{#1}[#1@alt]{#3}
  \expandafter\newcommand\csname #1@altname\endcsname{#3}
}
\makeatother

\numberwithin{equation}{section}

\newtheoremstyle{slanted}{\topsep}{\topsep}{\slshape}{}{\bfseries}{.}{.5em}{}

\theoremstyle{plain}
\newtheorem{theorem}{Theorem}[section]
\newaliasedtheorem{proposition}[theorem]{Proposition}
\newaliasedtheorem{lemma}[theorem]{Lemma}
\newaliasedtheorem{corollary}[theorem]{Corollary}

\theoremstyle{definition}
\newaliasedtheorem{definition}[theorem]{Definition}
\newaliasedtheorem{example}[theorem]{Example}

\theoremstyle{remark}
\newaliasedtheorem{remark}[theorem]{Remark}

\theoremstyle{question}
\newaliasedtheorem{question}[theorem]{Question}

\begin{document}
\title{A $C^m$ Whitney Extension Theorem for Horizontal Curves in Free Step $2$ Carnot Groups}

\author{Hyogo Shibahara
\thanks{Department of Mathematical Sciences, University of Cincinnati, Cincinnati, OH 45221, U.S.A., \url{shibahho@mail.uc.edu}.}} 
\maketitle

\begin{abstract}
We establish a $C^m$ Whitney extension theorem for horizontal curves in free step~$2$ Carnot groups $\mathbb{G}_r$ for an arbitrary number of generators  $r \geq 2$. This extends existing results in the Heisenberg group. New techniques include a quantitative linear dependence structure on jets on the horizontal layer, a discretization process to correct vertical areas, and studying the effect of unrelated vertical areas.
\end{abstract}


\allowdisplaybreaks{

\section{Introduction}
Let $K \subseteq \mathbb{R}$ be a compact subset.  Whitney introduced an intrinsic notion of  the class $C^m(K)$ through jets that satisfy an appropriate convergence rate for the remainder term of the Taylor polynomial of a jet, see Definition~\ref{def of whitney fields}. This definition of the class $C^m(K)$ coincides with the extrinsic definition of the class $C^m(K)$, i.e., $C^m(K) = C^m(\mathbb{R})|_{K}$. This is now called the Whitney extension theorem. The precise statement can be found in \cite{W} and Theorem~\ref{Classical whitney extension}.

Over the past decade, a number of $C^1$ Whitney extension theorems for horizontal curves in Carnot groups and sub-Riemannian manifolds have been established, see \cite{JS}, \cite{SZ}, \cite{Z1} and \cite{SS}. The first result in this direction was made by Zimmerman for $C^1$ horizontal curves in the Heisenberg group, see \cite{Z1} and \cite{Sp} for a related result. Inspired by \cite{M}, Juillet and Sigalotti defined an intrinsic notion of the class $C^1_{H}(K)$ through the coordinate-free $C^1$ Taylor polynomial for horizontal curves and established a $C^1$ Whitney extension theorem for horizontal curves in pliable Carnot groups. Moreover, they examined exactly when the equality $C^1_{H}(K)=C^1_{H}(\mathbb{R})|_{K}$ holds with their Whitney condition through the geometric property called pliability.

On the other hand, a $C^m$ Whitney extension theorem for horizontal curves in Carnot groups is poorly understood for $m>1$. The only known result is in the setting of the Heisenberg group by Pinamonti, Speight and Zimmerman, see \cite{PSZ}. To the best of our knowledge, an appropriate definition of a  coordinate-free $C^m$ Taylor polynomial for a horizontal curve is unknown for $m\geq 2$, while Juillet and Sigalotti found a canonical one in the $C^1$ case.

One notable application of the Whitney extension theorem is Lusin approximation, which allows one to approximate a given mapping by a smoother mapping outside a set of small measure. From \cite[Section~5]{JS}, we know that the $C^1$ Lusin approximation holds for pliable Carnot groups.  For the higher order case, Cappoli, Pinamonti and Speight established a $C^m$ Lusin approximation theorem for horizontal curves through a $C^m$ Whitney extension theorem for horizontal curves in the Heisenberg group and the $L^1$ differentiability of the first derivatives of functions on the horizontal layer, see \cite[Theorem~4.1]{CPS}. Whitney extension theorem also has other applications, for instance, in the study of rectifiability of the reduced boundary of a set of finite perimeter (\cite{FSC}), see also \cite[Remark~5.3]{JS} for another application. For this reason, it is important to explore the Whitney extension problem.

In order to better understand the $C^m$ Whitney extension problem for horizontal curves in Carnot groups, we study the setting of free step~$2$ Carnot groups with $r$ generators, denoted by $\mathbb{G}_r$. Every free step~$2$ Carnot group admits an explicit representation in $\mathbb{R}^N$ such that the group law for each vertical component is defined similar to that of the vertical component in the first Heisenberg group $\mathbb{H}^1$ (up to some constant and sign), see Definition~\ref{freegroupstep2}. Namely, each vertical component is given by the signed area on the associated plane, see Lemma~\ref{ij-th component representation}. Hence one could ask if a  generalization of the conditions which appeared in \cite{PSZ} gives the desired $C^m$ Whitney extension theorem for horizontal curves in $\mathbb{G}_r$.

The motivation to investigate the setting of $\mathbb{G}_r$ is that they form a much richer class of Carnot groups than the Heisenberg groups while the group law of each vertical component is that of the (first) Heisenberg group $\mathbb{H}^1$. Surprisingly, the direct generalization of the $C^m$ Whitney conditions from the Heisenberg group does \emph{not} work in $\mathbb{G}_3$, see Section~\ref{Counterexample subsection}. This example in $\mathbb{G}_3$ shows that the interplay among vertical areas needs to be taken into account and makes the $C^m$ Whitney extension problem in $\mathbb{G}_r$ challenging.

One of the new contributions in the present paper is the generalized $A/V$ condition that takes into account the linear dependence between jets on the horizontal layer, see Definition~\ref{Definition of generalized A/V} and Definition~\ref{The generalized A/V in Gr} for the $\mathbb{G}_3$ case and the general $\mathbb{G}_r$ case, respectively. Roughly speaking, the generalized $A/V$ condition captures the fact that if the jets in the horizontal layer are linearly dependent in some quantitative way, then the jets in the vertical layers must also be related. Since the Heisenberg groups (or $\mathbb{G}_2$) have only one vertical component, this phenomenon was hidden. We note that this generalized $A/V$ condition is stronger than the vertical component-wise $A/V$ condition that looks rather natural from the setting of the Heisenberg group, see Section~\ref{Counterexample subsection}. While it is not hard to see that $C^m$ horizontal curves satisfy the generalized $A/V$ condition (Propositions~\ref{New Condition} and \ref{necessary Gr}), it is far from obvious that jets satisfying the generalized $A/V$ condition extend to a $C^m$ horizontal curve. In the beginning of our proof of sufficiency, we follow the strategy employed in the Heisenberg group \cite{PSZ}. Namely, we use the classical Whitney extension theorem to obtain extensions of jets on the horizontal layer. Then we perturb the obtained extensions slightly outside $K$ to correct the vertical areas. As in the case of the Heisenberg group we employ techniques such as the Markov inequality for polynomials (Lemma~\ref{Classical Markov inequality}), but our arguments depart from their proof when constructing perturbations to correct multiple vertical areas.

In Sections~\ref{Necessity G3 Section} and \ref{Sufficiency G3 Section}, we first restrict our attention to $\mathbb{G}_3$ and establish  the $C^m$ Whitney extension theorem for horizontal curves in $\mathbb{G}_3$, see Theorem~\ref{Cm Whitney G3}. The significant difference from \cite{PSZ} when constructing perturbations is the way that we fix the last vertical component. For instance, after a suitable ordering of jets, one can fix the $21$ and $31$ areas using techniques similar to \cite{PSZ}. However, one must fix the $32$ area without affecting the $21$ and $31$ areas. The key step is a certain discretization process which allows us to work with $\mathbb{R}^N$ and use elementary linear algebra in order to find appropriate perturbations. We emphasize that we lose some information about the original Whitney extensions of jets on the horizontal layer by the above discretization process. The Markov inequality for polynomials, however, ensures us that the information we lost is somehow negligible.

In Sections~\ref{Necessity Gr Section} and \ref{Sufficiency Gr Section}, we prove a $C^m$ Whitney extension theorem for horizontal curves in $\mathbb{G}_r$ (Theorem~\ref{Cm Whitney Gr}) by mathematical induction with respect to $r\geq 3$. Somewhat surprisingly, the generalized $A/V$ condition in $\mathbb{G}_r$ with $r \geq 4$ (Definition~\ref{The generalized A/V in Gr}) cares about the interplay among \emph{all} vertical areas, which one would not expect from the group law of $\mathbb{G}_r$. Although there are some technical difficulties to establish the $C^m$ Whitney extension theorem for horizontal curves in $\mathbb{G}_r$, essential ideas, such as the linear dependence of jets and a discretization process play important roles again.

This paper is organized as follows. In Section~\ref{Preliminaries}, we recall some basic definitions and facts about Whitney extension theorems, the Markov inequality for polynomials and Carnot groups. In Section~\ref{Counterexample subsection}, we provide a family of jets of order $2$ in $\mathbb{G}_3$ that satisfies the natural conditions generalized from \cite{PSZ} but does not admit a $C^2$ extension (Proposition~\ref{Example}). In Section~\ref{Necessity G3 Section}, we define the generalized $A/V$ condition in $\mathbb{G}_3$ (Definition~\ref{Definition of generalized A/V}) and prove that all $C^m$ horizontal curves in $\mathbb{G}_3$ satisfy this condition (Proposition~\ref{New Condition}). In Section~\ref{Sufficiency G3 Section}, we prove that this condition is sufficient (Proposition~\ref{Our Goal}) and hence establish a $C^m$ Whitney extension theorem for horizontal curves in $\mathbb{G}_3$ (Theorem~\ref{Cm Whitney G3}). In Section~\ref{Necessity Gr Section}, we extend our notion of the generalized $A/V$ condition to the setting of $\mathbb{G}_r$ (Definition~\ref{The generalized A/V in Gr}) and check that all $C^m$ horizontal curves in $\mathbb{G}_r$ satisfy this condition (Proposition~\ref{necessary Gr}). Finally, in Section~\ref{Sufficiency Gr Section}, we prove sufficiency (Proposition~\ref{Our Ultimate Goal}) and hence establish a $C^m$ Whitney extension theorem for horizontal curves in $\mathbb{G}_r$ (Theorem~\ref{Cm Whitney Gr}).

\subsubsection*{\hfil Acknowledgement \hfil}
The author would like to thank his advisor Gareth Speight for numerous discussions and travel support (Simons Foundation \#576219). He thanks Gareth Speight and Andrea Pinamonti for suggesting Proposition~\ref{Example}. He also thanks Scott Zimmerman for helpful discussion. This research was supported by a Charles Phelps Taft Dissertation Fellowship and a Taft Graduate Summer Fellowship at the University of Cincinnati. The author appreciates their generous financial support throughout the 2022-2023 academic year.

\section{Preliminaries}\label{Preliminaries}

\subsection{Classical Whitney Extension Theorem}
In this subsection we review the classical Whitney extension theorem. Throughout this paper, the $k$-th derivative of a function  $f$ is denoted by $D^k f$ for any $k \in \mathbb{N}$ and we also set $D^0 f:= f$ for convenience. We also use the notation $f'$ for the first derivative of $f$. We start with the notion of jets.
\begin{definition}
		Let $K \subseteq \mathbb{R}$ be a compact subset. A \emph{jet} $F=(F^k)_{k=0, \cdots, m}$ of order $m$ on $K$ is a collection of $m+1$ continuous functions from $K$ to $\mathbb{R}$.
	\end{definition}
	
	Given a jet $F=(F^k)_{k=0, \cdots, m}$, we can define the formal $m$-th order Taylor polynomial at $a \in K$ by
	\begin{equation}
		T_a^m F(x):=\sum_{k=0}^{m} \frac{F^k(a)}{k!}(x-a)^m.
	\end{equation}
	When the order of the Taylor polynomial  $m \in \mathbb{N}$ is clear, we often use the notation $T_a F$  instead of $T^m_a F$ for simplicity. The class $C^m(K)$ is defined by the appropriate convergence rate of the formal remainder term of a given jet.
	\begin{definition}\label{def of whitney fields}
		Let $m\geq 1$ and let $F=(F^k)_{k=0, \cdots, m}$ be a jet on $K$. We say that $F=(F^k)_{k=0, \cdots, m} \in C^m(K)$ if for all $0\leq k\leq m$
		\[
		\limsup_{\substack{ |a-b|\to 0 \\ a, b \in K}} \frac{|(R^m_a F)^k(b)|}{|b-a|^{m-k}}=0 \ \ \text{where} \ \ (R^m_a F)^k(x):=F^k(x)-\sum_{l=0}^{m-k}\frac{F^{k+l}(a)}{l!}(x-a)^l.
		\] 
		We call $F \in C^m(K)$ a \emph{Whitney field of  class $C^m$}.
	\end{definition}
	Note that by Taylor's theorem, for every $f \in C^m(\mathbb{R})$ and compact subset $K \subseteq \mathbb{R}$, the jets defined by $(D^k f|_{K})_{0 \leq k \leq m}$ is a Whitney field of class $C^m$, see \cite{F}. The following is due to Whitney \cite{W}.
	
	\begin{theorem}\label{Classical whitney extension}
		We have $C^m(K)=C^m(\mathbb{R})|_{K}$, i.e., for any $F=(F^k)_{k=0, \cdots, m} \in C^m(K)$, there exists $f \in C^m(\mathbb{R})$ such that $D^kf|_{K}=F^k$ on $K$ for any $0 \leq k \leq m$.
	\end{theorem}

\subsection{Free Step $2$ Carnot Groups}
We define free step $2$ Carnot groups. We refer the readers to \cite[Definition~14.1.3]{BLU} for more details. Fix an integer $r\geq 2$ and denote $n=r+r(r-1)/2$. In $\mathbb{R}^{n}$ denote the coordinates by $x_{i}$, $1\leq i\leq r$, and $x_{ij}$, $1\leq j<i\leq r$. Let $\frac{\partial}{\partial x_{i}}$ and $\frac{\partial}{\partial x_{ij}}$ denote the standard basis vectors in this coordinate system. Define $n$ vector fields on $\mathbb{R}^n$ by:
\[
 X_{k}:=\frac{\partial}{\partial x_{k}}+\sum_{j>k} \frac{x_{j}}{2}\frac{\partial}{\partial x_{jk}}-\sum_{j<k}\frac{x_{j}}{2}\frac{\partial}{\partial x_{kj}} \qquad \mbox{if }1\leq k \leq r,\]
\[X_{kj}:=\frac{\partial}{\partial x_{kj}} \qquad \mbox{if } 1\leq j<k\leq r.
\]

\begin{definition}\label{freegroupstep2}
We define the \emph{free step $2$ Carnot group with $r$ generators} by $\mathbb{G}_r:=(\mathbb{R}^{n}, \, \cdot)$, where the product $x\cdot y \in \mathbb{G}_r$ of $x,y\in \mathbb{G}_r$ is given by:
\[(x\cdot y)_{k} = x_{k}+y_{k} \qquad \mbox{if }1\leq k\leq r,\]
\[(x\cdot y)_{ij} = x_{ij}+y_{ij}+\frac{1}{2}(x_{i}y_{j}-y_{i}x_{j}) \qquad \mbox{if }1\leq j<i\leq r.\]
One can check that $\mathbb{G}_r$ is a Lie group and the cooresponding Lie algebra $\mathfrak{g}$ admits the stratification
\[
\mathfrak{g}=V_1 \oplus V_2
\]
where
\[V_{1}=\mathrm{Span} \{X_{k}\colon 1\leq k \leq r\} \mbox{ and }V_{2}=\mathrm{Span} \{X_{kj}\colon 1\leq j<k\leq r\}.\]
\end{definition}

One can verify that for $1\leq j<k \leq r$ and $1\leq i\leq r$:
\[ 
[X_{k}, X_{j}]=X_{kj} \mbox{ and } [X_{i}, X_{kj}]=0.
\]

\begin{definition}
	An absolutely continuous curve $\gamma: [a,b] \to \mathbb{G}_r$ is \emph{horizontal} if the tangent vector $\gamma'(t)$ lies in the linear span of $\{X_{k}(\gamma(t)) \ | \ 1\leq k \leq r\}$ where $X_{k}(\gamma(t))$ is the tangent vector obtained by evaluating $X_k$ at $\gamma(t)$. 
\end{definition}
A direct computation leads us to the following explicit formula for horizontal curves in $\mathbb{G}_r$.
\begin{lemma}\label{ij-th component representation}
	An absolutely continuous curve $\gamma: [a,b] \to \mathbb{G}_r$ is horizontal in $\mathbb{G}_r$ if and only if 
	\begin{equation}
		\gamma_{ij}(t)-\gamma_{ij}(a)=\frac{1}{2}\int_a^t \gamma_i\gamma_j'-\gamma_i'\gamma_j,
	\end{equation}
	for all $t \in [a,b]$ and all $1 \leq j <i \leq r$.
\end{lemma}
For $k \in \mathbb{N}$, define  
\begin{equation}\label{def of polynomials}
	\mathcal{P}^k(x_0, y_0, x_1, y_1, \cdots, x_k, y_k):=\frac{1}{2} \sum_{l=0}^{k-1}\binom{k-1}{l}(x_l y_{k-l}- y_l x_{k-l}).
\end{equation}
For a $C^m$ horizontal curve $\gamma$,  differentiating the $ij$-th component $\gamma_{ij}$ and Lemma~\ref{ij-th component representation} give us the following.
\begin{lemma}\label{horizontal polynomial}
	For any $C^m$ horizontal curve $\gamma:(a, b) \to \mathbb{G}_r$, we have that
	\[
D^k \gamma_{ij}(t)=\mathcal{P}^k \left(\gamma_{i}(t),\gamma_{j}(t),\gamma_{i}^1(t),\gamma_{j}^1(t),\dots,\gamma_{i}^{k}(t),\gamma_{j}^{k}(t)\right),
\]
for every $t \in (a,b)$, $1\leq j<i\leq r$ and $1 \leq k \leq m$ where $\mathcal{P}^k$ are the polynomials defined in \eqref{def of polynomials}.
\end{lemma}

\subsection{A $C^m$ Whitney Extension Theorem For Horizontal Curves in $\mathbb{G}_2$}

We next review the $C^m$ Whitney extension theorem for horizontal curves in $\mathbb{G}_2$. For jets $F_1$, $F_2$, and $F_{21}$ of order $m$ on $K$, set
\begin{align}\label{Aab G2}
 A_{21}(a,b)&:=F_{21}(b)-F_{21}(a)-\frac{1}{2}\int_{a}^{b}((T_{a}^mF_1)'(T_{a}^mF_2)-(T_{a}^mF_2)'(T_{a}^mF_1))\\ \nonumber 
&\qquad +\frac{1}{2}F_1(a)(F_2(b)-T_{a}^mF_2(b))-\frac{1}{2}F_2(a)(F_1(b)-T_{a}^mF_1(b))
\end{align}
and
\begin{equation}\label{Vab G2}
V_{21}(a,b):= (b-a)^{2m} + (b-a)^{m} \int_{a}^{b}|(T_{a}^mF_1)'|+|(T_{a}^mF_2)'|.
\end{equation}

Pinamonti, Speight and Zimmerman established a $C^m$ Whitney extension theorem for horizontal curves in the Heisenberg group in \cite[Theorem~1.1]{PSZ}. The following is their main result adapted to $\mathbb{G}_2$, which has essentially the same structure as their setting of $\mathbb{H}^1$.
\begin{theorem}\label{iff}
Let $K \subseteq \mathbb{R}$ be a compact set and $F_1$, $F_2$, and $F_{21}$ be jets of order $m$ on $K$. 
Then the triple $(F_1,F_2,F_{21})$ extends to a $C^m$ horizontal curve $(f_1,f_2,f_{21})\colon \mathbb{R} \to \mathbb{G}_2$ if and only if
\begin{enumerate}
\item $F_1$, $F_2$, and $F_{21}$ are Whitney fields of class $C^m$ on $K$,
\item for every $1 \leq k \leq m$ and $t\in K$ we have
 \begin{equation}
\label{HorizAssume}
F_{21}^{k}(t) = \mathcal{P}^k \left(F_1^0(t),F_2^0(t),F_1^1(t),F_2^1(t),\dots,F_1^{k}(t),F_2^{k}(t)\right),
\end{equation}
where  $\mathcal{P}^k$ are the polynomials defined in \eqref{def of polynomials},
\item $A_{21}(a,b)/V_{21}(a,b)\to 0$ uniformly as $(b-a) \searrow 0$ with $a,b\in K$.
\end{enumerate}
\end{theorem}

\subsection{Some Consequences of the Markov Inequality}
In this subsection we collect some consequences of the Markov inequality for polynomials which we use in this paper. 
\begin{lemma}[Markov inequality]\label{Classical Markov inequality}
	Let $p$ be a polynomial of degree $m\geq 1$. Then we have
	\begin{equation}
		\max_{x \in [a,b]} |p'(x)|\leq \frac{2m^2}{b-a}\max_{x \in [a, b]}|p|
	\end{equation}
	for any interval $[a, b]$.
\end{lemma}
Applying the above Markov inequality for polynomials, one can prove the following, see \cite[Lemma~2.10]{PSZ}.
\begin{lemma}\label{nice subinterval}
	Let $p$ be a polynomial of degree $m\geq 1$ and $a<b$. Then there exists a subinterval $I \subseteq [a,b]$ with $|I|\geq (b-a)/(4m^2)$ such that 
	\begin{equation}
		\max_{x \in [a, b]}|p(x)|/2\leq  \min_{x \in I}|p(x)|.
	\end{equation}
\end{lemma}

Let $x, y \in \mathbb{R}$ and $C\geq 1$. We use the notation $x\overset{C}{\simeq} y$ if $(1/C)y \leq x \leq Cy$. The following lemma tells that for any polynomial of degree $m \geq 1$, the $L^1$ and $L^2$ norms with respect to the normalized measure is comparable to the supremum norm where the comparable constant only depends on $m$. This immediately follows from Lemma~\ref{nice subinterval}.

\begin{lemma}\label{L1 L2 equivalent}
	Let $p$ be a polynomial of degree $m \geq 1$ and $a<b$. There exists $\tilde{C}:=\tilde{C}(m)$ such that 
	\begin{equation}\label{L1 L2 equivalent equation}
		\max_{x \in [a,b]} |p(x)| \overset{\tilde{C}}{\simeq} \frac{1}{b-a}\int_a^b |p| \overset{\tilde{C}}{\simeq} \Big(\frac{1}{b-a}\int_a^b |p|^2 \Big)^{1/2}.
	\end{equation}
\end{lemma}

Let $(\mathcal{H}, \langle \ \cdot, \ \cdot \ \rangle_{\mathcal{H}})$ be a Hilbert space. The norm induced from the inner product $ \langle \ \cdot, \ \cdot \ \rangle_{\mathcal{H}}$ is denoted by $||\cdot||_{\mathcal{H}}$. For $v, w \in \mathcal{H}$, we use the notation $v \perp w$ if they are orthogonal to each other. For $v_1, \cdots, v_n \in \mathcal{H}$, the linear span of $v_1, \cdots, v_n$ is denoted by $\langle v_1, \cdots, v_n\rangle$. The orthogonal subspace of a subspace $V$ is denoted by $V^{\perp}$. We also use the notation $\text{Pr}_{V}$ for the orthogonal projection onto a closed subspace $V \subseteq \mathcal{H}$. The following are  well known facts from linear algebra and functional analysis, see \cite[Theorem~5.2 and Remark~5 in Subsection~5.2]{B}.
\begin{lemma}\label{projection fact}
	Let $(\mathcal{H}, \langle \ \cdot, \ \cdot \ \rangle_{\mathcal{H}})$ be a Hilbert space and let $V \subseteq \mathcal{H}$ be a closed subspace. Then 
	\begin{enumerate}
	\item if $W$ is a closed subspace and $V \subseteq W$, then 
	\begin{equation}\label{projection property 1}
	\normalfont ||\text{Pr}_{V}(x)||_{\mathcal{H}} \leq ||\text{Pr}_{W}(x)||_{\mathcal{H}}
	\end{equation}
	 for all $x \in \mathcal{H}$,
		\item it holds that $\normalfont \text{Pr}_{V^\perp}=I-\text{Pr}_{V}$ where $I: \mathcal{H} \to \mathcal{H}$ is an identity map,
		\item for all $x \in \mathcal{H}$,  we have that
		\begin{equation}\label{projection fact equation 2}
			\normalfont ||\text{Pr}_{V^\perp}(x)||_{\mathcal{H}}=||(I-\text{Pr}_{V})(x)||_{\mathcal{H}}=\min_{w \in V}||x-w||_{\mathcal{H}},
		\end{equation}
		\item and if $V$ is a finite dimensional linear subspace with an orthogonal basis $(e_i)_{i=1}^{N}$, then we have that for every $x \in \mathcal{H}$,
	\begin{equation}\label{projection fact equation 1}
		\normalfont \text{Pr}_{V}(x)=\sum_{i=1}^N \frac{\langle x, e_i \rangle_{\mathcal{H}}}{||e_i||_{\mathcal{H}}^2}e_i.
	\end{equation}
	\end{enumerate}
\end{lemma}

Let $[a,b]$ be an interval. In the rest of this section, $L^1(a,b)$ and $L^2(a,b)$ are the $L^1$ and $L^2$ function spaces with respect to the normalized Lebesgue measure. We will use the standard notation $||\cdot||_{L^1(a,b)}$, $||\cdot||_{L^2(a,b)}$ and $\langle \cdot, \cdot \rangle_{L^2(a,b)}$ to denote the $L^1$ and $L^2$ norms and the $L^2$ inner product. 

In the next lemma, for given polynomials $p_1, \cdots, p_r$, we denote by $\tilde{p}_1, \cdots, \tilde{p}_r$ the polynomials obtained by the Gram-Schmidt process in $L^2(a,b)$, i.e., 
	\[
	\tilde{p}_1:=p_1 \ \ \textit{and} \ \ \tilde{p}_i:=p_i- \sum_{k=1}^{i-1} \frac{\langle \tilde{p}_k, p_i \rangle_{L^2(a,b)}}{\langle \tilde{p}_k, \tilde{p}_k \rangle_{L^2(a,b)}} \tilde{p}_k  \ \ \text{for $2 \leq i \leq r$}.
	\]
  We later use the following technical lemma when we construct perturbations in Section~\ref{Sufficiency Gr Section}. 

\begin{lemma}\label{coefficients bounded polynomials}
	Let $r\geq 2$ and $m\geq 1$. There exists $\tilde{C}:=\tilde{C}(r, m)$ such that the following holds: for polynomials  $p_1, \cdots, p_r$ of degree $m$ and $a<b$ satisfying  
	\begin{equation}\label{coefficient bounded assumption 1}
		\int_a^b |p_1| \geq \max_{1 \leq l \leq r} \Big\{ \int_a^b |p_l| \Big\}
	\end{equation}
	and
	\begin{equation}\label{coefficient bounded assumption 2}
		\inf_{c_i \in \mathbb{R}}\int_a^b |p_k-\sum_{i=1}^{k-1}c_i p_i| \geq \max_{k+1\leq l \leq r}\Big\{ \inf_{c_i \in \mathbb{R}} \int_a^b|p_{l}-\sum_{i=1}^{k-1}c_i p_i| \Big\}
	\end{equation}
	for every $2\leq k\leq r-1$, we have
	\begin{equation}\label{coefficients bounded equation}
		\frac{\langle \tilde{p}_i, p_l \rangle_{L^2(a,b)}}{\langle \tilde{p}_i, \tilde{p}_i \rangle_{L^2(a,b)}} \leq \tilde{C}
	\end{equation}
	for every $1\leq i<l\leq r$.
	In particular, there exist constants $\tilde{C}_1:= \tilde{C}_1(m)$ and $\tilde{C}_2:=\tilde{C}_2(r, m)$ such that 
	\begin{equation}\label{inf reduction}
		\inf_{c_i\in \mathbb{R}}\int_a^b |p_l-\sum_{i=1}^{k}c_i p_i| \overset{\tilde{C}_1}{\simeq} \inf_{c_i\in A}\int_a^b|p_l-\sum_{i=1}^{k}c_i p_i|
	\end{equation}
	for every $1\leq k< l \leq r$ where $A:=[-\tilde{C}_2, \tilde{C}_2]$.
\end{lemma}
\begin{proof}
Let $1\leq k< l \leq r$ as in the statement. We first prove \eqref{inf reduction} for $1 \leq k <l \leq r$ under the assumption that \eqref{coefficients bounded equation} holds for all $1 \leq i \leq k$.  By Lemma~\ref{L1 L2 equivalent}, it is enough to establish
\begin{equation}
	\inf_{c_i\in \mathbb{R}}\Big(\frac{1}{b-a}\int_a^b |p_l-\sum_{i=1}^{k}c_i p_i|^2\Big)^{1/2} = \inf_{c_i\in A}\Big( \frac{1}{b-a} \int_a^b|p_l-\sum_{i=1}^{k}c_i p_i|^2\Big)^{1/2}.
\end{equation}
Note that 
\begin{align}
\inf_{c_i\in \mathbb{R}}\Big(\frac{1}{b-a}\int_a^b |p_l-\sum_{i=1}^{k}c_i p_i|^2\Big)^{1/2} \overset{\eqref{projection fact equation 2}}&{=}||\text{Pr}_{\langle p_1, \cdots, p_{k-1} \rangle^{\perp}}(p_l)||_{L^2(a,b)} \notag \\
	\overset{\eqref{projection fact equation 1}}&{=}||p_l-\sum_{i=1}^{k}\frac{\langle p_l, \tilde{p}_i \rangle}{\langle \tilde{p}_i, \tilde{p}_i \rangle}\tilde{p}_i ||_{L^2(a,b)}\notag \\
	&\geq \inf_{c_i \in A}||p_l-\sum_{i=1}^{k}c_i p_i||_{L^2(a, b)}.
\end{align}
where we used the bound \eqref{coefficients bounded equation} repeatedly to obtain a constant $\tilde{C}_2$ for which the last inequality holds. Since the converse inequality is obvious, this completes the proof of \eqref{inf reduction}.

We use mathematical induction to verify \eqref{coefficients bounded equation}. We first prove that \eqref{coefficients bounded equation} holds true for any pair $i=1< l\leq r$.
	By the Cauchy-Schwarz inequality and Lemma~\ref{L1 L2 equivalent}, we have
	\begin{align}
		\Big|\langle p_1, p_l \rangle_{L^2(a,b)}\Big| &\leq  ||p_1||_{L^2(a, b)}||p_l||_{L^2(a,b)} \overset{\eqref{L1 L2 equivalent equation}}{\leq} \tilde{C}||p_1||_{L^2(a, b)}||p_l||_{L^1(a,b)} \notag \\
		\overset{\eqref{coefficient bounded assumption 1}}&{\leq} \tilde{C}||p_1||_{L^2(a, b)}||p_1||_{L^1(a,b)} \overset{\eqref{L1 L2 equivalent equation}}{\leq} \tilde{C}||p_1||_{L^2(a, b)}||p_1||_{L^2(a,b)}
	\end{align}
	Hence we established \eqref{coefficients bounded equation} when $i=1<l\leq r$. 
	
Fix $k \geq 2$ and we now assume that 
\begin{equation}
		\frac{\langle \tilde{p}_i, p_l \rangle_{L^2(a,b)}}{\langle \tilde{p}_i, \tilde{p}_i \rangle_{L^2(a,b)}} \leq \tilde{C}
	\end{equation}
	for any $1\leq i \leq k<l \leq r$. Our goal is to prove that 
	\begin{equation}
		\frac{\langle \tilde{p}_{k+1}, p_l \rangle_{L^2(a,b)}}{\langle \tilde{p}_{k+1}, \tilde{p}_{k+1} \rangle_{L^2(a,b)}} \leq \tilde{C}
	\end{equation}
	for every $k+1<l\leq r$. We first note that from the proof of \eqref{inf reduction} and the assumption of mathematical induction, we know that 
	\begin{equation}\label{coefficient proof 1}
		\inf_{c_i \in A}||p_l - \sum_{i=1}^{k} c_i p_i||_{L^1(a,b)}\overset{\tilde{C}}{\simeq}\inf_{c_i \in \mathbb{R}}||p_l - \sum_{i=1}^{k} c_i p_i||_{L^1(a,b)}
	\end{equation}
	and 
	for any $1\leq k <l \leq r$.
	The fact that  $\tilde{p}_{k+1} \in \langle \tilde{p}_1, \cdots, \tilde{p}_k \rangle^{\perp}$ together with \eqref{coefficient bounded assumption 2} and \eqref{coefficient proof 1} gives us that
	\begin{align}
		||\frac{\langle \tilde{p}_{k+1}, p_l \rangle}{\langle  \tilde{p}_{k+1}, \tilde{p}_{k+1} \rangle} \tilde{p}_{k+1}||_{L^2(a, b)}
		\overset{\eqref{projection fact equation 1}}&{=} ||\text{Pr}_{\langle \tilde{p}_{k+1} \rangle}(p_l)||_{L^2(a,b)} \overset{\eqref{projection property 1}}{\leq} ||\text{Pr}_{\langle \tilde{p}_1, \cdots, \tilde{p}_k \rangle^{\perp}}(p_l)||_{L^2(a,b)} \notag \\
		\overset{\eqref{projection fact equation 2}}&{=} \inf_{c_i \in \mathbb{R}}||p_l - \sum_{i=1}^{k} c_i \tilde{p}_i||_{L^2(a,b)} 
		\leq \inf_{c_i \in A}||p_l - \sum_{i=1}^{k} c_i p_i||_{L^2(a,b)} \notag \\
		\overset{\eqref{L1 L2 equivalent equation}}&{\leq} \tilde{C}\inf_{c_i \in A}||p_l - \sum_{i=1}^{k} c_i p_i||_{L^1(a,b)}   \overset{\eqref{coefficient proof 1}}{\leq}  \tilde{C}\inf_{c_i \in \mathbb{R}}||p_l - \sum_{i=1}^{k} c_i p_i||_{L^1(a,b)} \notag \\
		\overset{\eqref{coefficient bounded assumption 2}}&{\leq} \tilde{C}\inf_{c_i \in \mathbb{R}}||p_{k+1} - \sum_{i=1}^{k} c_i p_i||_{L^1(a,b)} \overset{\eqref{L1 L2 equivalent equation}}{\leq}  \tilde{C} \inf_{c_i \in \mathbb{R}}||p_{k+1} - \sum_{i=1}^{k} c_i p_i||_{L^2(a,b)}\notag \\
		&=\tilde{C} ||\tilde{p}_{k+1}||_{L^2(a,b)}. \notag 
	\end{align}
	 This completes the proof.
\end{proof}

\section{A Natural Generalization and a Counterexample}\label{Counterexample subsection}
\subsection{A Natural Question Motivated by $\mathbb{G}_2$}
In order to establish a $C^m$ Whitney extension theorem for horizontal curves in $\mathbb{G}_r$, we generalize the quantities \eqref{Aab G2} and \eqref{Vab G2} from the setting of $\mathbb{G}_2$. Let $(F_{i})_{1\leq i \leq r}$ and $(F_{ij})_{1\leq j<i\leq r}$ be jets of order $m$ on a compact set $K \subseteq \mathbb{R}$. For every $1 \leq j<i \leq r$ and $a, b \in K$, we define
\begin{align}\label{Aab}
 A_{ij}(a,b)&:=F_{ij}(b)-F_{ij}(a)-\frac{1}{2}\int_{a}^{b}((T_{a}^m F_i)(T_{a}^m F_j)'-(T_{a}^m F_i)'(T_{a}^mF_j))\\ \nonumber 
&\qquad +\frac{1}{2}F_j(a)(F_i(b)-T_{a}^mF_i(b))-\frac{1}{2}F_i(a)(F_j(b)-T_{a}^mF_j(b))
\end{align}
and
\begin{equation}\label{Vab}
V_{ij}(a,b):= (b-a)^{2m} + (b-a)^{m} \int_{a}^{b}|(T_{a}^mF_{i})'|+|(T_{a}^m F_{j})'|.
\end{equation}

This leads to the following natural question. 
\begin{question}\label{a natural question from G3}
 Let $K \subseteq \mathbb{R}$ be a compact set and $(F_{i})_{1\leq i \leq r}$ and $(F_{ij})_{1\leq j<i\leq r}$ be jets of order $m$ on $K$. 
Then is it true that there exists a $C^m$ horizontal curve $\gamma : \mathbb{R} \to \mathbb{G}_r$ such that $D^k \gamma_{i}|_{K}=F^k_{i}$ for every $0 \leq k \leq m$ and $1\leq i \leq r$ and $D^k \gamma_{ij}|_{K}=F^k_{ij}$ for every $0 \leq k \leq m$ and $1 \leq j<i\leq r$ if and only if
\begin{enumerate}
\item $(F_{i})_{i=1}^{r}$ and $(F_{ij})_{1\leq j<i\leq r}$ are Whitney fields of class $C^m$ on $K$,
\item for every $1 \leq k \leq m$, $1\leq j < i \leq r$, and $t\in K$ we have
 \begin{equation}
\label{HorizAssume}
F_{ij}^{k}(t) = \mathcal{P}^k \left(F_{i}^0(t),F_{j}^0(t),F_{i}^1(t),F_{j}^1(t),\dots,F_{i}^{k}(t),F_{j}^{k}(t)\right),
\end{equation}
where $\mathcal{P}^k$ are the polynomials defined in \eqref{def of polynomials},
\item and the component wise $A/V$ condition holds, i.e., 
\begin{equation}\label{component wise A/V}
	A_{ij}(a,b)/V_{ij}(a,b)\to 0
\end{equation}
 uniformly as $(b-a) \searrow 0$ with $a,b\in K$ for all $1\leq j<i \leq r$?

\end{enumerate}
\end{question}

\subsection{A Counterexample in $\mathbb{G}_3$}
The answer to Question~\ref{a natural question from G3} is negative. The following example in $\mathbb{G}_3$ for $m=2$ was suggested by Pinamonti and Speight.
\begin{proposition}\label{Example}
	There exist a compact set $K \subseteq \mathbb{R}$ and jets $(F_{i})_{1\leq i \leq 3}$ and $(F_{ij})_{1\leq j<i\leq 3}$ of order $2$ on $K$ such that 
	\begin{enumerate}
\item $(F_{i})_{1\leq i \leq 3}$ and $(F_{ij})_{1\leq j<i\leq 3}$ are Whitney fields of class $C^2$ on $K$,
\item for every $1 \leq k \leq 2$, $1 \leq j<i \leq 3$ and $t\in K$ we have
 \begin{equation}
\label{HorizAssume}
F_{ij}^{k}(t) = \mathcal{P}^k \left(F_{i}^0(t),F_{j}^0(t),F_{i}^1(t),F_{j}^1(t),\dots,F_{i}^{k}(t),F_{j}^{k}(t)\right)
\end{equation}
where $\mathcal{P}^k$ are the polynomials defined in \eqref{def of polynomials},
\item and for all $1\leq j<i \leq 3$, $A_{ij}(a,b)/V_{ij}(a,b)\to 0$ uniformly as $(b-a) \searrow 0$ with $a,b\in K$,
\end{enumerate}
but there is no $C^2$ horizontal curve that extends $(F_{i})_{1\leq i \leq 3}$ and $(F_{ij})_{1\leq j<i\leq 3}$.
\end{proposition}

\begin{proof}
	 Set
	\[
	\lambda_n=10^{-n} \ \ \text{and} \ \ K:=\cup_{n=1}^{\infty}[c_n, d_n]\cup\{1\}
	\] 
	 where $c_n:=1-2^{-n}$ and $d_n:=1-\frac{3}{4}2^{-n}$. We first define jets of order $2$ on $K$ by
	 \begin{equation}
	\begin{pmatrix}F^0_1(t) \\ F^1_1(t) \\ F^2_1(t) \end{pmatrix}  := \begin{pmatrix} t \\ 1 \\ 0 \end{pmatrix}
    , \ \ \ 
     \begin{pmatrix}F^0_2(t) \\ F^1_2(t) \\ F^2_2(t) \end{pmatrix}  := \begin{pmatrix} 0 \\ 0 \\ 0 \end{pmatrix}
    , \ \ \ 
    \begin{pmatrix}F^0_3(t) \\ F^1_3(t) \\ F^2_3(t) \end{pmatrix}  := \begin{pmatrix} -t \\ -1 \\ 0 \end{pmatrix}
     \notag 
\end{equation}
and 
	 \begin{equation}
	 \begin{pmatrix}F^0_{21}(t) \\ F^1_{21}(t) \\ F^2_{21}(t) \end{pmatrix}  := \begin{pmatrix} \sum_{n=1}^{\infty}\lambda_n \mathds{1}_{[c_n, d_n]}(t) \\ 0 \\ 0 \end{pmatrix}
    , \ \ \ 
      \begin{pmatrix}F^0_{31}(t) \\ F^1_{31}(t) \\ F^2_{31}(t) \end{pmatrix}  := \begin{pmatrix} 0 \\ 0 \\ 0 \end{pmatrix}
     ,\ \ \ 
     \begin{pmatrix}F^0_{32}(t) \\ F^1_{32}(t) \\ F^2_{32}(t) \end{pmatrix}  := \begin{pmatrix} 0 \\ 0 \\ 0 \end{pmatrix}.
     \notag
\end{equation}
 where $\mathds{1}_{[c_n, d_n]}$ is the characteristic function of $[c_n, d_n]$.
We only check that $F_{21}$ is a Whitney field since the others are obvious. If $a \in [c_n, d_n]$ and $b \in [c_m, d_m]$ with $n<m$, we have
\begin{align}\label{condition on lambda 1}
	\frac{|F_{21}(b)-F_{21}(a)|}{(b-a)^2} &\leq \frac{|1/10^{m}-1/10^{n}|}{(d_n-c_m)^2} \notag \\
	&=16\frac{|10^m-10^n|}{(3(2^m)-2^n)^2}\frac{4^{n+m}}{10^{n+m}} \notag \\
	&=16\frac{1-10^{n-m}}{(3-2^{n-m})^2}\Big(\frac{4}{10}\Big)^{n} \to 0,
\end{align}
as $m,n \to \infty$.

Next we check the $A/V$ condition for $F_{21}$. 
Let $a \in [c_n, d_n]$ and $b \in [c_m, d_m]$ with $n<m$. It is not hard to check that
\[
A_{21}(a, b)=F_{21}(b)-F_{21}(a) \ \ \text{and} \ \ V_{21}(a, b)=(b-a)^{4}+(b-a)^3.
\]
By the similar computation as in \eqref{condition on lambda 1}, we have
\begin{equation}
	A_{21}(a, b)/V_{21}(a, b) \leq 16\frac{1-10^{n-m}}{(3-2^{n-m})^3}\Big(\frac{8}{10}\Big)^{n} \to 0.
\end{equation}
as $n, m \to \infty$. Other components satisfy the $A/V$ condition. Therefore, the jets defined above satisfies all the conditions.

We now suppose that there exists a $C^2$ horizontal curve $\gamma=(\gamma_1, \gamma_2, \gamma_3, \gamma_{21}, \gamma_{31}, \gamma_{32})$ which extends the above jets. For any $\epsilon>0$, there exists $\delta>0$ such that if $|s-1|<\delta$, then the Taylor polynomials of $\gamma_i$ and $\gamma_i'$ at $1$ gives 
\begin{equation}\label{counterexample estimate 1}
|\gamma_i(s)-(\gamma_i(1)+\gamma_i'(1)(s-1))| \leq \epsilon|s-1|^2 \ \text{and} \ |\gamma_i'(s)-\gamma_i'(1)| \leq \epsilon|s-1|
\end{equation}
for $i=1, 3$ and 
\begin{equation}\label{counterexample estimate 2}
	|\gamma_2(s)| \leq \epsilon |s-1|^2 \ \text{and} \ |\gamma_2'(s)|\leq \epsilon |s-1|.
\end{equation}
 If $a \in [c_n, d_n]$ and $1-\delta<a$, then 
\begin{align}
	|\lambda_n| &= |\gamma_{21}(a)-\gamma_{32}(a)| \notag \\
	&= \Big| \Big(\gamma_{21}(a)-\gamma_{21}(1)\Big)-\Big(\gamma_{32}(a) - \gamma_{32}(1)\Big)\Big| \notag \\
	&=\frac{1}{2}\Big|\int_a^{1}(\gamma_1\gamma_2'-\gamma_2\gamma_1')-(\gamma_2\gamma_3'-\gamma_3\gamma_2')\, ds \Big| \notag \\
	&= \frac{1}{2}\Big|\int_a^{1} \gamma_2(-\gamma_1'-\gamma_3')+\gamma_2'(\gamma_1+\gamma_3)\Big| \notag \\
	&\leq  \frac{1}{2}\int_a^{1} |\gamma_2(-\gamma_1'-1+1-\gamma_3')+\gamma_2'(\gamma_1-(1+s-1)+(1+s-1)+\gamma_3)| \notag \\
	\overset{\eqref{counterexample estimate 1}, \eqref{counterexample estimate 2}}&{\leq} \frac{1}{2}\epsilon^2 (1-a)^4 \leq  \frac{1}{2}\epsilon^2 16^{-n}.
\end{align}
This leads to a contradiction for large enough $n \in \mathbb{N}$ since $|\lambda_n|=10^{-n}$. Therefore, the jets in this example cannot have a $C^2$ horizontal extension.\end{proof}
\begin{remark}
	The jets considered above have a $C^1$ horizontal extension. The existence of a $C^1$ horizontal extension follows by checking the condition in  \cite[Definition~2.3]{JS} in $\mathbb{G}_3$.
\end{remark}

\section{Necessity: The Case $\mathbb{G}_3$}\label{Necessity G3 Section}
\subsection{The Generalized $A/V$ Condition}

From  Proposition \ref{Example}, we observed that the component-wise $A/V$ condition generalized from \cite{PSZ} is not enough to prove a $C^{m}$ Whitney extension result for horizontal curves in $\mathbb{G}_r$. We first restrict our attention to $\mathbb{G}_3$. Here we introduce a generalized $A/V$ condition. 
\begin{definition}\label{Definition of generalized A/V}
	Let $(F_i)_{1 \leq i \leq 3}$ and $(F_{ij})_{1\leq j<i\leq 3}$ be jets of order $m$ on $K$. We say that $(F_i)_{1 \leq i \leq 3}$ and $(F_{ij})_{1\leq j<i\leq 3}$ satisfy the \emph{generalized $A/V$ condition} if for any compact set $A \subseteq \mathbb{R}$, the following hold:
	\begin{equation}\label{key11}
		\limsup_{\substack{ |a-b|\to 0 \\ a, b \in K}}\sup_{c_2, c_3 \in A}\frac{|c_{3}A_{21}(a,b)+A_{32}(a,b)-c_2A_{31}(a,b)|}{(b-a)^{2m}+(b-a)^m\int_a^b|T_aF_2'-c_2T_aF_1'|+|T_aF_3'-c_3T_aF_1'|}=0, 
	\end{equation}
	\begin{equation}\label{key22}
		\limsup_{\substack{ |a-b|\to 0 \\ a, b \in K}}\sup_{c_1, c_3 \in A}\frac{|-c_{3}A_{21}(a,b)-c_1A_{32}(a,b)+A_{31}(a,b)|}{(b-a)^{2m}+(b-a)^m\int_a^b|T_aF_1'-c_2T_aF_2'|+|T_aF_3'-c_3T_aF_2'|}=0 ,
	\end{equation}
	and
	\begin{equation}\label{key33}
		\limsup_{\substack{ |a-b|\to 0 \\ a, b \in K}}\sup_{c_1, c_2 \in A}\frac{|A_{21}(a,b)+c_1A_{32}(a,b)-c_2A_{31}(a,b)|}{(b-a)^{2m}+(b-a)^m\int_a^b|T_aF_1'-c_1T_aF_3'|+|T_aF_2'-c_2T_aF_3'|} =0.
	\end{equation}
	\end{definition}
	
	\begin{remark}
	Note that \eqref{key11}, \eqref{key22}, \eqref{key33} recovers the component wise $A/V$ condition \eqref{component wise A/V} by plugging $c_i=0$ for $i=1, 2, 3$.
\end{remark}

	\subsection{Our Main Theorem in $\mathbb{G}_3$}
Our main result in the case of $\mathbb{G}_3$ is the following.
\begin{theorem}\label{Cm Whitney G3}
	Let $K \subseteq \mathbb{R}$ be a compact set and $(F_{i})_{1\leq i \leq 3}$ and $(F_{ij})_{1\leq j<i\leq 3}$ be jets of order $m$ on $K$. 
Then there exists a $C^m$ horizontal curve $\gamma : \mathbb{R} \to \mathbb{G}_3$ such that for each $0\leq k \leq m$, $D^k \gamma_{i}|_{K}=F^k_{i}$ for all $1 \leq i \leq 3$ and $D^k \gamma_{ij}|_{K}=F^k_{ij}$ for all $1 \leq j <i\leq 3$ if and only if
\begin{enumerate}
\item $(F_{i})_{1 \leq i \leq 3}$ and $(F_{ij})_{1\leq j<i\leq 3}$ are Whitney fields of class $C^m$ on $K$,
\item for every $1 \leq k \leq m$ and $t\in K$ and all $1 \leq j <i\leq 3$ we have
 \begin{equation}
\label{HorizAssume}
F_{ij}^{k}(t) = \mathcal{P}^k \left(F_{i}^0(t),F_{j}^0(t),F_{i}^1(t),F_{j}^1(t),\dots,F_{i}^{k}(t),F_{j}^{k}(t)\right)
\end{equation}
where $\mathcal{P}^k$ are the polynomials defined in \eqref{def of polynomials},
\item  and the jets $(F_{i})_{1 \leq i \leq 3}$ and $(F_{ij})_{1\leq j<i\leq 3}$ satisfy the generalized $A/V$ condition.
\end{enumerate}
\end{theorem}
 We remark that for any $C^m$ horizontal curve $\gamma: \mathbb{R}\to \mathbb{G}_3$ and any compact set $K \subseteq \mathbb{R}$, we observed that the jets obtained by $F_i=(D^k \gamma_i|_{K})_{k=0, \cdots, m}$ for $1 \leq i \leq 3$ and $F_{ij}=(D^k \gamma_{ij}|_{K})_{k=0, \cdots, m}$ for $1 \leq j<i\leq 3$ satisfy the first two properties, see Lemma~\ref{horizontal polynomial}. In the next subsection we will check that every $C^m$ horizontal curve in $\mathbb{G}_3$ satisfies the generalized $A/V$ condition, see Proposition~\ref{New Condition}.

	\subsection{$C^m$ Horizontal Curves in $\mathbb{G}_3$ Satisfy The Generalized $A/V$ Condition} 
In this subsection we observe that jets obtained by restricting any $C^m$ horizontal curve satisfy the generalized $A/V$ condition.
\begin{proposition}\label{New Condition}
	Let $\gamma:\mathbb{R} \to \mathbb{G}_3$ be a $C^m$ horizontal curve in $\mathbb{G}_3$ and $K \subseteq \mathbb{R}$ be a compact set. Then the jets obtained by $F_i=(D^k \gamma_i|_{K})_{k=0, \cdots, m}$ and $F_{ij}=(D^k \gamma_{ij}|_{K})_{k=0, \cdots, m}$ satisfy the generalized $A/V$ condition.
\end{proposition}
\begin{proof}
	We only prove the convergence \eqref{key11}. One can derive \eqref{key22} and \eqref{key33} in a similar manner.  We may assume that $K$ is a closed bounded interval, otherwise replace $K$ with $[\min K, \max K]$.  For a jet $F$ on $K$ and $a \in K$, we simply denote the formal $m$-th order Taylor polynomial by $T_aF$ instead of $T_a^m F$. Fix a compact subset $A \subseteq \mathbb{R}$. We examine the following two cases.

\textbf{Case 1}: We first assume that $\gamma(a)=0$.

	Fix $0<\epsilon<1$. There exists $\delta>0$ such that if $[a,b] \subset K$ and $(b-a)<\delta$ then:
\begin{enumerate}[(i)]
\item $|D^k\gamma_i-D^k\gamma_i(a)|\leq \epsilon$ on $[a,b]$ for $1\leq i \leq 3$ and $0\leq k\leq m$.
\item $|\gamma_i-T_a \gamma_i|\leq \varepsilon (b-a)^m$ on $[a,b]$ for $1\leq i \leq 3$.
\item $|\gamma_i'-(T_a \gamma_i)'|\leq \varepsilon (b-a)^{m-1}$ on $[a,b]$ for $1\leq i \leq 3$.
\end{enumerate}
 For any $c_2, c_3 \in A$, we write
	\[
	 \gamma_2=c_2\gamma_1+\epsilon_2 \ \ \ \text{and} \ \ \ \gamma_3=c_3\gamma_1+\epsilon_3
	\]
	where $\epsilon_2:=\gamma_2-c_2\gamma_1$ and $\epsilon_3:=\gamma_3-c_3\gamma_1$.
	Then note that 
	\begin{align}\label{linear dependence1}
		(\gamma_2\gamma_1'-\gamma_2'\gamma_1)&=(c_2\gamma_1+\epsilon_2)\gamma_1'-(c_2\gamma_1'+\epsilon_2')\gamma_1 \notag \\
		&=\epsilon_2\gamma_1'-\epsilon_2' \gamma_1,
	\end{align}
	\begin{align}\label{linear dependence2}
		(\gamma_3\gamma_2'-\gamma_3'\gamma_2)&=(c_3\gamma_1+\epsilon_3)(c_2\gamma_1'+\epsilon_2')-(c_3\gamma_1'+\epsilon_3')(c_2\gamma_1+\epsilon_2) \notag \\
		&=(c_3\gamma_1\epsilon_2'+c_2\gamma_1'\epsilon_3+\epsilon_2'\epsilon_3)-(c_3\gamma_1'\epsilon_2+c_2\gamma_1\epsilon_3'+\epsilon_2\epsilon_3'),
	\end{align}
	and
	\begin{align}\label{linear dependence3}
		(\gamma_3\gamma_1'-\gamma_3'\gamma_1)&=(c_3\gamma_1+\epsilon_3)\gamma_1'-(c_3\gamma_1'+\epsilon_3')\gamma_1 \notag \\
		&=\epsilon_3\gamma_1'-\epsilon_3'\gamma_1.
	\end{align}
	Then combining \eqref{linear dependence1}, \eqref{linear dependence2}, and \eqref{linear dependence3}, we obtain
	\[
	c_3(\gamma_2\gamma_1'-\gamma_2'\gamma_1)+(\gamma_3\gamma_2'-\gamma_3'\gamma_2)-c_2(\gamma_3\gamma_1'-\gamma_3'\gamma_1)=\epsilon_2'\epsilon_3-\epsilon_2\epsilon_3'.
	\]
	One can do the above argument in the same way for Taylor polynomials $T_a \gamma_i$ and obtain
	\begin{align}
		c_3\Big(T_a \gamma_2 (T_a \gamma_1)'-(T_a \gamma_2)'T_a \gamma_1\Big)&+\Big(T_a \gamma_3 (T_a \gamma_2)'-(T_a \gamma_3)'T_a \gamma_2\Big)-c_2\Big(T_a \gamma_3(T_a \gamma_1)'-(T_a \gamma_3)'T_a \gamma_1\Big)\notag \\
		&=(T_a \epsilon_2)'T_a \epsilon_3-T_a \epsilon_2 (T_a \epsilon_3)'.\notag 
	\end{align}
Therefore, from the above two equalities and \eqref{Aab} and Lemma~\ref{ij-th component representation}, we have
\begin{align}
	|A_{32}(a,b)+c_3A_{21}(a,b)-c_2A_{31}(a,b)|=|\frac{1}{2}\int_a^b(\epsilon_2'\epsilon_3-\epsilon_2\epsilon_3')-((T_a\epsilon_2)'T_a\epsilon_3-T_a\epsilon_2(T_a\epsilon_3)')|.\notag 
\end{align}
By the standard argument as in \cite[Proof of Proposition~5.2]{PSZ}, we have
\begin{align}
	|\int_a^b(\epsilon_2'\epsilon_3 &-\epsilon_2\epsilon_3')-((T_a\epsilon_2)'T_a\epsilon_3-T_a\epsilon_2(T_a\epsilon_3)')| \notag \\
	&\leq C(1+|c_2|)(1+|c_3|)\epsilon^2(b-a)^{2m}+C(2+|c_2|+|c_3|)\epsilon (b-a)^m \int_a^b|(T_a\epsilon_2)'|+|(T_a\epsilon_3)'|.\notag
\end{align}
where $C$ is a constant depending only on $m$. Since $\epsilon>0$ is arbitrary and $A$ is compact, this completes the proof of \textbf{Case 1}. 

\textbf{Case 2}: In the general case without assuming $\gamma_i(a)=0$ for all $1 \leq i \leq 3$ and $\gamma_{ij}(a)=0$ for all $1\leq j<i \leq 3$.

 We first set
\begin{equation}
	\tilde{\gamma}:=\gamma(a)^{-1}\cdot \gamma.
\end{equation}
Then one can easily verify from the group law of $\mathbb{G}_3$ in Definition~\ref{freegroupstep2} that 
\begin{equation}
	\tilde{\gamma}_k=\gamma_k-\gamma_k(a) \ \ \text{and} \ \ \tilde{\gamma}_{ij}=\gamma_{ij}-\gamma_{ij}(a)+\frac{1}{2}(\gamma_j(a)\gamma_i-\gamma_i(a)\gamma_j)
\end{equation}
for all $1 \leq k \leq 3$ and all $1\leq j<i\leq 3$. It is immediate that the curve $\tilde{\gamma}$ is a $C^m$  horizontal curve with $\tilde{\gamma}(a)=0$. Also, we set 
\begin{align}
 \tilde{A}_{ij}(a,b)&:=\tilde{\gamma}_{ij}(b)-\tilde{\gamma}_{ij}(a)-\frac{1}{2}\int_{a}^{b}((T_{a} \tilde{\gamma}_i)(T_{a} \tilde{\gamma}_j)'-(T_{a} \tilde{\gamma}_i)'(T_{a}\tilde{\gamma}_j)) \notag 
\end{align}
and 
\begin{equation}
	\tilde{\epsilon}_2:=\tilde{\gamma}_2-c_2\tilde{\gamma}_1 \ \ \text{and} \ \ \tilde{\epsilon}_3:=\tilde{\gamma}_3-c_2\tilde{\gamma}_1
\end{equation}
for any $c_2, c_3 \in A$. A direct calculation yields
\begin{equation}
	\tilde{A}_{ij}(a,b)=A_{ij}(a,b) \ \ \text{and} \ \ \tilde{\epsilon}_k'=\epsilon_k'
\end{equation}
for all $1\leq j <i \leq 3$ and $k=2,3$. We refer the readers to \cite[Lemma~3.3]{Z2} for the arguments in the Heisenberg group that can be simply generalized to the setting of $\mathbb{G}_3$.
Then by the arguments in \textbf{Case 1}, for any $\epsilon>0$, there exists $\delta>0$ such that whenever $|b-a|<\delta$, we obtain
\begin{align}
	|A_{32}(a,b) &+c_3A_{21}(a,b)-c_2A_{31}(a,b)|=|\tilde{A}_{32}(a,b) +c_3\tilde{A}_{21}(a,b)-c_2\tilde{A}_{31}(a,b)| \notag \\
	\overset{\textbf{Case 1}}&{\leq} C\epsilon^2(b-a)^{2m}+C\epsilon (b-a)^m \int_a^b|(T_a\tilde{\epsilon}_2)'|+|(T_a\tilde{\epsilon}_3)'| \notag \\
	&\leq C\epsilon^2(b-a)^{2m}+C\epsilon (b-a)^m \int_a^b|(T_a \epsilon_2)'|+|(T_a \epsilon_3)'|. \notag 
\end{align}
where $C$ is a constant depending only on $m$ and $A$. This completes the proof.
\end{proof}


\section{Sufficiency: The case $\mathbb{G}_3$}\label{Sufficiency G3 Section}
In the rest of the paper, we use $a_1\vee \cdots\vee a_n$ (resp. $a_1\wedge \cdots\wedge a_n$) to denote the maximum (resp. the minimum) of real numbers $a_1, \cdots, a_n$.

\begin{proposition}\label{Our Goal}
	Let $K \subseteq \mathbb{R}$ be compact and let $(F_i)_{1 \leq i \leq 3}, (F_{ij})_{1\leq j<i \leq 3}$ be jets of order $m$ on $K$. Assume that
	\begin{enumerate}
		\item $(F_i)_{1 \leq i \leq 3}, (F_{ij})_{1\leq j<i \leq 3}$ are Whitney fields of class $C^m$,
		\item for every $1 \leq k \leq m$ and $t\in K$ and all $1 \leq j <i\leq 3$ we have
 \begin{equation}
F_{ij}^{k}(t) = \mathcal{P}^k \left(F_{i}^0(t),F_{j}^0(t),F_{i}^1(t),F_{j}^1(t),\dots,F_{i}^{k}(t),F_{j}^{k}(t)\right),
\end{equation}
where $\mathcal{P}^k$ are the polynomials defined in \eqref{def of polynomials},
\item and the generalized $A/V$ condition holds.
\end{enumerate}
	Then there exists a $C^m$ horizontal curve $\gamma$ such that for each $0 \leq k \leq m$, $D^k\gamma_i|_{K}=F_i^k$ for all $1\leq i \leq 3$ and $D^k\gamma_{ij}|_{K}=F_{ij}^k$ for all $1 \leq j <i \leq 3$.
\end{proposition}
\begin{proof}
	
We begin the proof of Proposition~\ref{Our Goal} by following the argument of \cite[Proposition~6.2]{PSZ}. Let $I:=[\min K, \max K]$ and $I\setminus K := \cup_{i=1}^{\infty} (a_i, b_i)$ where $(a_i, b_i)$ are disjoint open intervals. It is enough to construct a $C^m$ horizontal extension on $I$.  Applying the classical Whitney extension theorem (Theorem~\ref{Classical whitney extension}) to $F_1, F_2, F_3$, we obtain $f_1, f_2, f_3 \in C^m(\mathbb{R})$ such that $D^kf_i|_{K}=F_i^k$ for all $1\leq i\leq 3$ and $0\leq k \leq m$. Take a modulus of continuity $\alpha$ so that for all $a, x \in K$, $0 \leq k \leq m$ and $1\leq i \leq 3$, we have
\begin{equation}\label{good modulus G3 1}
	| F^k_i(x)-D^k T_a F_i(x)| \leq \alpha(|x-a|)|x-a|^{m-k},
\end{equation}
\begin{equation}\label{good modulus G3 2}
	|D^k f_i(x)-D^k T_a F_i(x)|\leq \alpha(|x-a|)|x-a|^{m-k},
\end{equation}
and
\begin{equation}\label{good modulus G3 3}
	|D^k f_i(x)-D^kf_i(a)| \leq \alpha(x-a).
\end{equation}
Moreover, by \eqref{key11}, \eqref{key22}, and \eqref{key33}, we further assume that for a fixed compact set $A \subseteq \mathbb{R}$, 
\begin{equation}\label{to fix 32 in G3}
	\sup_{c_2, c_3 \in A}\frac{|A_{32}(a,b)+c_{3}A_{21}(a,b)-c_2A_{31}(a,b)|}{(b-a)^{2m}+(b-a)^m\int_a^b|T_aF_2'-c_2T_aF_1'|+|T_aF_3'-c_3T_aF_1'|} \leq \alpha(b-a),
\end{equation} 
\begin{equation}\label{to fix 31 in G3}
	\sup_{c_1, c_3 \in A}\frac{|A_{31}(a,b)-c_{3}A_{21}(a,b)-c_1A_{32}(a,b)|}{(b-a)^{2m}+(b-a)^m\int_a^b|T_aF_1'-c_2T_aF_2'|+|T_aF_3'-c_3T_aF_2'|} \leq \alpha(b-a),
\end{equation}
and
\begin{equation}\label{to fix 21 in G3}
	\sup_{c_1, c_2 \in A}\frac{|A_{21}(a,b)+c_1A_{32}(a,b)-c_2A_{31}(a,b)|}{(b-a)^{2m}+(b-a)^m\int_a^b|T_aF_1'-c_1T_aF_3'|+|T_aF_2'-c_2T_aF_3'|} \leq \alpha(b-a).
\end{equation}
where $\alpha$ is made larger and depends on $A$.  In the rest of this section, we use $C$ to denote a constant which depends on $K$, $A$ and $m$. The constant $\tilde{C}$ is reserved to denote a constant which only depends on $K$ and $m$. These constants vary from line to line.  The compact set $A$ chosen above is an interval $[-\tilde{C}_1, \tilde{C}_1]$ where the constant $\tilde{C}_1$ is the largest possible $\tilde{C}$ in this proof. 
The main step of this section is to prove the following.

\begin{proposition}\label{New Perturbation}
	There exists a modulus of continuity $\beta$ for which the following holds:  for each interval $[a_l, b_l]$, there exist $C^{\infty}$ perturbations $\phi_1, \phi_2, \phi_3$ on $[a_l, b_l]$ such that 
	\begin{enumerate}
		\item $D^k \phi_i(a_l)=D^k \phi_i(b_l)=0 $ for all $0\leq k \leq m$,
		\item $\max_{[a_l, b_l]}\{ |D^k \phi_1|, |D^k \phi_2|, |D^k \phi_3|\} \leq \beta(b_l-a_l)$ for all $0\leq k \leq m$,
		\item $F_{ij}(b_l)-F_{ij}(a_l)=1/2\int_{a_l}^{b_l}(f_i+\phi_i)(f_j'+\phi_j')-(f_i'+\phi_i')(f_j+\phi_j)$ for all $1 \leq j<i \leq 3$.
	\end{enumerate}
\end{proposition}
\begin{proof}
	
 Since we fix the interval $[a_l, b_l]$ and construct perturbations, we simply denote $[a_l, b_l], A_{ij}(a_l,b_l),$ $\alpha(b_l-a_l)$ and $T_{a_l} F_i$ by $[a, b], A_{ij}, \alpha$ and $TF_i$, respectively. By the same reason as in \cite[Remark~6.4]{PSZ}, we always assume that $\alpha < 1/(2C_1)\wedge 1/(2\tilde{C}_1)$ by removing finitely many intervals $[a_l, b_l]$ where $C_1$ is the largest possible constant $C$ that appears in the present proof. For  the intervals with $\alpha \geq 1/(2C_1)\wedge 1/(2\tilde{C}_1)$, one can make perturbations based on the proof of Proposition~\ref{case5} so we omit the proof of this case. As in \cite[Claim~6.3]{PSZ}, we may assume that $F_i(a)=0$ for all $1 \leq i \leq 3$ and $F_{ij}(a)=0$ for all $1\leq j<i \leq 3$.

We first assume that the $|TF_1'|$ is the largest integral, i.e., 
\begin{equation}\label{F1 is the largest G3}
	\int_{a}^{b} |TF_1'| \geq \int_{a}^{b} |TF_2'|\vee \int_{a}^{b} |TF_3'|.
\end{equation}

For $1 \leq j<i\leq 3$, let
\[
\mathcal{A}_{ij}:=F_{ij}(b)-F_{ij}(a)-\frac{1}{2}\int_a^bf_if_j'-f_i'f_j.
\]
The argument similar to the proof of Proposition~\ref{New Condition} gives us the estimate
\begin{align}\label{estimate after extension G3}
|\mathcal{A}_{32}&+c_3\mathcal{A}_{21}-c_2\mathcal{A}_{31}|\notag \\
&\leq |(\mathcal{A}_{32}+c_3\mathcal{A}_{21}-c_2\mathcal{A}_{31})-(A_{32}+c_3 A_{21}-c_2 A_{31})|+|A_{32}+c_3 A_{21}-c_2 A_{31}|\notag \\
&\leq C \alpha(b-a)^{2m}+ C \alpha (b-a)^m\int_a^b \Big( |T_aF_2'-c_2T_aF_1'|+|T_aF_3'-c_3T_aF_1'|\Big)
\end{align}
for any $c_2, c_3 \in A$. We now divide the construction into several cases.

\begin{proposition}\label{Case1}
Suppose that
	\begin{equation}\label{F1 bigger than interval G3}
		\int_a^b |TF_1'|\geq (b-a)^m 
		\end{equation}
		and
		\begin{equation}\label{case1 assumption 2}
		\Big(\inf_{c_2 \in \mathbb{R}}\int_a^b |TF_2'-c_2TF_1'|\Big)\vee \Big(\inf_{c_3 \in \mathbb{R}}\int_a^b |TF_3'-c_3TF_1'|\Big) < (b-a)^m.
			\end{equation}
			 Then there exist $C^{\infty}$ perturbations $(\phi_i)_{i=1,2,3}$ on $[a, b]$ satisfying the properties in Proposition~\ref{New Perturbation}.
\end{proposition}
\begin{proof}
	 Take $c_2, c_3 \in \mathbb{R}$ satisfying 
	 \begin{equation}
	 	\Big( \int_a^b |TF_2'-c_2TF_1'| \Big) \vee \Big( \int_a^b |TF_3'-c_3TF_1'| \Big) < (b-a)^m.
	 \end{equation}
	 Then we have
	\[
	(|c_2|\vee |c_3|) \int_a^b|TF_1'| \leq \int_a^b(|TF_2'|\vee|TF_3'|) +(b-a)^m,
	\]
	which implies that $|c_2|\vee |c_3| \leq 2$ by \eqref{F1 is the largest G3} and \eqref{F1 bigger than interval G3}. Hence we may assume that $c_2, c_3 \in A$. Therefore, \eqref{estimate after extension G3} with this choice of $c_2$ and $c_3$ becomes
\begin{equation}\label{case1 remained area small}
|\mathcal{A}_{32}+c_3\mathcal{A}_{21}-c_2\mathcal{A}_{31}| \overset{\eqref{estimate after extension G3}, \eqref{case1 assumption 2}}{\leq} C\alpha (b-a)^{2m}.	
\end{equation}
By Lemma~\ref{nice subinterval}, one can take a closed subinterval $I_1 \subseteq [a, b]$  of length at least $(b-a)/4m^2$ such that $|TF_1'|\geq M_1/2$ where $M_1=\max_{[a, b]}|TF_1'|$. We first set $\phi_i\equiv 0$ on $[a, b] \setminus I_1$ for all $1\leq i \leq 3$. We split $I_1$ into three subintervals, $I_1^{21}$, $I_1^{31}$ and $I_1^{32}$. Note that by pugging in $c_1=c_2=0$ in \eqref{to fix 21 in G3}, we have 
\begin{equation}\label{wanna use Heisenberg}
A_{21}/V_{21} \leq \alpha.	
\end{equation}
 The above inequality \eqref{wanna use Heisenberg} together with the assumptions \eqref{F1 is the largest G3} and \eqref{F1 bigger than interval G3} allows us correct $F_{21}$ using the technique in the Heisenberg group, see \cite[Lemma~6.5]{PSZ}. Namely, one can construct $\phi_2$ on $I_1^{21}$ (and $\phi_1\equiv\phi_3\equiv 0$ on $I_1^{21}$) such that
	\begin{enumerate} 
		\item $D^k\phi_2|_{\partial I_1^{21}}=0$ for all $0\leq k \leq m$ where $\partial I_1^{21}$ is the set of endpoints of $I_1^{21}$,
		\item $|D^k \phi_2| \leq \tilde{C} \alpha$ on $[a, b]$ for all $0\leq k \leq m$,
		\item $\mathcal{A}_{21}=\int_{I_1^{21}} \phi_2 f_1'$.
	\end{enumerate}
	
	Similarly, by pugging $c_1=c_3=0$ in \eqref{to fix 31 in G3}, we have 
\begin{equation}
A_{31}/V_{31} \leq \alpha.	
\end{equation}
Hence we can construct $\phi_3$ as above on $I_1^{31}$ (and $\phi_1\equiv\phi_2\equiv 0$ on $I_1^{31}$) such that
	\begin{enumerate} 
		\item $D^k\phi_3|_{\partial I_1^{31}}=0$ for all $0\leq k \leq m$ where $\partial I_1^{31}$ is the set of endpoints of $I_1^{31}$,
		\item $|D^k \phi_3| \leq \tilde{C} \alpha$ on $[a, b]$ for all $0\leq k \leq m$,
		\item $\mathcal{A}_{31}=\int_{I_1^{31}} \phi_3 f_1'$.
	\end{enumerate}
		We finally correct $F_{32}$, i.e., we need to construct $\phi_2$ and $\phi_3$ on $I_1^{32}$ such that
		\begin{equation}\label{32 component perturbation 1}
			\mathcal{A}_{32}=\int_a^b \phi_3 f_2'+f_3\phi_2'+\phi_3\phi_2'.
		\end{equation}
		 Note that due to $\phi_2$ on $I_1^{21}$ and $\phi_3$ on $I_1^{31}$ that are already constructed, for any $c_2, c_3 \in A$ we have
		 \begin{align}\label{case1 dep1}
		 	|c_3\mathcal{A}_{21}-\int_{I_1^{21}}\phi_2 f_3'|&=|\int_{I_1^{21}}c_3\phi_2 f_1'-\phi_2 f_3'| \notag \\
		 	&\leq |\int_{I_1^{21}}c_3\phi_2 f_1'- c_3\phi_2 TF_1'|+|c_3\phi_2 TF_1'-\phi_2 TF_3'|+|\phi_2 TF_3'-\phi_2 f_3'| \notag \\
		 	\overset{\eqref{good modulus G3 2}, \eqref{case1 assumption 2}}&{\leq} C\alpha^2(b-a)^{2m}+\tilde{C}\alpha (b-a)^{2m}+\tilde{C} \alpha (b-a)^{2m} \notag \\
		 	&\leq C \alpha (b-a)^{2m}
		 \end{align}
		 and similarly,
		 \begin{equation}\label{case1 dep2}
		 	|c_2\mathcal{A}_{31}-\int_{I_1^{31}}\phi_3 f_2'|=|\int_{I_1^{31}}c_2\phi_3 f_1'-\phi_3 f_2'|\leq C\alpha(b-a)^{2m}.
		 \end{equation}

The above estimates imply that for any $c_2, c_3 \in A$, 
\begin{align}\label{case1 area remained}
	|\mathcal{A}_{32}&-\int_{I_1^{31}}\phi_3 f_2'-\int_{I_1^{21}}f_3\phi_2'| \notag \\
	&\leq|A_{32}+c_3\mathcal{A}_{21}-c_2\mathcal{A}_{31}|+|\int_{I_1^{21}}f_3'\phi_2-c_3\mathcal{A}_{21}|+|c_2\mathcal{A}_{31}-\int_{I_1^{31}}\phi_3 f_2'| \notag \\
	\overset{\eqref{case1 remained area small}, \eqref{case1 dep1}, \eqref{case1 dep2}}&{\leq}  C\alpha (b-a)^{2m}.
\end{align}
Taking into account \eqref{case1 area remained} and the fact that we should not change the fixed components, we will construct $\phi_2$ and $\phi_3$ on $I_1^{32}$ such that
\begin{equation}\label{32 component perturbation 2}
	\phi_2 \perp f_i'  \ \ \text{and} \ \ \phi_3 \perp f_i' \ \ \text{on $L^2(I_1^{32})$ for all $1\leq i \leq 3$},
\end{equation} 
and 
\begin{equation}\label{32 component perturbation 3}
	\mathcal{A}_{32}-\int_{I_1^{31}}\phi_3 f_2'-\int_{I_1^{21}}f_3\phi_2'=\int_{I_1^{32}}\phi_3\phi_2'.
\end{equation}
We note that constructing $\phi_2$ and $\phi_3$ on $I_1^{32}$ satisfying \eqref{32 component perturbation 2} and \eqref{32 component perturbation 3} is equivalent to  constructing $\phi_2$ and $\phi_3$ on $I_1^{32}$ satisfying \eqref{32 component perturbation 1}. The construction of $\phi_2$ and $\phi_3$ on $I_1^{32}$ proceeds as follows. We first split $I_1^{32}$ into seven subintervals (denoted by $I(1), \cdots, I(7)$) of equal length. For $l=1, \cdots, 7$, construct perturbations $\xi_l$ and $\eta_l$ supported on $I(l)$ as in \cite[Lemma~6.6]{PSZ}, i.e., 
\begin{enumerate}
	\item $D^k \xi_l|_{\partial I(l)}=0$ for all $0\leq k \leq m$ where $\partial I(l)$ is the set of endpoints of $I(l)$,
	\item $|D^k \xi_l| \leq \tilde{C} \sqrt{\alpha}$ on $I(l)$ for all $0\leq k \leq m$,
	\item $\xi_l'\geq \tilde{C}^{-1} \sqrt{\alpha}(b-a)^{m-1}$ on the middle third of $I(l)$,
\end{enumerate}
and 
\begin{enumerate}
	\item $D^k \eta_l|_{\partial I(l)}=0$  for all $0\leq k \leq m$ where $\partial I(l)$ is the set of endpoints of $I(l)$,
	\item $|D^k \eta_l| \leq \tilde{C} \sqrt{\alpha}$ on $I(l)$ for all $0\leq k \leq m$,
	\item $\eta_l \geq \tilde{C}^{-1} \sqrt{\alpha}(b-a)^m$ on the middle third of $I(l)$.
\end{enumerate}
 Since we can construct $\xi_l$ and $\eta_l$ in exactly the same manner on each $I(l)$, we may assume that $L:=\int_{I(l)}\eta_l'\xi_l$ is a constant which is independent of $l$. 
 Set 
\[
\phi_2:=\sum_{l=1}^{7}v_l\xi_l \ \ \ \text{and} \ \ \ \phi_3:=\sum_{l=1}^{7}v_l'\eta_l
\]
where $v_k, v_k' \in \mathbb{R}$. Consider
\[
V_1:=\Big\{(v_1, \cdots, v_7) \in \mathbb{R}^7| \ \phi_2 \perp f_i' \ \ \ \text{on $L^2(I_1^{32})$ for all $1\leq i \leq 3$} \Big\}
\]
and 
\[
V_2:=\Big\{(v_1', \cdots, v_7') \in \mathbb{R}^7| \ \phi_3 \perp f_i' \ \ \ \text{on $L^2(I_1^{32})$ for all $1\leq i \leq 3$} \Big\}.
\]
Then dim$V_1$, dim$V_2 \geq 4$ by linear algebra. Therefore there exists $x=(x_1, \cdots, x_7) \in V_1\cap V_2$ such that $||x||_{\mathbb{R}^7}=1$. Choosing $v_l=v_l'=x_l$ for each $1 \leq l \leq 7$, we have
\begin{equation}\label{estimate of product of perturbations}
\int_{I_1^{32}}\phi_3'\phi_2=\sum_{l=1}^{7}x_l^2\int_{I(l)}\eta_l'\xi_l=L \sum_{l=1}^{7}x_l^2=L \geq \tilde{C}^{-1} \alpha (b-a)^{2m}	
\end{equation}
where we used the properties of $\xi_l$ and $\eta_l$ to obtain the last inequality. Since we have \eqref{case1 area remained} and \eqref{estimate of product of perturbations}, scaling $\phi_2$ and $\phi_3$ by a constant which is independent of $[a, b]$ finishes the proof.
\end{proof}

\begin{proposition}\label{case2}
 Suppose that
 \begin{equation}\label{case2 assumption 1 G3}
		\int_a^b |TF_1'|\geq (b-a)^m 
		\end{equation}
		and
	\begin{equation}\label{case2 assumption 2 G3}
	 \Big(\inf_{c_2 \in \mathbb{R}}\int_a^b |TF_2'-c_2TF_1'|\Big)\vee \Big(\inf_{c_3 \in \mathbb{R}}\int_a^b |TF_3'-c_3TF_1'|\Big) \geq (b-a)^m.
			\end{equation}
			 Then there exist $C^{\infty}$ perturbations $(\phi_i)_{i=1,2,3}$ on $[a, b]$ satisfying the properties in Proposition~\ref{New Perturbation}.
\end{proposition}
\begin{proof}
	 We first assume that 
	 \begin{equation}\label{case2 technical}
	 	\inf_{c_2 \in \mathbb{R}}\int_a^b |TF_2'-c_2TF_1'| \geq \inf_{c_3 \in \mathbb{R}}\int_a^b |TF_3'-c_3TF_1'|.
	 \end{equation}
	  Set  $N=16m^2$. We split $[a, b]$ into $N$ subintervals of equal length $(I_k)_{k=1, \cdots, N}$ so that $(b-a)/(16m^2) =|I_k|$ for all $k=1, \cdots, N$, i.e.,  each interval $I \subseteq [a, b]$ with $|I|\geq (b-a)/(4m^2)$ contains at least three of $I_k$ and the size of each $I_k$ is still comparable to $(b-a)$ with a uniform comparable constant which only depends on $m$.

	Now we start fixing vertical components. By the Markov inequality (Lemma~\ref{nice subinterval}), we can choose a subinterval $I_k$ such that $|TF_1'(x)| \geq \max_{y \in [a, b]}|TF_1'(y)|/2$ for any $x \in I_k$.  Note that by pugging in $c_1=c_2=0$ in \eqref{to fix 21 in G3}, we have 
\begin{equation}\label{wanna use Heisenberg}
A_{21}/V_{21} \leq \alpha.	
\end{equation}
 The above inequality \eqref{wanna use Heisenberg} together with the assumptions \eqref{F1 is the largest G3} and \eqref{F1 bigger than interval G3} allows us to correct $F_{21}$ using the technique in the Heisenberg group, see \cite[Lemma~6.5]{PSZ}. Namely, one can construct $\phi_2$ on $I_k$ (and $\phi_1\equiv\phi_3\equiv 0$ on $I_k$) such that
	\begin{enumerate} 
		\item $D^l\phi_2|_{\partial I_k}=0$ for all $0\leq l \leq m$ where $\partial I_k$ is the set of endpoints of $I_k$,
		\item $|D^l \phi_2| \leq \tilde{C} \alpha$ on $[a, b]$ for all $0\leq l \leq m$,
		\item we have 
		\begin{equation}\label{case2 21 fixed G3}
		\mathcal{A}_{21}=\int_{I_k} \phi_2 f_1'.	
		\end{equation}
		
	\end{enumerate}

	Secondly, by the Markov inequality (Lemma~\ref{nice subinterval}) again, choose $k'$ such that $k'\neq k$ and $|TF_1'(x)| \geq \max_{y \in [a, b]}|TF_1'(y)|/2$ for any $x \in I_{k'}$. Note that we can find such $I_{k'}$ with $k'\neq k$ since each interval $I \subseteq [a, b]$ with $|I|\geq (b-a)/(4m^2)$ contains at least three of $I_k$. By pugging in $c_1=c_3=0$ in \eqref{to fix 31 in G3}, we have 
\begin{equation}
A_{31}/V_{31} \leq \alpha.	
\end{equation}
Hence we can construct $\phi_3$ as above on $I_{k'}$ (and $\phi_1\equiv\phi_2\equiv 0$ on $I_{k'}$) such that
	\begin{enumerate} 
		\item $D^l\phi_3|_{\partial I_{k'}}=0$ for all $0\leq l \leq m$ where $\partial I_{k'}$ is the set of endpoints of $I_{k'}$,
		\item $|D^l \phi_3| \leq \tilde{C} \alpha$ on $[a, b]$ for all $0\leq l \leq m$,
		\item we have 
		\begin{equation}\label{case2 31 fixed G3}
			\mathcal{A}_{31}=\int_{I_{k'}} \phi_3 f_1'.
		\end{equation}
	\end{enumerate}
	
	Lastly we fix the $F_{32}$ component. Note that we need to construct  perturbations $\phi_2$ and $\phi_3$ on $[a, b]\setminus (I_k \cup I_{k'})$ such that 
	\begin{equation}
		\mathcal{A}_{32}=\int_a^bf_3\phi_2'+f_2'\phi_3+\phi_3\phi_2'.
	\end{equation}
	Moreover, we have to achieve this without changing other areas since they have been fixed already. Note that for any $c_2, c_3 \in A$
	\begin{align}\label{case2 key estimate0}
		|&\int_{I_k\cup I_{k'}}f_3'\phi_2-c_3\mathcal{A}_{21}|\notag \\
		\overset{\eqref{case2 21 fixed G3}}&{=}|\int_{I_k}f_3'\phi_2-c_3\int_{I_k}f_1'\phi_2| \notag \\
		&\leq |\int_{I_k}f_3'\phi_2-TF_3'\phi_2|+|\int_{I_k}TF_3'\phi_2-c_3TF_1'\phi_2|+|\int_{I_k}c_3TF_1'\phi_2-c_3\int_{I_k}f_1'\phi_2| \notag \\
		\overset{\eqref{good modulus G3 2}, \eqref{case2 assumption 2 G3}}&{\leq} C\alpha (b-a)^m\int_a^b |TF_3'-c_3TF_1'|
	\end{align}
	and similarly,
	\begin{equation}
		|\int_{I_k\cup I_{k'}}f_2'\phi_3-c_2\mathcal{A}_{31}|\overset{\eqref{case2 31 fixed G3}}{=}|\int_{I_{k'}
		}f_2'\phi_3-c_2\int_{I_{k'}}f_1'\phi_3| \leq C\alpha (b-a)^m\int_a^b |TF_2'-c_2TF_1'|.
	\end{equation}
Therefore, we have
\begin{align}\label{case2 key estimate1}
	|\int_{I_k\cup I_{k'}}&f_3'\phi_2 +\mathcal{A}_{32}-\int_{I_k\cup I_{k'}}f_2'\phi_3| \notag \\
	&\leq |\int_{I_k\cup I_{k'}}f_3'\phi_2-c_3\mathcal{A}_{21}|+|\mathcal{A}_{32}+c_3\mathcal{A}_{21}-c_2\mathcal{A}_{31}|+|c_2\mathcal{A}_{31}-\int_{I_k\cup I_{k'}}f_2'\phi_3| \notag \\
	\overset{\eqref{estimate after extension G3}}&{\leq} C \alpha (b-a)^m \Big(\int_a^b |TF_2'-c_2TF_1'|+|TF_3'-c_3TF_1'|\Big) 
\end{align}
for any $c_2, c_3 \in A$. This implies that
\begin{align}\label{case2 key estimate1 conclusion}
	|\int_{I_k\cup I_{k'}}&f_3'\phi_2 +\mathcal{A}_{32}-\int_{I_k\cup I_{k'}}f_2'\phi_3|  \notag \\
	&\leq C\alpha (b-a)^m \Big( \inf_{c_2 \in A} \int_a^b |TF_2'-c_2TF_1'|+ \inf_{c_3 \in A}\int_a^b |TF_3'-c_2TF_1'|\Big) \notag \\
	\overset{ \text{Lemma}~\ref{coefficients bounded polynomials}}&{\leq} C\alpha (b-a)^m \Big( \inf_{c_2 \in \mathbb{R}} \int_a^b |TF_2'-c_2TF_1'|+ \inf_{c_3 \in \mathbb{R}}\int_a^b |TF_3'-c_2TF_1'|\Big) \notag \\
	\overset{\eqref{case2 technical}}&{\leq} C\alpha (b-a)^m  \inf_{c_2 \in \mathbb{R}} \int_a^b |TF_2'-c_2TF_1'|
\end{align}

For each $i\in \{1, \cdots, N\}\setminus \{k, k'\}$, we construct a perturbation $\eta_i$ supported on $I_i$ as in \cite[Proof of Lemma~6.5]{PSZ}, i.e., 
	\begin{enumerate} 
		\item $D^l\eta_i|_{\partial I_i}=0$ for all $0\leq l \leq m$ where $\partial I_l$ is the set of endpoints of $I_i$,
		\item $|D^l \eta_i| \leq \tilde{C} \alpha$ on $I_i$ for all $0\leq l \leq m$,
		\item $|\eta_l| \geq \tilde{C}^{-1}\alpha (b-a)^m$ on the middle third of $I_i$.
	\end{enumerate}
	
For $l=1,2, 3,$ and $i\in \{1, \cdots, N\}\setminus \{k, k'\}=:W$, we set 
\[
\alpha^l_i:=\int_a^b \eta_i f_l'
\]
and
	\begin{equation}
		\alpha^l:=(\alpha^l_i)_{i \in W} \in \mathbb{R}^{N'}.
	\end{equation}
	where $N'=N-2$. In the following lemma, we make a specific choice of $c_2$ and $c_3$. Since we later use \eqref{case2 key estimate1} with this choice of  $c_2$ and $c_3$ (see \eqref{final perturbation nice bound G3}), it is crucial that these coefficients lie in $A$. We prove that these coefficients can be chosen to lie in $A$ below.

	\begin{lemma}\label{coefficients bounded G3}
	Set
	\begin{equation}
		c_2:=\frac{\langle \alpha^1, \alpha^2 \rangle_{\mathbb{R}^{N'}}}{||\alpha^1||_{\mathbb{R}^{N'}}^2} \ \ \text{and}  \ \ c_3:=\frac{\langle \alpha^1, \alpha^3 \rangle_{\mathbb{R}^{N'}}}{||\alpha^1||_{\mathbb{R}^{N'}}^2}.
	\end{equation}
		Then the constants $c_2$ and $c_3$ are bounded above by some constant $\tilde{C}$ which only depends on $K$ and $m$.
	\end{lemma}
	\begin{proof}
	We first estimate $|\alpha^l_i|$. By the properties of $\eta_i$ and the assumption, we have that
	\begin{align}\label{alpha estimate}
		|\alpha^l_i| &\leq |\int_a^b \eta_i f_l'| \leq |\int_a^b \eta_i f_l'-\eta_i TF_l'|+|\int_a^b \eta_i TF_l'| \notag \\
		\overset{\eqref{good modulus G3 2}}&{\leq} \tilde{C} \alpha (b-a)^m \Big(\alpha(b-a)^m+\int_a^b |TF_l'|\Big)\notag \\
		\overset{\eqref{F1 is the largest G3}, \eqref{case2 assumption 1 G3}}&{\leq} \tilde{C} \alpha (b-a)^m \int_a^b |TF_1'|
	\end{align}
	for every $1\leq l \leq 3$ and $1 \leq i \leq N'$. The inequality \eqref{alpha estimate} and Cauchy\nobreakdash-Schwarz inequality gives us that
	\begin{align}\label{constant bounded1}
		|\langle \alpha^1, \alpha^l \rangle_{\mathbb{R}^{N'}}| &\leq ||\alpha^1||_{\mathbb{R}^{N'}} ||\alpha^l||_{\mathbb{R}^{N'}} \notag \\
		\overset{\eqref{alpha estimate}}&{\leq} \tilde{C}(N')^2\alpha^2 (b-a)^{2m}\Big(\int_a^b|TF_1'|\Big)^2
	\end{align}
	for $l=2, 3$. We next estimate $||\alpha^1||_{\mathbb{R}^{N'}}$ from below. By the Markov inequality (Lemma~\ref{nice subinterval}) again, we can find $\tilde{k}$ such that $\tilde{k} \neq k, k'$ and  $|TF_1'(x)| \geq \max_{y \in [a, b]}|TF_1'(y)|/2$ for any $x \in I_{\tilde{k}}$. We remark again that we can find such $I_{\tilde{k}}$ since each interval $I \subseteq [a, b]$ with $|I|\geq (b-a)/(4m^2)$ contains at least three of $I_k$.  Since $TF_1'$ has the same sign on $I_{\tilde{k}}$ and so does $\eta_{\tilde{k}}$, we may assume that $\eta_{\tilde{k}} TF_1' \geq 0$ on $I_{\tilde{k}}$. Remark that $\tilde{k} \in W$ and  $\eta_{\tilde{k}} \geq \tilde{C}^{-1}\alpha (b-a)^m$ on the middle third of $I_{\tilde{k}}$, which implies that $|\eta_{\tilde{k}}TF_1'| \geq \tilde{C}^{-1}\alpha (b-a)^m \max_{x \in [a, b]}|TF_1'|$ on the middle third of $I_{\tilde{k}}$. Therefore, we have
	\begin{align}\label{estimate of eta TF1}
		\alpha (b-a)^m \int_a^b |TF_1'| &\leq  \alpha (b-a)^{m+1} \max_{x \in [a, b]}|TF_1'| \notag \\
		&\leq \tilde{C} \int_{I_{\tilde{k}}}|\eta_{\tilde{k}} TF_1'| \notag \\
		&\leq \tilde{C} \Big|\int_{I_{\tilde{k}}}\eta_{\tilde{k}}TF_1'\Big|,
	\end{align}
	where we used the fact that $\eta_{\tilde{k}} TF_1' \geq 0$ on $I_{\tilde{k}}$ to obtain the last inequality. Hence we obtain
	\begin{align}\label{estimate of eta TF1 continued}
		\alpha (b-a)^m \int_a^b |TF_1'| \overset{\eqref{estimate of eta TF1}}&{\leq} \tilde{C}\Big|\int_{I_{\tilde{k}}}\eta_{\tilde{k}}TF_1'\Big| \notag \\
		 &\leq \tilde{C}\Big|\int_{I_{\tilde{k}}}\eta_{\tilde{k}}TF_1'-\eta_{\tilde{k}}f_1'\Big|+\tilde{C}\Big|\int_{I_{\tilde{k}}}\eta_{\tilde{k}}f_1'\Big| \notag \\
		 \overset{\eqref{good modulus G3 2}}&{\leq} \tilde{C}\alpha^2 (b-a)^{2m}+\tilde{C}\Big|\int_{I_{\tilde{k}}}\eta_{\tilde{k}}f_1'\Big| \notag \\
		 \overset{\eqref{case2 assumption 1 G3}}&{\leq} \tilde{C}\alpha^2(b-a)^m \int_a^b |TF_1'|+\tilde{C}\Big|\int_{I_{\tilde{k}}}\eta_{\tilde{k}}f_1'\Big|.
	\end{align}
	The above estimate together with $\alpha < 1/(2C)\vee 1/(2\tilde{C})$  implies that 
	\begin{align}\label{constant bounded2}
		||\alpha^1||_{\mathbb{R}^{N'}} &\geq |\alpha^1_{\tilde{k}}| \geq  \tilde{C}^{-1} \alpha (b-a)^m\int_a^b |TF_1'|.
	\end{align}
	Therefore, by \eqref{constant bounded1} and \eqref{constant bounded2}, $c_2$ and $c_3$ are uniformly bounded by some constant $\tilde{C}$. Hence the lemma follows.
	\end{proof}
	  By the above lemma, we may assume that 
	$c_2, c_3~\in~A$. Write
	\[
	\alpha^{l}=\text{Pr}_{\langle\alpha^1\rangle}\alpha^l+\text{Pr}_{\langle \alpha^1\rangle^{\perp}}\alpha^l=\frac{\langle \alpha^1, \alpha^l \rangle_{\mathbb{R}^{N'}}}{||\alpha^1||_{\mathbb{R}^{N'}}^2}\alpha^1+\Big(\alpha^l-\frac{\langle \alpha^1, \alpha^l \rangle_{\mathbb{R}^{N'}}}{||\alpha^1||_{\mathbb{R}^{N'}}^2}\alpha^1\Big)=c_l \alpha^1 +(\alpha^l-c_l\alpha^1),
	\]
where $\text{Pr}_{\langle \alpha^1\rangle}$ and $\text{Pr}_{\langle \alpha^1\rangle^{\perp}}$ are canonical projections onto the span of $\alpha^1$, $\langle\alpha^1\rangle$, and its orthogonal space $\langle\alpha^1\rangle^{\perp}$, respectively. We note that as a consequence of the Markov inequality (Lemma~\ref{nice subinterval}) for $|TF_l' -c_l TF_1'|$ for $l=2,3$, one can find $k_l \in W$ such that $|TF_l'(x)-c_lTF_1'(x)| \geq \max_{y \in [a, b]}|TF_l'-c_lTF_1'(y)|/2$ for any $x \in I_{k_l}$. By doing a similar argument as in \eqref{estimate of eta TF1},  we have
\begin{align}\label{linear independence 0}
	0 &\neq \alpha (b-a)^m \int_a^b|TF_l'-c_l TF_1'| \leq \tilde{C}\Big|\int_{I_{k_l}}\eta_l(TF_l'-c_l TF_1')\Big| \notag \\
	&\leq  \tilde{C}\Big|\int_{I_{k_l}}\eta_{l}(TF_l'-c_l TF_1')-\eta_{l}(f_l'-c_lf_1')\Big|+\tilde{C}\Big|\int_{I_{k_l}}\eta_{l}(f_l'-c_lf_1')\Big| \notag \\
	\overset{\eqref{good modulus G3 2}}&{\leq} C \alpha^2 (b-a)^{2m}+\tilde{C}\Big|\int_{I_{k_l}}\eta_{l}(f_l'-c_lf_1')\Big| \notag \\
	\overset{\eqref{case2 assumption 2 G3}}&{\leq} C\alpha^2 (b-a)^m \int_a^b|TF_l'-c_l TF_1'|+\tilde{C}\Big|\int_{I_{k_l}}\eta_{l}(f_l'-c_lf_1')\Big|\notag \\
	&= C\alpha^2 (b-a)^m \int_a^b|TF_l'-c_l TF_1'|+\tilde{C}|\alpha^l_{k_l}-c_l\alpha^1_{k_l}|
\end{align}
where we obtain the last inequality from the assumption. Therefore, the above estimate \eqref{linear independence 0} together with $\alpha < 1/(2C)\vee 1/(2\tilde{C})$ gives us that
\begin{equation}\label{linear independence}
	||\alpha^l-c_l\alpha^1||_{\mathbb{R}^{N'}} \geq |\alpha^l_{k_l}-c_l\alpha^1_{k_l}| \geq \tilde{C}^{-1} \alpha (b-a)^m \int_a^b|TF_l'-c_l TF_1'|\neq 0
\end{equation}
for $l=2, 3$. This implies that $\alpha^l$ is linearly independent of $\alpha^1$ and $\text{Pr}_{\langle\alpha^1\rangle^{\perp}}\alpha^l \neq 0$ for $l=2, 3$. Hence we set $v_l:=\text{Pr}_{\langle\alpha^1\rangle^{\perp}}\alpha^l/||\text{Pr}_{\langle\alpha^1\rangle^{\perp}}\alpha^l||_{\mathbb{R}^{N'}}$ for $l=2,3$. Then we have $\langle \alpha^1, v_l \rangle_{\mathbb{R}^{N'}}=0$ and
\begin{align}\label{choice of vector is nice G3}
	\langle \alpha^l, v_l \rangle_{\mathbb{R}^{N'}}&=\langle \text{Pr}_{\langle\alpha^1\rangle}\alpha^l+\text{Pr}_{\langle \alpha^1\rangle^{\perp}}\alpha^l, v_l \rangle_{\mathbb{R}^{N'}} \notag \\
	&=||\text{Pr}_{\langle \alpha^1\rangle^{\perp}}\alpha^l||_{\mathbb{R}^{N'}}\notag \\
	&=||\alpha^l-c_l\alpha^1||_{\mathbb{R}^{N'}} \notag \\
	\overset{\eqref{linear independence}}&{\geq} \tilde{C}^{-1}\alpha (b-a)^m \int_a^b|TF_l'-c_l TF_1'|
\end{align}
for $l=2, 3$. 
Now we set 
\begin{equation}\label{def of pre final perturbation G3}
	\psi_2:=\sum_{i=1}^{N'} v_i^3 \eta_i \ \  \text{and} \ \ \psi_3:=\sum_{i=1}^{N'} v_i^2 \eta_i \ \  \text{on $[a, b]\setminus (I_k \cup I_{k'})$}
\end{equation}
where $v_l~:=~(v_1^l, \cdots, v_{N'}^l)$ for $l=2, 3$. Then we have
\begin{equation}
	\int_{[a, b]\setminus(I_k\cup I_{k'})}\psi_2 f_1'\overset{\eqref{def of pre final perturbation G3}}{=} \langle v_3, \alpha^1 \rangle_{\mathbb{R}^{N'}} = \int_{[a, b]\setminus(I_k\cup I_{k'})}\psi_3 f_1'\overset{\eqref{def of pre final perturbation G3}}{=} \langle v_2, \alpha^1\rangle_{\mathbb{R}^{N'}} = 0, \notag 
\end{equation}

\begin{equation}\label{case2 key estimate2}
	\int_{[a, b]\setminus(I_k\cup I_{k'})}\psi_2 f_3'\overset{\eqref{def of pre final perturbation G3}}{=}\langle \alpha^3, v_3 \rangle_{\mathbb{R}^{N'}} \overset{\eqref{choice of vector is nice G3}}{\geq} \tilde{C}^{-1}\alpha (b-a)^m \int_a^b|TF_3'-c_3 TF_1'|,
\end{equation}
and 
\begin{equation}\label{case2 key estimate3}
	\int_{[a, b]\setminus(I_k\cup I_{k'})}\psi_3 f_2'\overset{\eqref{def of pre final perturbation G3}}{=}\langle \alpha^2, v_2 \rangle_{\mathbb{R}^{N'}} \overset{\eqref{choice of vector is nice G3}}{\geq}  \tilde{C}^{-1}\alpha (b-a)^m \int_a^b|TF_2'-c_2 TF_1'|.
\end{equation}

Finally, set
\begin{equation}\label{def of final perturbation}
	\phi_3=\frac{\int_{I_k\cup I_{k'}} f_3'\phi_2 +\mathcal{A}_{32}-\int_{I_k\cup I_{k'}}f_2'\phi_3}{\int_{[a, b]\setminus(I_k\cup I_{k'})}\psi_3 f_2'}\psi_3
\end{equation}
and $\phi_1\equiv\phi_2 \equiv 0$ on $[a, b]\setminus (I_k \cup I_{k'})$. Then,
we obtain
\begin{align}
	\int_a^b f_3\phi_2'+f_2'\phi_3+\phi_3\phi_2' & =\int_{I_k \cup I_{k'}}f_3\phi_2' +\int_{I_k \cup I_{k'}}f_2'\phi_3+\int_{[a, b]\setminus (I_k \cup I_{k'})}f_2'\phi_3 \notag \\
	\overset{\eqref{def of final perturbation}}&{=}-\int_{I_k \cup I_{k'}}f_3'\phi_2+\int_{I_k \cup I_{k'}}f_2'\phi_3+\Big(\int_{I_k\cup I_{k'}} f_3'\phi_2 +\mathcal{A}_{32}-\int_{I_k\cup I_{k'}}f_2'\phi_3\Big) \notag \\
	&=\mathcal{A}_{32}.
\end{align}
Also, by the estimates \eqref{case2 key estimate1 conclusion} and \eqref{case2 key estimate3}, we have 
\begin{equation}\label{final perturbation nice bound G3}
	\max_{x \in [a,b]}|D^{k} \phi_3(x)|\leq C \alpha
\end{equation}
for $k=1. \cdots, m$, which implies that we are able to construct the desired perturbation $\phi_3$ with the modulus of continuity $\beta:=C_1 \alpha$ if \eqref{case2 technical} holds.  In the case that the converse of \eqref{case2 technical} holds, switch the role of $\phi_2$ and $\phi_3$ and use $\psi_2$ instead of $\psi_3$. This completes the proof of Proposition~\ref{case2}.
\end{proof}

\begin{proposition}\label{case5}
	 Suppose that
	 \begin{equation}\label{case5 assumption}
	 	(b-a)^m > \int_a^b |TF_1'|.
	 \end{equation}
 Then there exist $C^{\infty}$ perturbations $(\phi_i)_{i=1,2,3}$ on $[a, b]$ satisfying the properties in Proposition~\ref{New Perturbation}.
\end{proposition}
\begin{proof}
Since we assume that \eqref{F1 is the largest G3} and \eqref{case5 assumption} hold, we have
	\begin{equation}
		 \int_a^b |TF_2'|< (b-a)^m \ \ \text{and} \ \ \int_a^b |TF_3'| < (b-a)^m.
			\end{equation}
			Note also that by choosing $c_2=c_3=0$ in \eqref{estimate after extension G3}, we have
			\begin{equation}\label{case5 key estimate0}
				|\mathcal{A}_{ij}| \leq  C\alpha (b-a)^{2m},
			\end{equation}
			 see also \cite[(6.8)]{PSZ}. The construction of perturbations $\phi_i$ in this case is similar to the last part of Proposition~\ref{Case1}.
	 Split $[a, b]$ into three pieces of equal length, $I^{21}, I^{31},$ and $I^{32}$. We will fix areas one by one using these three intervals. 
	We first fix the $F_{21}$ component. Split $I^{21}$ into seven subintervals (denoted by $I(1), \cdots, I(7)$) of equal length. For each subinterval $I(k)$, construct perturbations $\xi_k$ and $\eta_k$ as in the proof of Proposition~\ref{Case1}. Note that $L:=\int_{I(k)}\eta_k' \xi_k$ is a constant which is independent of $1\leq k \leq 7$.
	
	 Set
\[
\phi_1:=\sum_{k=1}^{7}v_k\xi_k, \ \ \ \phi_2:=\sum_{k=1}^{7}v_k'\eta_k,  \ \ \ \text{and} \ \ \ \phi_3\equiv 0
\]
on $I_{21}$ where $v_k, v_k' \in \mathbb{R}$. Consider
\[
V_1:=\Big\{(v_1, \cdots, v_7) \in \mathbb{R}^7|  \ \phi_1 \perp f_i' \ \ \ \text{on $L^2(I_1^{32})$ for all $1\leq i \leq 3$} \Big\}
\]
and 
\[
V_2:=\Big\{(v_1', \cdots, v_7') \in \mathbb{R}^7|  \ \phi_2 \perp f_i' \ \ \ \text{on $L^2(I_1^{32})$ for all $1\leq i \leq 3$} \Big\}.
\]
Then dim$V_1$, dim$V_2 \geq 4$ by linear algebra. Therefore there exists $x=(x_1, \cdots, x_7) \in V_1\cap V_2$ such that $||x||_{\mathbb{R}^7}=1$. Choosing $v_k=v_k'=x_k$, we have
\begin{equation}\label{case5 key estimate}
	\int_{I^{21}}\phi_2'\phi_1=\sum_{k=1}^{7}x_k^2\int_{I(k)}\eta_k'\xi_k=\int_{I(i)}\eta_i'\xi_i \geq \tilde{C}^{-1} \alpha (b-a)^{2m}
\end{equation}
Since we have \eqref{case5 key estimate0} and \eqref{case5 key estimate}, scaling $\phi_2$ and $\phi_3$ by a constant which is independent of $[a, b]$ finishes the proof.
Applying the technique in \cite[Proof of Lemma~6.6, Case~3]{PSZ} completes fixing the $F_{21}$ area without affecting the $F_{31}$ and $F_{32}$ components. One has to apply the exactly same technique on $I^{31}$ and $I^{32}$ to fix the $F_{31}$ and $F_{32}$ components, respectively.
\end{proof}

As we summarize how we constructed perturbations in the case where $\int_a^b |TF_1'|$ is the largest integral, we claim that the cases considered in Propositions~\ref{Case1}, \ref{case2} and \ref{case5} are sufficient. When $\int_a^b |TF_1'|$ is the largest integral and smaller than $(b-a)^m$, then we can simply construct perturbations component wise (Proposition~\ref{case5}). When $\int_a^b |TF_1'|$ is the largest integral and greater than $(b-a)^m$, we first examine if the quantity
\begin{equation}\label{quantity to look at}
	\Big(\inf_{c_2 \in \mathbb{R}} \int_a^b |TF_2'-c_2TF_1'|\Big) \vee \Big(\inf_{c_3 \in \mathbb{R}}\int_a^b |TF_3'-c_3TF_1'|\Big)
\end{equation}
is less than $(b-a)^m$ or not. We considered the cases based on \eqref{quantity to look at} as in Propositions~\ref{Case1} and \ref{case2}.

   Note that we change the way of making perturbations $\phi_i$ based on how big the quantity \eqref{quantity to look at} is. On the other hand, in any case, the $F_{21}$ component is always the first one that we fix by $\phi_2$. Then we next fix the $F_{31}$ component by $\phi_3$ without affecting the $F_{21}$ component that has been already corrected. Then finally we fix the $F_{32}$ component based on  \eqref{quantity to look at} using \eqref{key11}. Hence we finish constructing perturbations $\phi_1, \phi_2, \phi_3$ on $[a, b]$ when $\int_a^b |TF_1|$ is the largest integral.

We now assume that  $\int_a^b |TF_2'|$ is the largest one. If $\int_a^b |TF_2'| < (b-a)^m$, then doing exactly the same argument as in  Proposition~\ref{case5} gives the desired perturbations $\phi_1, \phi_2, \phi_3$ on $[a, b]$. So we assume that  $\int_a^b |TF_2'| \geq  (b-a)^m$. We next check if
\begin{equation}\label{quantity to look at2}
	\Big(\inf_{c_1} \int_a^b |TF_1'-c_1TF_2'|\Big) \vee \Big( \inf_{c_3}\int_a^b |TF_3'-c_3TF_2'|\Big)
\end{equation}
 is less than $(b-a)^m$ or not. In both cases, we can fix the $F_{21}$ component by $\phi_1$ and the $F_{32}$ component by $\phi_3$ as we corrected the $F_{21}$ and $F_{31}$ components in the case that $\int_a^b|TF_1'|$ is the largest integral. Then one can finally fix $F_{31}$ component in the same manner as in Propositions~\ref{Case1} and \ref{case2} based on how big \eqref{quantity to look at2} is and \eqref{key22}.

Similarly, assume that  $\int_a^b |TF_3'|$ is the largest one. If $\int_a^b |TF_3'| < (b-a)^m$, then doing exactly the same argument as in  Proposition~\ref{case5} gives the desired perturbations $\phi_1, \phi_2, \phi_3$ on $[a, b]$. We next check if
\begin{equation}\label{quantity to look at3}
	\Big( \inf_{c_1} \int_a^b |TF_1'-c_1TF_3'| \Big) \vee \Big( \inf_{c_2}\int_a^b |TF_2'-c_2TF_3'|\Big)
\end{equation}
is less than $(b-a)^m$ or not. In both cases, we can fix the $F_{31}$ component by $\phi_1$ and the $F_{32}$ component by $\phi_2$. Then finally fix the $F_{21}$ component based on how big \eqref{quantity to look at3} is and \eqref{key33}.

This completes the proof of Proposition~\ref{New Perturbation}.
\end{proof}
 
The remaining work to prove Proposition~\ref{Our Goal} is standard, see  \cite[Lemma~6.7 and Proposition~6.8]{PSZ}. We give the construction of the horizontal curve here but leave the details to interested readers.

For each $[a_l, b_l]$, we have constructed perturbations $\phi^l_1, \phi^l_2$ and $\phi^l_3$ on each $[a_l, b_l]$ satisfying all the properties in Proposition~\ref{New Perturbation}. Set 
\begin{equation}\label{final horizontal extension1}
	\hat{f}_i(x)  :=
     \begin{cases}
       f_i(x)+\phi^l_i(x) &\quad\text{if $x \in [a_l, b_l]$}\\
       f_i(x) &\quad\text{if $x \in K$} 
     \end{cases}
\end{equation}
for $i=1, 2, 3$ and 
\begin{equation}\label{final horizontal extension2}
	\hat{f}_{ij}(x)  :=
     \begin{cases}
       F_{ij}(a_l)+ 1/2\int_{a_l}^x \hat{f}_i\hat{f}_j'- \hat{f}_i'\hat{f}_j &\quad\text{if $x \in [a_l, b_l]$}\\
       F_{ij}(x) &\quad\text{if $x \in K$} 
     \end{cases}
\end{equation}
for $1\leq j<i\leq 3$. As in \cite[Lemma~6.7 and Proposition~6.8]{PSZ}, one can check that a map $\gamma : [-\text{min}K, \max K] \to \mathbb{G}_3$ defined by $\gamma_{i}:=\hat{f}_i$ for $i=1, 2, 3$ and $\gamma_{ij}:=\hat{f}_{ij}$ for $1 \leq j<i\leq 3$ is a $C^m$ horizontal curve such that for any $0 \leq k \leq m$, $\gamma^k_i|_{K}=F^k_i$ for all $1\leq i \leq 3$ and $\gamma^k_{ij}|_{K}=F^k_{ij}$ for all  $1 \leq j<i\leq 3$. This completes the proof of Proposition~\ref{Our Goal}.
\end{proof}

Collecting Lemma~\ref{horizontal polynomial}, Proposition~\ref{New Condition}, and Proposition~\ref{Our Goal}, we obtain a $C^m$ Whitney extension theorem for horizontal curves in $\mathbb{G}_3$ (Theorem~\ref{Cm Whitney G3}).

\section{Necessity: The Case $\mathbb{G}_r$}\label{Necessity Gr Section}
\subsection{The Generalized $A/V$ Condition in $\mathbb{G}_r$}
Let $r\geq 3$ and $m\geq 1$ be integers. In order to extend the generalized $A/V$ condition in Definition~\ref{Definition of generalized A/V} to the setting of $\mathbb{G}_r$ where $r\geq 3$, we first introduce some quantities. Let $(F_i)_{1 \leq i \leq r}$ and $(F_{ij})_{1\leq j<i\leq r}$ be jets of order $m$ on $K$. Fix $1 \leq j<i\leq r$. For collections of real numbers $c:=(c_k)_{k \in \{1, \cdots, r\}\setminus \{i, j\}}$, $\tilde{c}:=(\tilde{c}_k)_{k \in \{1, \cdots, r\}\setminus \{i, j\}}$ and $a, b \in K$, we set 
\begin{align}
	E_{ij}(a,b, c, \tilde{c}):=A_{ij}(a,b)&+\sum_{k=1}^{j-1} \tilde{c}_k A_{jk}(a,b)-\sum_{k>j, \ k\neq i}^r \tilde{c}_k A_{kj}(a,b) \notag \\
		&-\sum_{k=1, \ k\neq j}^{i-1} c_k A_{ik}(a,b)+ \sum_{k>i}^r c_k A_{ki}(a,b)\notag \\
		&+\sum_{k>n, \ k,n \in \{1, \cdots, r\}\setminus \{i, j\}}(c_n \tilde{c}_k-\tilde{c}_n c_{k})A_{kn}(a,b),
\end{align}
\begin{equation}
	\Delta_{i}(a, b, \tilde{c}):=\int_a^b|T_aF_i' -\sum_{k=1, k\neq i,j}^r \tilde{c}_k T_aF_k'|,
\end{equation}
and 
\begin{equation}
			\Delta_{j}(a, b, c):=\int_a^b|T_aF_j' -\sum_{k=1, k\neq i,j}^r c_k T_aF_k'|.
		\end{equation}

\begin{definition}\label{The generalized A/V in Gr}
	Let $(F_i)_{1 \leq i \leq r}$ and $(F_{ij})_{1\leq j<i\leq r}$ be jets of order $m$ on $K$. We say that $(F_i)_{1 \leq i \leq r}$ and $(F_{ij})_{1\leq j<i\leq r}$ satisfy the \emph{generalized $A/V$ condition} if for any compact set $A \subseteq \mathbb{R}$, the following holds: 
	\begin{equation}\label{The generalized A/V in Gr equation}
	\limsup_{\substack{ |a-b|\to 0 \\ a, b \in K}} \sup_{c_k, \tilde{c}_k \in A}\frac{E_{ij}(a,b, c, \tilde{c})}{(b-a)^{2m}+(b-a)^m\Big(\Delta_{i}(a, b, \tilde{c})+ \Delta_{j}(a, b, c)\Big)}=0
\end{equation} 
for all $1 \leq j<i\leq r$.
\end{definition}
\begin{remark}
	In the case when $r=3$, the above generalized $A/V$ condition (Definition~\ref{The generalized A/V in Gr}) coincides with that of the $\mathbb{G}_3$ case (Definition~\ref{Definition of generalized A/V}).
\end{remark}

\subsection{Our Main Result in $\mathbb{G}_r$}
The rest of the present paper is devoted to proving the following.

\begin{theorem}\label{Cm Whitney Gr}
	Let $K \subseteq \mathbb{R}$ be a compact set and $(F_{i})_{1\leq i \leq r}$ and $(F_{ij})_{1\leq j<i\leq r}$ be jets of order $m$ on $K$. 
Then there exists a $C^m$ horizontal curve in $\mathbb{G}_r$ such that $D^k \gamma_{i}|_{K}=F^k_{i}$ for every $0 \leq k \leq m$ and $1\leq i \leq r$ and $D^k \gamma_{ij}|_{K}=F^k_{ij}$ for every $0 \leq k \leq m$ and $1 \leq j<i\leq r$ if and only if
\begin{enumerate}
\item $(F_{i})_{1\leq i \leq r}$ and $(F_{ij})_{1\leq j<i\leq r}$ are Whitney fields of class $C^m$ on $K$,
\item for every $1 \leq k \leq m$ and $t\in K$ and all $1 \leq j <i\leq r$ we have
 \begin{equation}
\label{HorizAssume}
F_{ij}^{k}(t) = \mathcal{P}^k \left(F_{i}^0(t),F_{j}^0(t),F_{i}^1(t),F_{j}^1(t),\dots,F_{i}^{k}(t),F_{j}^{k}(t)\right),
\end{equation}
where $\mathcal{P}^k$ are the polynomials defined in \eqref{def of polynomials},
\item  and the jets $(F_{i})_{1\leq i \leq r}$ and $(F_{ij})_{1\leq j<i\leq r}$ satisfy the generalized $A/V$ condition \eqref{The generalized A/V in Gr equation}.
\end{enumerate}
\end{theorem}

\subsection{$C^m$ Horizontal Curves in $\mathbb{G}_r$ Satisfy the Generalized $A/V$ Condition}

 We first provide the following algebraic lemma.
\begin{lemma}\label{algebraic lemma}
Let $r\geq 3$. Let $(f_i)_{i=1, \cdots, r}$ be $C^1$ functions. For each $1 \leq j < i \leq r$ and collections of real numbers $(c_k)_{k \in \{1, \cdots, r\}\setminus \{i, j\} }$ and $(\tilde{c}_k)_{k \in \{1, \cdots, r\}\setminus \{i, j\}}$, set 
\begin{equation}
a_{ij}:=f_i f_j'-f_i'f_j	,
\end{equation}
\begin{equation}
			\epsilon_i:=f_i-\sum_{k=1, k\neq i,j}^r \tilde{c}_k f_k, \ \ \text{and} \ \ 
		\epsilon_j:=f_j-\sum_{k=1, k\neq i,j}^r c_k f_k.
		\end{equation}
 Then
	\begin{align}\label{algebraic equality}
		a_{ij}&+\sum_{k=1}^{j-1} \tilde{c}_k a_{jk}-\sum_{k>j, \ k\neq i}^r \tilde{c}_k a_{kj} -\sum_{k=1, \ k\neq j}^{i-1} c_k a_{ik}+ \sum_{k>i}^r c_k a_{ki} \notag \\
		&+\sum_{k>n, \ k,n \in \{1, \cdots, r\}\setminus \{i, j\}}(c_n \tilde{c}_k-\tilde{c}_n c_{k})a_{kn} \notag \\
		&=\epsilon_{j}'\epsilon_i-\epsilon_j \epsilon_i'.
	\end{align}	
\end{lemma}
\begin{proof}
	This lemma follows by expanding the right hand side of \eqref{algebraic equality}, i.e., 
	\begin{align}
		\epsilon_{j}'\epsilon_i &-\epsilon_j \epsilon_i' =f_i f_j'-f_i'f_j - \sum_{k=1, k\neq i,j}^r \tilde{c}_k f_j'f_k + \sum_{k=1, k\neq i,j}^r \tilde{c}_k f_j f_k' \notag \\
		 &- \sum_{k=1, k\neq i,j}^r c_k f_i f_k' + \sum_{k=1, k\neq i,j}^r c_k f_i' f_k + \sum_{ k,k' \in \{1, \cdots, r\}\setminus \{i, j\}}c_k\tilde{c}_{k'}(f_k'f_{k'}-f_kf_{k'}') \notag \\
		 &=a_{ij}+\sum_{k=1}^{j-1} \tilde{c}_k a_{jk}-\sum_{k>j, \ k\neq i}^r \tilde{c}_k a_{kj} -\sum_{k=1, \ k\neq j}^{i-1} c_k a_{ik}+ \sum_{k>i}^r c_k a_{ki}\notag \\
		 &+ \sum_{ k,k' \in \{1, \cdots, r\}\setminus \{i, j\}}c_k\tilde{c}_{k'}(f_k'f_{k'}-f_kf_{k'}').
	\end{align}
	The last term is transformed into
	\begin{align}
		\sum_{ k,k' \in \{1, \cdots, r\}\setminus \{i, j\}}c_k\tilde{c}_{k'}(f_k'f_{k'}-f_kf_{k'}') &= \sum_{ k<k' \in \{1, \cdots, r\}\setminus \{i, j\}}(c_k \tilde{c}_{k'})(f_{k'}f_k'-f_{k'}'f_{k})\notag \\
		 &+ \sum_{ k > k' \in \{1, \cdots, r\}\setminus \{i, j\}} (c_k \tilde{c}_{k'})(f_{k'}f_k'-f_{k'}'f_{k})\notag \\
		 &= \sum_{ k<k' \in \{1, \cdots, r\}\setminus \{i, j\}}(c_k \tilde{c}_{k'})a_{k'k} - \sum_{ k > k' \in \{1, \cdots, r\}\setminus \{i, j\}} (c_k \tilde{c}_{k'})a_{kk'}\notag \\
		 &=\sum_{k>n, \ k,n \in \{1, \cdots, r\}\setminus \{i, j\}}(c_n \tilde{c}_k-\tilde{c}_n c_{k})a_{kn}.\notag 
	\end{align}
	This completes the proof.
\end{proof}

\begin{proposition}\label{necessary Gr}
Let $\gamma$ be a $C^m$ horizontal curve in $\mathbb{G}_r$ and $ K \subseteq \mathbb{R}$ be a compact set. Then the jets obtained by $F_i=(D^k \gamma_i|_{K})_{k=0, \cdots, m}$ for $1 \leq i \leq r$ and $F_{ij}=(D^k \gamma_{ij}|_{K})_{k=0, \cdots, m}$ for $1 \leq j<i\leq r$ satisfy the generalized $A/V$ condition \eqref{The generalized A/V in Gr equation}.
\end{proposition}
\begin{proof}
	Use Lemma~\ref{algebraic lemma} and follow the arguments in Proposition~\ref{New Condition} and the standard estimates in \cite[Proof of Proposition~5.2]{PSZ}.
\end{proof}
From Lemma~\ref{horizontal polynomial} and Propsition~\ref{necessary Gr}, for every $C^m$ horizontal curve $\gamma: \mathbb{R}\to \mathbb{G}_r$ and any compact set $K \subseteq \mathbb{R}$, the jets obtained by $F_i=(D^k \gamma_i|_{K})_{k=0, \cdots, m}$ and $F_{ij}=(D^k \gamma_{ij}|_{K})_{k=0, \cdots, m}$ satisfy all of the three properties listed in Theorem~\ref{Cm Whitney Gr}. Hence the rest of the present paper is devoted to proving the converse, i.e., Whitney fields satisfying \eqref{HorizAssume} and the generalized $A/V$ condition are always extended by a horizontal curve in $\mathbb{G}_r$.

\section{Sufficiency: The Case $\mathbb{G}_r$}\label{Sufficiency Gr Section}

\subsection{Outline of the Proof}
Throughout this section, $m\geq 1$ is a fixed integer.  Let us first outline our proof compared with the case of $\mathbb{G}_3$. We explain it step-by-step.

\textbf{Step 1}: Let $(F_i)_{1 \leq i \leq r}$ and $(F_{ij})_{1 \leq j <i \leq r}$ be jets of order $m$ on a compact set $K$. Recall that our strategy in the $\mathbb{G}_3$ case was that after extending jets $(F_i)_{1 \leq i \leq r}$ on the horizontal layer, we perturb those extensions outside the compact set $K$ to correct vertical areas.  We follow this strategy again in $\mathbb{G}_r$.  In other words, we write $[\min K, \max K]\setminus K := \cup_{l=1}^{\infty} (a_l, b_l)$ where $(a_l, b_l)$ are disjoint open intervals.  Let us denote the $C^m$ Whitney extension of $F_i$ by $f_i$. Since we focus on each $[a_l, b_l]$, we denote it by $[a, b]$ for simplicity. After translating to the origin, our goal is to construct $C^{\infty}$ perturbations $\phi_1, \cdots, \phi_r$ on  $[a, b]$ that satisfy
\begin{equation}\label{outline key 1}
F_{ij}(b)-F_{ij}(a)=\frac{1}{2}\int_{a}^{b} (f_i+\phi_i)(f_j'+\phi_j')-(f_i'+\phi_i')(f_j+\phi_j)
\end{equation}
for all $1 \leq j <i\leq r$. Since these perturbations $\phi_1, \cdots, \phi_r$  are supposed to be zero at the endpoints of $[a, b]$, by integration by parts, \eqref{outline key 1} is transformed into
\begin{equation}\label{outline key 2}
	\mathcal{A}_{ij}:=F_{ij}(b)-F_{ij}(a)-\frac{1}{2}\int_{a}^{b} f_i f_j'-f_if_j'=\int_{a}^{b} \phi_i f_j'-\phi_j f_i' + \phi_i \phi_j'.
\end{equation}

\textbf{Step 2}:  In order to achieve \eqref{outline key 2} for all vertical areas $(\mathcal{A}_{ij})_{1 \leq j<i \leq r}$, we divide each interval $[a, b]$ into finitely many pieces $(Q_{ij})_{1\leq j<i\leq r}$ which we call \emph{uniform good subsets}, see Definition~\ref{a good partition definition}. Each $Q_{ij}$ consists of subintervals of $[a,b]$ whose length is comparable with $|b-a|$ where the comparable constant is independent of $[a, b]$. This partition allows us to apply the Markov inequality and linear algebra techniques on each $Q_{ij}$. Also, we assume that the appropriate ordering \eqref{ordering1} and \eqref{ordering2} holds similar to how we assumed that $\int_a^b |TF_1'|$ is the largest integral in the case of $\mathbb{G}_3$.

 \textbf{Step 3}: Our main step is to prove Proposition~\ref{New Perturbation Gr} (construction of perturbations) by mathematical induction with respect to $r\geq 3$. Note that the starting case for the induction is $\mathbb{G}_3$ which was already covered in Section~\ref{Sufficiency G3 Section}. There are two cases that we need to examine:
\begin{itemize}
	\item \textbf{Case 1}: When $\inf_{c_1, \cdots, c_{k-1} \in \mathbb{R}} \int_{a}^{b}|TF_k'-c_1TF_1'-\cdots-c_{k-1}TF_{k-1}|\geq (b-a)^m$ holds  for all $1 \leq k \leq r-1$,
	\item \textbf{Case 2}: When $\inf_{c_1, \cdots, c_{k-1} \in \mathbb{R}} \int_{a}^{b}|TF_k'-c_1TF_1'-\cdots-c_{k-1}TF_{k-1}| < (b-a)^m$ holds for some $1 \leq k \leq r-1$.
 \end{itemize}

 \textbf{Step 4}:  In both cases, thanks to the induction hypothesis, we can find $C^{\infty}$ perturbations $\psi_1, \cdots, \psi_{r-1}$ easily and achieve \eqref{outline key 1} for all $1 \leq j<i \leq r-1$. This allows us to assume that  $\mathcal{A}_{ij}=0$ for all $1 \leq j<i\leq r-1$, see \eqref{reduction by mathematical induction case 1} and \eqref{reduction by mathematical induction case 2}. Therefore, our task is to construct another family of $C^{\infty}$ perturbations $\xi_1, \cdots, \xi_r$ on $[a, b]$ that satisfy 
\begin{equation}\label{outline key 3}
	\mathcal{A}_{rj}=\int_a^b \xi_r f_j'-\xi_j f_r' + \xi_r \xi_j'
\end{equation}
for all $1 \leq j \leq r-1$ \emph{without changing $\mathcal{A}_{ij}$ for all $1\leq j<i\leq r-1$}.

\textbf{Step 5}: Using the induction hypothesis of Proposition~\ref{New Perturbation Gr} again, we further achieve \eqref{outline key 3} for the vertical areas $\mathcal{A}_{r1}, \cdots, \mathcal{A}_{rr-2}$ \emph{while  holding \eqref{outline key 1} for $\mathcal{A}_{ij}$ for all $1\leq j<i\leq r-1$}, see Lemma~\ref{correcting everything except the last component case 1} and Lemma~\ref{correcting everything except the last component case 2}.  This is done by removing the components related to $r-1$. Namely, define Whitney fields $(G_i)_{1 \leq i \leq r-1}$ and $(G_{ij})_{1 \leq j<i \leq r-1}$ by
	 \begin{equation}
	 	G_i:=F_i \ \ \text{for $1 \leq i\leq r-2$} \ \ \text{and} \ \ G_{r-1}:= F_{r},
	 \end{equation}
	 and 
\begin{equation}
G_{ij}:=F_{ij} \ \ \text{for $1 \leq j<i\leq r-2$} \ \ \text{and} \ \ G_{r-1j}:= F_{rj} \ \ \text{for $1 \leq j\leq r-2$},
\end{equation}
	  and viewing them as Whitney fields in $\mathbb{G}_{r-1}$ that satisfy the generalized $A/V$ condition in $\mathbb{G}_{r-1}$. Thanks to the properties of perturbations listed in Proposition~\ref{New Perturbation Gr}, we can correct $\mathcal{A}_{r1}, \cdots, \mathcal{A}_{rr-2}$ \emph{without changing $\mathcal{A}_{ij}$ for all $1\leq j<i\leq r-1$}. This allows us to further assume that $\mathcal{A}_{rj}=0$ for all $1\leq j \leq r-2$.

\textbf{Step 6}: Our final task is to find perturbations $\xi_{r-1}$ and $\xi_r$ that satisfy \eqref{outline key 1} for the vertical area $\mathcal{A}_{rr-1}$ \emph{without changing all the other areas}.

 \textbf{Step 7}: We first treat \textbf{Case 1}. Recall that in $\mathbb{G}_3$, if $\int_a^b |TF_1'|$ is the largest integral and 
\[
\inf_{c_2 \in \mathbb{R}}\int_a^b|TF_2'-c_2TF_1'| \geq \inf_{c_3 \in \mathbb{R}}\int_a^b|TF_3'-c_3TF_1'|\vee(b-a)^m, 
\]
 then we can achieve \eqref{outline key 2} for $\mathcal{A}_{32}$ by \emph{only using $\phi_3$ (and setting $\phi_2\equiv 0$) after we constructed perturbations for $\mathcal{A}_{21}$ and $\mathcal{A}_{31}$}, see \eqref{def of final perturbation}. Similarly in $\mathbb{G}_r$, we can achieve \eqref{outline key 2} for $\mathcal{A}_{rr-1}$ by \emph{only constructing $\xi_r$} and setting $\xi_{r-1}\equiv 0$ under the appropriate assumption of ordering, see \eqref{ordering1} and \eqref{ordering2}.  The perturbation $\xi_r$ obviously do \emph{not} affect  $(\mathcal{A}_{ij})_{1\leq j<i\leq r-1}$. Furthermore, we can construct $\xi_r$ in a way that $\xi_r$ does not affect $\mathcal{A}_{r1}, \cdots, \mathcal{A}_{rr-2}$ by a discretization process. This is possible due to the fact that the remained area that needs to be created for $\mathcal{A}_{rr-1}$ is controlled by the generalized $A/V$ condition and \eqref{estimate after extension}.

\textbf{Step 8}: We next deal with \textbf{Case 2}. Recall that in the $\mathbb{G}_3$ case, if 
\[
\inf_{c_2 \in \mathbb{R}} \int_a^b |TF_2'-c_2TF_1'| \vee \inf_{c_3 \in \mathbb{R}}\int_a^b |TF_3'-c_3TF_1'| < (b-a)^m,
\]
then we needed to use both $\phi_2$ and $\phi_3$ to correct $\mathcal{A}_{32}$ while we could make $\phi_2$ and $\phi_3$ orthogonal to $f_1', f_2'$ and $f_3'$, see Proposition~\ref{Case1}. 
Also in \textbf{Case 2}, we end up constructing nontrivial pertubations $\xi_{r-1}$ and $\xi_r$ in order for \eqref{outline key 2} to hold for $\mathcal{A}_{rr-1}$. However, we can find such $\xi_{r-1}, \xi_r$ that are orthogonal to all of $f_1', \cdots, f_r'$. Hence these perturbations do \emph{not} affect $(\mathcal{A}_{ij})_{1\leq j<i \leq r-1}$ and $(\mathcal{A}_{rj})_{1 \leq j \leq r-2}$.

\textbf{Step 9}: The final $C^{\infty}$ perturbations $\phi_1, \cdots, \phi_r$ are defined by the combination of $\psi_1, \cdots, \psi_{r-1}$ and $\xi_1, \cdots, \xi_r$, see \eqref{case 1 our final perturbation Gr} for \textbf{Case 1} and \eqref{case 2 our final perturbation Gr} for \textbf{Case 2}. This completes the proof of Proposition~\ref{New Perturbation Gr}.

\textbf{Step 10}: Finally, we construct a $C^m$ horizontal curve as in  the $\mathbb{G}_3$ case (\eqref{final horizontal extension1 Gr} and \eqref{final horizontal extension2 Gr}), which finishes the proof.

\subsection{Good Subsets}
Inspired by the case of $\mathbb{G}_3$, we first split every interval into several pieces in a nice way and construct perturbations on it later. The role of the parameter $R$ will appear in Proposition~\ref{New Perturbation Gr}.
\begin{definition}\label{a good partition definition}
	Let $r \geq 3$, $R \geq 0$ be integers and let $([a_l, b_l])_{l\in \mathbb{N}}$ be a family of disjoint intervals. We say that $(Q^l_{ij})_{1\leq j <i \leq r, l \in \mathbb{N}}$ are \emph{uniform $(r, R)$-good subsets} if there exist uniform constants $L_1, L_2 >0$ such that 
	\begin{enumerate}
	\item $Q^l_{ij}:= \sqcup_{k=1}^{L_1} I^l_{ij}(k)$ for all $1 \leq j<i \leq r$ and each $l \in \mathbb{N}$, where all $I^l_{ij}(k)$ are disjoint closed subintervals of $[a_l, b_l]$ except for endpoints,
			\item  we have that
			\begin{equation}
	|Q^l_{kn}\cap Q^l_{k'n'}|  =
       0 \ \ \ \text{if $k\neq k'$ or $n \neq n'$,}
\end{equation}

\item the lower bound $2(r+R) <  L_1$ holds,
\item we have that $|I^l_{ij}(k)|= L_2(b_l-a_l)$ for all $1 \leq j<i \leq r$ and all $1\leq k\leq L_1$,
\item and for  every $l \in \mathbb{N}$ and $1\leq j<i\leq r$, it holds that for any $I \subseteq [a_l, b_l]$ with $|I|\geq (b_l-a_l)/(4m^2)$, we have that $I^l_{ij}(k) \subseteq I $ for some $1 \leq k \leq L_1$.
		\end{enumerate}
		Here $|\cdot|$ is the Lebesgue measure.

\end{definition}
	\begin{remark}
	In the above definition, we do \emph{not} require $\cup_{1\leq j<i\leq r} Q_{ij}^l = [a_l, b_l]$, see Property~$1$ in Definition~\ref{a good partition definition}. This definition will be helpful when constructing perturbations later.
	\end{remark}

	We prove the existence and some properties of  uniform $(r, R)$-good subsets for a family of disjoint intervals $([a_l, b_l])_{l \in \mathbb{N}}$.
	
	\begin{lemma}\label{a good partition properties}
		Let $r \geq 3$, $R \geq 0$ be integers and let $([a_l, b_l])_{l \in \mathbb{N}}$ be a family of disjoint intervals. Then there exist  uniform $(r, R)$-good subsets. Moreover, for any $(r, R)$-good subsets $(Q^l_{ij})_{1 \leq j<i\leq r, l \in \mathbb{N}}$ with $L_1, L_2>0$, we have the following properties:
		\begin{enumerate}
		\item for any integer $R'>R$, there exist uniform $(r, R')$-good subsets $(\tilde{Q}^l_{ij})_{1 \leq j<i\leq r, l \in \mathbb{N}}$ with $\tilde{L}_1,\tilde{L}_2>0$ such that 
		\begin{itemize}
			\item $\tilde{Q}_{ij}^l \subseteq Q_{ij}^l$ for all $1 \leq j<i\leq r$ and all $l \in \mathbb{N}$,
			\item $\tilde{L}_1 = 2^M L_1$ and $\tilde{L_2} = L_2/2^M $ hold for some $M:=M(r, R')$,
		\end{itemize}
		\item there exist uniform $(r, R)$-good subsets $(P^l_{ij})_{1 \leq j<i\leq r, l \in \mathbb{N}}$  with $L_1', L_2'>0$ and $(\tilde{P}^l_{ij})_{1 \leq j<i\leq r, l \in \mathbb{N}}$  with $\tilde{L}_1, \tilde{L}_2>0$ such that
		\begin{itemize}
			\item  $Q^l_{ij}= P^l_{ij} \cup \tilde{P}^l_{ij}$ and $|P^l_{ij}\cap  \tilde{P}^l_{ij}|=0$ for all $1 \leq j<i \leq r$ and each $l \in \mathbb{N}$,
			\item $L_1'=\tilde{L}_1= L_1$ and $L_2'= \tilde{L}_2=L_2/2$ hold,
		\end{itemize}
 			\item for any $1 \leq k \leq r$, $(\tilde{Q}^l_{ij})_{1 \leq j <i \leq r-1, l \in \mathbb{N}}$ defined by
	 \begin{equation}
	\tilde{Q}^l_{ij}  :=
     \begin{cases}
       Q^l_{ij} &\quad\text{if $1\leq j<i<k$ and $l \in \mathbb{N}$,}\\
       Q^l_{i+1j}  &\quad \text{if $1\leq j \neq k<i\leq r$ and $l \in \mathbb{N}$},
     \end{cases}
\end{equation}
		 are  uniform $(r-1, R)$-good subsets with $L_1, L_2>0$.
		\end{enumerate}
	\end{lemma}
	\begin{proof}
		We first prove the existence of uniform $(r, R)$-good subsets. Since one has to do the same argument for each $[a_l, b_l]$, we remove the index $l \in \mathbb{N}$ and simply denote $[a_l, b_l]$,  $Q^l_{ij}$ and $I^l_{ij}(k)$ by $[a, b]$, $Q_{ij}$ and $I_{ij}(k)$, respectively. It will be clear from the proof that $L_1$ and $L_2$ chosen below are independent of $l \in \mathbb{N}$.
		Take a minimum number $L_1 \in \mathbb{N}$ such that both $L_1> 2(r+R)$ and $L_1 > 8m^2$ hold. We divide $[a, b]$ into $L_1$ subintervals of equal length, i.e., $[a, b]=\cup_{1\leq k \leq L_1} I(k)$ where $I(k)$ are disjoint except for endpoints. We further split each $I(k)$ into $r(r-1)/2$ subintervals of equal length, denoted by $(I_{ij}(k))_{1\leq j<i\leq r}$. Finally, set 
	\begin{equation}
		Q_{ij}:= \sqcup_{k=1}^{L_1} I_{ij}(k)
	\end{equation}
	for each $1\leq j<i\leq r$. Note that for all $1 \leq j <i \leq r$ and all $1\leq k \leq L_1$ we have
	\[
	|I_{ij}(k)|=L_2(b-a)\ \ \text{where} \ \  L_2:=\frac{1}{L_1}\frac{2}{r(r-1)}.
	\]
	 We only check Property 5 in Definition~\ref{a good partition definition} since the other properties are obvious. For any subinterval $I \subseteq  [a, b]$ with $|I| \geq (b-a)/(4m^2)$,  there exists $I(k)$ such that $I(k) \subseteq I$ for some $1 \leq k \leq L_1$ since $L_1>8m^2$. So it is clear that for every $1 \leq j<i \leq r$, $I_{ij}(k)$ is contained in $I$. This completes the proof of existence of uniform $(r, R)$-good subsets. 
	 
	 We next prove Property $1$ in the statement. Let $R'>R$ be an integer and let $(Q^l_{ij})_{1 \leq j<i\leq r, l \in \mathbb{N}}$  be uniform $(r, R)$-good subsets with $L_1, L_2>0$. Again we omit the index $l \in \mathbb{N}$.  We split each $I_{ij}(k)$ into two subintervals of equal size, denoted by $I^1_{ij}(k)$ and $I^2_{ij}(k)$. Note that 
	 \[
	 |I^1_{ij}(k)|=|I^2_{ij}(k)|=\frac{L_2}{2}(b-a).
	 \]
	 for all $1 \leq j<i\leq r$ and all $1\leq k \leq N$. For each $1 \leq j<i\leq r$, set 
	 \[
	 \tilde{Q}_{ij}:=\cup_{k=1}^{2L_1} \tilde{I}_{ij}(k),
	 \]
	 where $\tilde{I}_{ij}(k):=I^1_{ij}(k)$ for $1 \leq k \leq L_1$ and $\tilde{I}_{ij}(L_1+k):=I^2_{ij}(k)$ for $1 \leq k \leq L_1$. If $2(r+R')< \tilde{L}_1:=2L_1$, then we obtained desired uniform $(r, R')$-good subsets. If not, then we continue doing this procedure. In other words, take $M:=M(r, R') \in \mathbb{N}$ which is the smallest number satisfying $2(r+R')<2^M L_1$ and do the above argument $M$ times so that we get uniform $(r, R')$-good subsets with $\tilde{L}_1, \tilde{L}_2>0$ such that 
	 \[
	 \tilde{L}_1 = 2^M L_1 \ \ \text{and} \ \ \tilde{L}_2 = \frac{L_2}{2^M}(b-a).
 	 \]
 	 
 	 To prove Property $2$ in the statement, for any uniform $(r, R)$-good subsets with $L_1, L_2>0$,  we split each $I_{ij}(k)$ into $I^1_{ij}(k)$ and $I^2_{ij}(k)$ in the same way as above and setting
 	 \[
 	 P_{ij}:=\sqcup_{k=1}^{L_1} I^1_{ij}(k) \ \ \text{and} \ \ \tilde{P}_{ij}:=\sqcup_{k=1}^{L_1} I^2_{ij}(k).
 	 \]
 	 completes the proof. Since Property $3$ in the statement is obvious, it completes the proof.
	\end{proof}

\subsection{Proof of Sufficiency in $\mathbb{G}_r$}
\begin{proposition}\label{Our Ultimate Goal}
	Let $K \subseteq \mathbb{R}$ be compact and $(F_i)_{1\leq i \leq r}, (F_{ij})_{1\leq j<i \leq r}$ be jets of order $m$ on $K$. Assume that
	\begin{enumerate}
		\item $(F_i)_{1\leq i \leq r}, (F_{ij})_{1\leq j<i \leq r}$ are Whitney fields of class $C^m$.
		\item for every $1 \leq k \leq m$ and $t\in K$ and all $1 \leq j <i\leq r$ we have
 \begin{equation}
F_{ij}^{k}(t) = \mathcal{P}^k \left(F_{i}^0(t),F_{j}^0(t),F_{i}^1(t),F_{j}^1(t),\dots,F_{i}^{k}(t),F_{j}^{k}(t)\right),
\end{equation}
where $\mathcal{P}^k$ are the polynomials defined in \eqref{def of polynomials},
\item and the generalized $A/V$ condition \eqref{The generalized A/V in Gr equation} holds.
\end{enumerate}
Then there exists a $C^m$ horizontal curve $\gamma$ such that for each $0\leq k \leq m$, $D^k\gamma_i|_{K}=F^k_i$ for all $1\leq i \leq r$ and $D^k\gamma_{ij}|_{K}=F_{ij}^k$ for all $1\leq j<i\leq r$.
\end{proposition}
 \begin{proof}
We start the proof as in the case of $\mathbb{G}_3$. Let $I:=[\text{min}K, \max K]$ and $I\setminus K := \cup_{l=1}^{\infty} (a_l, b_l)$ where $(a_l, b_l)$ are disjoint open intervals. 

We fix arbitrary $l \in \mathbb{N}$. For each $1 \leq k \leq r-2$, choose $i_k \in \{1, \cdots, r\}$ such that
\begin{equation}
		\int_{a_l}^{b_l}|TF_{i_1}'| \geq \max_{\tilde{k}\in \{1, \cdots, r\}}\Big\{\int_{a_l}^{b_l}|TF_{\tilde{k}}'|\Big\} \notag
	\end{equation}
	and
\begin{equation}
		\inf_{c_{i_1}, \cdots, c_{i_{k}} \in \mathbb{R}} \int_{a_l}^{b_l}|TF_{i_{k+1}}'-\sum_{l=1}^{k}c_{i_l} TF_{i_l}'| \geq \max_{\tilde{k} \in \{1, \cdots, r\} \setminus \{i_1, \cdots, i_k\}} \Big\{ \inf_{c_{i_1}, \cdots, c_{i_{k}} \in \mathbb{R}}  \int_{a_l}^{b_l}|TF_{\tilde{k}}'-\sum_{l=1}^{k}c_{i_l} TF_{i_l}'|\Big\} \notag 
	\end{equation}
	In order to simplify our notation, we may assume that 
	\begin{equation}\label{ordering1}
		\int_{a_l}^{b_l}|TF_1'| \geq \max_{\tilde{k}\in \{1, \cdots, r\}}\Big\{\int_{a_l}^{b_l}|TF_{\tilde{k}}'|\Big\}
	\end{equation}
	and 
	\begin{equation}\label{ordering2}
		\inf_{c_{1}, \cdots, c_{k} \in \mathbb{R}}\int_{a_l}^{b_l}|TF_{k+1}'-\sum_{l=1}^{k}c_l TF_l'| \geq \max_{\tilde{k} \in \{k+1, \cdots, r\}} \inf_{c_{1}, \cdots, c_{k} \in \mathbb{R}}\Big\{\int_{a_l}^{b_l}|TF_{\tilde{k}}'-\sum_{l=1}^{k}c_l TF_l'|\Big\}
	\end{equation}
	for any $1\leq k \leq r-2$.

The essential part of this section is to prove the following construction of perturbations by mathematical induction with respect to $r\geq 3$. Since we construct perturbations on each $[a_l, b_l]$ individually, without loss of generality, we may assume that \eqref{ordering1} and \eqref{ordering2} hold whenever $l \in \mathbb{N}$ is fixed.

\begin{proposition}\label{New Perturbation Gr}
	Let $r \geq 3$, $R \geq 0$ be integers and let $(F_i)_{1\leq i \leq r}, (F_{ij})_{1\leq j<i \leq r}$ be jets with the properties given in Proposition~\ref{Our Ultimate Goal}. For any uniform $(r, R)$-good subsets $(Q^l_{ij})_{1 \leq j <i \leq r}$ with $L_1, L_2>0$ and any $C^m$ functions $f_1, \cdots, f_r$ where $D^kf_i|_{K}=F_i^k$ for all $1\leq i \leq r$ and $0\leq k \leq m$, there exists a modulus of continuity $\beta$ for which the following holds:  for each $l \in \mathbb{N}$ and  any $h_1, \cdots, h_R \in L^2(a_l, b_l)$, if \eqref{ordering1} and \eqref{ordering2} hold, then there exist $C^{\infty}$ perturbations $\phi_1, \cdots, \phi_r$ on $[a_l, b_l]$  such that 
	\begin{enumerate}
		\item $D^k \phi_1|_{\partial I^l_{ij}(n)}=\cdots=D^k \phi_r|_{\partial I^l_{ij}(n)}=0 $ for all $0\leq k \leq m$, all $1\leq n \leq L_1$ and all $1\leq j<i\leq r$,
		\item $\max_{[a_l, b_l]}\{ |D^k \phi_1|, \cdots, |D^k \phi_r|\} \leq \beta(b_l-a_l)$ for all $0\leq k \leq m$,
		\item $\phi_1 \equiv \cdots \equiv \phi_r \equiv 0$ on $[a_l, b_l]\setminus \cup_{1 \leq j<i \leq r}Q^l_{ij}$,
		\item If $F_{ij}(b_l)-F_{ij}(a_l)=\frac{1}{2}\int_{a_l}^{b_l}f_i f_j'-f_i'f_j$ for some $1 \leq j<i\leq r$, then $\phi_i\equiv \phi_j \equiv 0$ on $Q^l_{ij}$,
		\item If $\inf_{c_1, \cdots, c_{k-1} \in \mathbb{R}} \int_{a_l}^{b_l}|TF_k'-c_1TF_1'-\cdots-c_{k-1}TF_{k-1}|\geq (b_l-a_l)^m$ for some $1 \leq k \leq r-1$ (the above is understood as $\int_a^b |TF_1'| \geq (b_l-a_l)^m$ in the case of $k=1$), then for each $i>k$,
		\begin{itemize}
		\item $\phi_i \perp f_{\tilde{k}}'$ on $L^2(Q^l_{ik})$ for all $1\leq {\tilde{k}} \leq k-1$,
		\item $\phi_j\equiv 0$ on $Q^l_{ik}$ for all $j \neq i$,
		\end{itemize}
	
		\item  If $\inf_{c_1, \cdots, c_{k-1} \in \mathbb{R}}\int_{a_l}^{b_l}|TF_k'-c_1TF_1'-\cdots-c_{k-1}TF_{k-1}|< (b_l-a_l)^m$ for some $1 \leq k \leq r-1$, then for each $i>k$,
		\begin{itemize}
			\item $\phi_i, \phi_k \perp f_{\tilde{k}}'$ on $L^2(Q^l_{ik})$ for all $1\leq \tilde{k} \leq r$,
			\item $\phi_i, \phi_k \perp h_{\tilde{k}}$ on $L^2(Q^l_{ik})$ for all $1\leq {\tilde{k}} \leq R$,
			\item $\phi_j\equiv 0$ on $Q^l_{ik}$ for all $j \neq i, k$,
		\end{itemize}
		\item $F_{ij}(b_l)-F_{ij}(a_l)=\int_{a_l}^{b_l}(f_i+\phi_i)(f_j'+\phi_j')-(f_i'+\phi_i')(f_j+\phi_j)$ for all $1\leq j<i\leq r$.
	\end{enumerate}
\end{proposition}
\begin{proof}

	Take a modulus of continuity $\alpha$ so that for all $a, x \in K$, $0 \leq k \leq m$ and $1\leq i \leq r$, we have
\begin{equation}\label{good modulus Gr 1}
	|D^k F_i(x)-D^k T_a F_i(x)| \leq \alpha(|x-a|)|x-a|^{m-k},
\end{equation}
\begin{equation}\label{good modulus Gr 2}
	|D^k f_i(x)-D^k T_a F_i(x)|\leq \alpha(|x-a|)|x-a|^{m-k},
\end{equation}
and
\begin{equation}\label{good modulus Gr 3}
	|D^k f_i(x)-D^kf_i(a)| \leq \alpha(x-a).
\end{equation}
Moreover, by \eqref{The generalized A/V in Gr equation}, we further assume that for a fixed compact set $A \subseteq \mathbb{R}$, we have that for every $1\leq j < i \leq r$ and $a, b \in K$ 
\begin{equation}
	 \sup_{c_k, \tilde{c}_k \in A}\frac{|E_{ij}(a,b, c, \tilde{c})|}{(b-a)^{2m}+(b-a)^m\Big(\Delta_{i}(a, b, \tilde{c})+ \Delta_{j}(a, b, c)\Big)}\leq \alpha(b-a), 
\end{equation} 
where  $c=(c_k)_{k \in \{1, \cdots, r\} \setminus \{i, j\}}, \tilde{c}=(\tilde{c}_k)_{k \in \{1, \cdots, r\} \setminus \{i, j\}} \subseteq A$.

In the rest of this proof, we use $C$ to denote a constant which depends on $K$, $A$, $r$, $R$, $L_1$, $L_2$ and $m$. The constant $\tilde{C}$ is reserved to denote a constant which only depends on $K$, $r$, $R$, $L_1$, $L_2$ and $m$, i.e., not on $A$. The compact set $A$ is chosen to be an interval $[-\tilde{C}_1, \tilde{C}_1]$ where the constant $\tilde{C}_1$ is the largest possible $\tilde{C}$ in the proof. We also set $C_1$ to be the largest possible $C$ in the proof. After choosing $\tilde{C}_1$ and the compact set $A$, we can find the constant $C_1$. Then by the same reason in \cite[Remark~6.4]{PSZ}, we always assume that $\alpha(b_l-a_l) < 1/(2C_1)\vee 1/(2\tilde{C}_1)$ by removing finitely many intervals.

We start constructing perturbations on each interval $[a_l,b_l]$. Since we fix each interval and construct them, we simply denote $[a_l, b_l], A_{ij}(a_l,b_l), \alpha(b_l-a_l)$, $T_{a_l} F_i$, $E_{ij}(a_l, b_l, c, \tilde{c})$, $\mathcal{E}_{ij}(a_l, b_l, c, \tilde{c})$, $\Delta_i(a_l, b_l, \tilde{c})$, and $\Delta_j(a_l, b_l, c)$ by $[a, b], A_{ij}, \alpha$, $TF_i$, $E_{ij}(c, \tilde{c})$, $\mathcal{E}_{ij}(c, \tilde{c})$, $\Delta_i(\tilde{c})$, and $\Delta_j(c)$ respectively. It will be clear that the modulus of continuity $\beta$  obtained is independent of $l \in \mathbb{N}$.

	 As in \cite[Claim~6.3]{PSZ}, we may assume that $F_k(a)=F_{ij}(a)=0$ for $1\leq k \leq r$ and $1\leq j<i\leq r$.

	 For $1 \leq j<i\leq r$, set
\[
\mathcal{A}_{ij}:=F_{ij}(b)-F_{ij}(a)-\frac{1}{2}\int_a^bf_if_j'-f_i'f_j.
\]
Define  
\begin{align}
	\mathcal{E}_{ij}(c, \tilde{c}):=\mathcal{A}_{ij}+\sum_{k=1}^{j-1} \tilde{c}_k \mathcal{A}_{jk}&-\sum_{k>j, \ k\neq i}^r \tilde{c}_k \mathcal{A}_{kj} -\sum_{k=1, \ k\neq j}^{i-1} c_k \mathcal{A}_{ik}+ \sum_{k>i}^r c_k \mathcal{A}_{ki}\notag \\
		&+\sum_{k>n, \ k,n \in \{1, \cdots, r\}\setminus \{i, j\}}(c_n \tilde{c}_k-\tilde{c}_n c_{k})\mathcal{A}_{kn}.
\end{align}
for all $1 \leq j<i\leq r$. The argument similar to the proof of Proposition~\ref{New Condition} gives us the estimate
\begin{align}\label{estimate after extension}
|\mathcal{E}_{ij}(c, \tilde{c})|\leq C \alpha(b-a)^{2m}+ C \alpha (b-a)^m(\Delta_i( \tilde{c})+\Delta_j(c))
\end{align}
for any $c=(c_k)_{k \in \{1, \cdots, r\}\setminus \{i, j\} },\tilde{c}=(\tilde{c}_k)_{k \in \{1, \cdots, r\}\setminus \{i, j\}} \subseteq A$ and all $1\leq j<i\leq r$.
	
	Note that in the statement we have two parameters, $r \geq 3$ and $R \geq 0$. The case when $r=3$ and $R\geq 0$ can be easily verified from the straightforward modification of the proof of Proposition~\ref{New Perturbation}. Hence we use mathematical induction with respect to $r$. Let $r\geq 4$. We assume that Proposition~\ref{New Perturbation Gr} holds true up to $r-1$ and all $R \geq 0$. We verify the case for $r$ and for all $R \geq 0$.  We divide into two cases.
	
	\hfill
	
	 \textbf{Case 1}: When $\inf_{c_1, \cdots, c_{k-1} \in \mathbb{R}} \int_{a}^{b}|TF_k'-c_1TF_1'-\cdots-c_{k-1}TF_{k-1}|\geq (b-a)^m$ for all $2 \leq k \leq r-1$.

	For given uniform $(r, R)$-good subsets $(Q_{ij})_{1 \leq j <i \leq r}$ with $L_1, L_2>0$, take uniform $(r, R)$-good subsets $(P^1_{ij})_{1 \leq j<i\leq r}$ with $L_1', L_2'>0$ and $(P^2_{ij})_{1 \leq j<i \leq r}$ with $\tilde{L}_1, \tilde{L}_2>0$ as in Property~$2$ of Lemma~\ref{a good partition properties}. Note that $L_1'= \tilde{L}_1 = L^1$ and $L_2'= \tilde{L}_2= L_2/2$.   By Lemma~\ref{a good partition properties} again, $(P^1_{ij})_{1 \leq j<i\leq r-1}$ are uniform $(r-1, R)$-good subsets. In order to use the induction hypothesis, we choose
	\begin{itemize}
		\item jets to be $(F_i)_{1\leq i \leq r-1}$ and $(F_{ij})_{1\leq j<i \leq r-1}$,
		\item $C^m$ extensions of $F_1, \cdots, F_{r-1}$ to be $f_1 \cdots, f_{r-1}$,
		\item $(r-1, R)$-uniform good subsets to be $(P^1_{ij})_{1 \leq j<i \leq r-1}$.
	\end{itemize}
	  One can easily see that the above jets $(F_i)_{1\leq i \leq r-1}$ and $(F_{ij})_{1\leq j<i \leq r-1}$ can be viewed as Whitney fields in $\mathbb{G}_{r-1}$ and satisfy the ordering(\eqref{ordering1} and  \eqref{ordering2}) and the generalized $A/V$ condition in $\mathbb{G}_{r-1}$ by plugging $c_{r-1}=\tilde{c}_{r-1}=0$ in \eqref{The generalized A/V in Gr equation}. By the induction hypothesis with the above inputs, we can find a modulus of continuity $\beta$ and perturbations $\psi_1, \cdots, \psi_{r-1}$ on $[a, b]$  satisfying the properties in the statement of Proposition~\ref{New Perturbation Gr} with the above inputs. 
	
	We now set $\tilde{f}_i:=f_i+\psi_i$ and $\tilde{\alpha}:=\alpha+\beta$.
	Note that since $\psi_1 \equiv \cdots \equiv \psi_{r-1}\equiv 0$ on $[a, b]\setminus \cup_{1 \leq j<i \leq r-1} P^1_{ij}$,  we have 
	\begin{equation}\label{updating f does not change case 1}
		\tilde{f}_i \equiv f_i \ \ \text{on}  \ \ \cup_{1 \leq j<i \leq r} P^2_{ij}.
	\end{equation} 
	Since we will construct perturbations on $\cup_{1 \leq j<i \leq r-1} P^2_{ij}$ later, we may assume that 
	\begin{equation}\label{reduction by mathematical induction case 1}
	 	\mathcal{A}_{ij}=0 \ \ \text{for all $1\leq j<i\leq r-1$},
	 \end{equation}
	and we keep \eqref{good modulus Gr 1}, \eqref{good modulus Gr 2}, \eqref{good modulus Gr 3}, and \eqref{estimate after extension} by updating Whitney extensions $f_i$ and the modulus of continuity $\alpha$ to $\tilde{f}_i$ and $\tilde{\alpha}$, respectively. Note that if $\tilde{\alpha}(b-a) < 1/(2C_1)\vee 1/(2\tilde{C}_1)$ fails, then we remove finitely many intervals again to keep assuming that $\tilde{\alpha}(b-a) < 1/(2C_1)\vee 1/(2\tilde{C}_1)$.
	
	Using Property $3$ in Lemma~\ref{a good partition properties} with $k=r-1$, $(\tilde{P}^2_{ij})_{1 \leq j<i \leq r-1}$ defined by 
	  \begin{equation}
	\tilde{P}^2_{ij}  :=
     \begin{cases}
       P^2_{ij} &\quad\text{if $1\leq j<i\leq r-2$,}\\
       P^2_{rj}  &\quad \text{if $i=r-1$ and $1\leq j \leq r-2$, } 
     \end{cases}
\end{equation}
	are uniform $(r-1, R)$-good subsets. We also define Whitney fields $(G_i)_{1 \leq i \leq r-1}$ and $(G_{ij})_{1 \leq j<i \leq r-1}$ by
	 \begin{equation}\label{def of Gi 1}
	 	G_i:=F_i \ \ \text{for $1 \leq i\leq r-2$} \ \ \text{and} \ \ G_{r-1}:= F_{r},
	 \end{equation}
	 and 
\begin{equation}\label{def of Gi 2}
G_{ij}:=F_{ij} \ \ \text{for $1 \leq j<i\leq r-2$} \ \ \text{and} \ \ G_{r-1j}:= F_{rj} \ \ \text{for $1 \leq j\leq r-2$}.
\end{equation}
	 Then $(G_i)_{1 \leq i \leq r-1}$ and $(G_{ij})_{1 \leq j<i \leq r-1}$ can be viewed as Whitney fields in $\mathbb{G}_{r-1}$ and satisfy the generalized $A/V$ condition in $\mathbb{G}_{r-1}$ by plugging $c_{r-1}=\tilde{c}_{r-1}=0$ in \eqref{The generalized A/V in Gr equation}. Moreover, they also satisfy the ordering \eqref{ordering1} and \eqref{ordering2}. In order to use the induction hypothesis again, we choose
	\begin{itemize}
		\item jets to be $(G_i)_{1\leq i \leq r-1}$ and $(G_{ij})_{1\leq j<i \leq r-1}$,
		\item $C^m$ extensions of $G_1, \cdots, G_{r-1}$ to be $g_1:=\tilde{f}_1, \cdots, g_{r-2}:=\tilde{f}_{r-2}, g_{r-1}:=\tilde{f}_{r}$,
		\item $(r-1, R)$-uniform good partitions to be $(\tilde{P}^2_{ij})_{1 \leq j<i \leq r-1}$.
	\end{itemize}
	 By the induction hypothesis with the above inputs, we can find a modulus of continuity $\tilde{\beta}$ and $C^{\infty}$ perturbations $\tilde{\psi}_1, \cdots, \tilde{\psi}_{r-1}$ on $[a, b]$ satisfying the properties in the statement of Proposition~\ref{New Perturbation Gr} with the above inputs.

	  Since these perturbations $\tilde{\psi}_1, \cdots, \tilde{\psi}_{r-1}$ may affect the vertical areas for $1\leq j<i \leq r-1$ when we further update $\tilde{f}_1, \cdots, \tilde{f}_r$, we have to check the following.

	 \begin{lemma}\label{correcting everything except the last component case 1}
	 Set $\xi_i:=\tilde{\psi}_i$ for all $1 \leq i \leq r-2$,  $\xi_{r-1}\equiv 0$ and $\xi_r:=\tilde{\psi}_{r-1}$ on $[a, b]$.
	 	The above perturbations do not affect the vertical areas for all $1\leq j<i \leq r-1$ while correcting $(\mathcal{A}_{rj})_{1 \leq j \leq r-2}$, i.e., we have that for all $1 \leq j<i\leq r-1$,
	 	\begin{equation}\label{vertical areas not affected 1 case 1}
	 		0=\mathcal{A}_{ij}=\int_a^b \xi_i \tilde{f}_j'-\tilde{f}_i'\xi_j+ \xi_i \xi_j' \ \Big(\overset{\eqref{updating f does not change case 1}}{=}\int_a^b \xi_i f_j'-f_i'\xi_j+ \xi_i \xi_j' \Big),
	 	\end{equation}
	 	and for all  $1\leq j \leq r-2$,
	 	\begin{equation}\label{vertical areas not affected 2 case 1}
	 		\mathcal{A}_{rj}=\int_a^b \xi_r \tilde{f}_j'-\tilde{f}_r'\xi_j+ \xi_r \xi_j' \ \Big( \overset{\eqref{updating f does not change case 1}}{=} \int_a^b \xi_i f_j'-f_i'\xi_j+ \xi_i \xi_j' \Big).
	 	\end{equation}
	 \end{lemma}
	 \begin{proof}
	 It is obvious that \eqref{vertical areas not affected 2 case 1} holds for all $1\leq j \leq r-2$ by the induction hypothesis. So we only verify \eqref{vertical areas not affected 1 case 1} for all $1 \leq j<i\leq r-1$.

	 We first note that $\tilde{\psi}_1\equiv \cdots \equiv \tilde{\psi}_{r-2}\equiv 0$ on $\cup_{1 \leq j<i \leq r-2} \tilde{P}^2_{ij}$ since $\mathcal{A}_{ij}=0$ for all $1 \leq j < i \leq r-2$, see Property $4$ in Proposition~\ref{New Perturbation Gr}. Hence we check the behavior of perturbations $\tilde{\psi}_1, \cdots, \tilde{\psi}_{r-2}$ on $\cup_{1 \leq j \leq r-2} \tilde{P}^2_{r-1j}$.

	 	Fix arbitrary $1\leq j \leq r-2$. Since
	 	\[
	 	 \inf_{c_1, \cdots, c_{k-1} \in \mathbb{R}} \int_{a}^{b}|TG_k'-c_1TG_1'-\cdots-c_{k-1}TG_{k-1}'|\geq (b-a)^m
	 	\]
	 	 for all $1 \leq k \leq r-2$, from Property $5$ in Proposition~\ref{New Perturbation Gr}, we know that $\xi_1\equiv \cdots \equiv \xi_{r-2}\equiv 0$ on $ \tilde{P}^2_{r-1j}$. Hence we have $\xi_1\equiv \cdots \equiv \xi_{r-2}\equiv 0$ on $\cup_{1 \leq j \leq r-2} \tilde{P}^2_{r-1j}$ and conclude that $\xi_1 \equiv \cdots \equiv \xi_{r-1} \equiv 0$ on $[a, b]$.  This implies \eqref{vertical areas not affected 1 case 1}.
	 \end{proof}
	 
	 We now set $\tilde{\tilde{f}}_i:=\tilde{f}_i+\xi_i$ for $i=1, \cdots, r$.
	 By the above lemma, we may further assume that $\mathcal{A}_{rj}=0$ for all $1 \leq j \leq r-2$ by updating $(\tilde{f}_i)_{1 \leq i \leq r}$ and the modulus of continuity $\alpha$ to $(\tilde{\tilde{f}}_i)_{1 \leq i \leq r}$ and $\tilde{\tilde{\alpha}}:=\tilde{\alpha}+\tilde{\beta}$ again.

	 Note that $\xi_r\equiv 0$ on $P^2_{rr-1}$ since $P^2_{rr-1} \subseteq [a, b]\setminus \cup_{1 \leq j <i\leq r-1} (P^1_{ij} \cup \tilde{P}^2_{ij})$. The rest of the proof is devoted to correcting the last $\mathcal{A}_{rr-1}$ component by constructing a perturbation $\xi_r$ on $P^2_{rr-1}$.

	  Recall now that $\mathcal{A}_{ij}=0$ for all $1 \leq j <i\leq r-1$ and $\mathcal{A}_{rk}=0$ for all $1 \leq k \leq r-2$ since we updated $f_1, \cdots, f_r$ to $\tilde{\tilde{f}}_1, \cdots, \tilde{\tilde{f}}_r$. Hence we have that
	\begin{align}
		|\mathcal{A}_{rr-1}| &=|\mathcal{E}_{rr-1}(c, \tilde{c})| \notag \\
		\overset{\eqref{estimate after extension}}&{\leq} C \tilde{\tilde{\alpha}} (b-a)^{2m}+ C \tilde{\tilde{\alpha}} (b-a)^m(\Delta_r(\tilde{c})+\Delta_{r-1}(c)) \notag \\
		\overset{\textbf{Case 1}}&{\leq} C\tilde{\tilde{\alpha}} (b-a)^m(\Delta_r(\tilde{c})+\Delta_{r-1}(c))
	\end{align}
		for any $c, \tilde{c}\subseteq A$. By taking the infimum over $c, \tilde{c} \subseteq A$, the above implies that 
		\begin{align}\label{bound of the coeeficient of the last perturbation1}
			|\mathcal{A}_{rr-1}| &\leq C \tilde{\tilde{\alpha}} (b-a)^m \Big( \inf_{\tilde{c} \subseteq A}\Delta_r(\tilde{c})+ \inf_{c \subseteq A}\Delta_{r-1}(c)\Big)  \notag \\ 
			\overset{\text{Lemma~\ref{coefficients bounded polynomials}}}&{\leq} C \tilde{\tilde{\alpha}} (b-a)^m \Big( \inf_{\tilde{c} \subseteq \mathbb{R}}\Delta_r(\tilde{c})+ \inf_{c \subseteq \mathbb{R}}\Delta_{r-1}(c)\Big) \notag \\
			\overset{\eqref{ordering2}}&{\leq} C \tilde{\tilde{\alpha}} (b-a)^m  \inf_{c \subseteq \mathbb{R}}\Delta_{r-1}(c).
		\end{align}

	Note that $\tilde{\tilde{f}}_i=f_i$ on $P^2_{rr-1}$ for $i=1, \cdots, r$. Our goal is to construct a smooth perturbation $\xi_r$ on $P^2_{rr-1}$ such that
	\begin{align}\label{Case 1 Gr our goal for the last perturbation}
		\mathcal{A}_{rr-1}=\int_{P^2_{rr-1}} \xi_r \tilde{\tilde{f}}_{r-1}'\ \Big(=\int_{P^2_{rr-1}} \xi_r f_{r-1}'\Big) \ \ \text{and} \ \ \xi_r \perp \tilde{\tilde{f}}_1', \cdots, \tilde{\tilde{f}}_{r-2}' \  \ \text{on $P^2_{rr-1}$}.
	\end{align}

	To construct a perturbation $\xi_r$ on $P^2_{rr-1}=\sqcup_{k=1}^{\tilde{L}_1} I_{rr-1}(k)$, we first set a $C^{\infty}$ function $\eta_k$ supported on each $I_{rr-1}(k)$ as in \cite[Proof of Lemma~6.5]{PSZ}, i.e., for all $1 \leq k \leq \tilde{L}_1$ we have that
	\begin{enumerate} 
		\item $D^i\eta_k|_{\partial I_{rr-1}(k)}=0$  where $\partial I_k$ is the set of endpoints of $I_{rr-1}(k)$,
		\item $|D^i \eta_k| \leq \tilde{C} \tilde{\tilde{\alpha}}$ on $I_{rr-1}(k)$,
		\item $|\eta_k| \geq \tilde{C}^{-1}\tilde{\tilde{\alpha}} (b-a)^m$ on the middle third of $I_{rr-1}(k)$.
	\end{enumerate} 
	For any $l=1,\cdots, r-1$ and any $1 \leq k \leq \tilde{L}_1$, we set 
\[
\alpha^l_k:=\int_a^b \eta_k \tilde{\tilde{f}}_l'\ \Big( =\int_a^b \eta_k f_l'\Big)
\]
and
	\begin{equation}
		\alpha^l:=(\alpha^l_i)_{i =1, \cdots, N} \in \mathbb{R}^{N}.
	\end{equation}
	where $N:= \tilde{L}_1$. We consider the orthogonal basis of the linear span of $(\alpha^l)_{l=1}^{r-1}$ by the Gram-Schmidt process, i.e., 
	\begin{equation}\label{gram schmidt Gr}
		\tilde{\alpha}^1:=\alpha^1,\ \ \tilde{\alpha}^2:=\alpha^2-\text{Pr}_{\langle \alpha^1 \rangle}(\alpha^2), \ \cdots ,\ \tilde{\alpha}^{r-1}:=\alpha^{r-1}-\text{Pr}_{\langle \alpha^1, \cdots, \alpha^{r-2}\rangle}(\alpha^{r-1}). 
	\end{equation}
	Note that one can prove that $||\alpha^l||_{\mathbb{R}^N} \neq 0$ for all $1 \leq l \leq r-1$ in the same way as in the case of $\mathbb{G}_3$, see Lemma~\ref{coefficients bounded G3}. By Lemma~\ref{projection fact}, for each $1 \leq l \leq r-1$, $\tilde{\alpha}^l$ can be written as
	\begin{align}\label{the choice of coeeficients Gr}
		\tilde{\alpha}^l &=\alpha^l-\frac{\langle \alpha^l, \tilde{\alpha}^1 \rangle_{\mathbb{R}^N}}{||\tilde{\alpha}^1||^2_{\mathbb{R}^N}}\tilde{\alpha}^1-\cdots -\frac{\langle \alpha^l, \tilde{\alpha}^{l-1} \rangle_{\mathbb{R}^N}}{||\tilde{\alpha}^{l-1}||^2_{\mathbb{R}^N}}\tilde{\alpha}^{l-1} \notag  \\
		&=\alpha^l-\Big( \sum_{k=1}^{l-1}c_{k,l} \alpha_k \Big)
	\end{align}
	 where $c_{k, l}$ are coefficients obtained by expanding the above. We claim that all of the constants $c_{k, l}$ with $ 1\leq l\leq r-1$ and $1\leq k \leq l-1$ are uniformly bounded by some uniform constant $\tilde{C}$. For this sake, we have the following lemma.

	\begin{lemma}\label{bound on coefficients}
	There exists $\tilde{C}$ such that for every $1\leq l \leq r-1$ and $1 \leq k < l \leq r-1$, we have
		\begin{equation}\label{bound on coefficients equation Gr}
		\Big|\frac{\langle \alpha^l, \tilde{\alpha}^k \rangle_{\mathbb{R}^N}}{||\tilde{\alpha}^k||^2_{\mathbb{R}^N}} \Big| \leq \tilde{C}.
	\end{equation}  
	 In particular, there exists another constant $\tilde{C}$ (possibly larger than the constant $\tilde{C}$ in \eqref{bound on coefficients equation Gr} above) such that the constants $c_{k, l}$ with $ 1\leq l\leq r-1$ and $1\leq k \leq l-1$ in \eqref{the choice of coeeficients Gr} are uniformly bounded by $\tilde{C}$.
	\end{lemma}
	\begin{proof}
	The argument is similar to the proof of Lemma~\ref{coefficients bounded polynomials} but we provide the proof for the sake of completeness.
	We prove this by mathematical induction. The case when $1=k<l \leq r-1$ follows by the argument in the $\mathbb{G}_3$ case, see Lemma~\ref{coefficients bounded G3}. For a fixed $k>1$, suppose that 
	\begin{equation}\label{coefficient bounded Gr induction assumption}
		\Big|\frac{\langle \alpha^l, \tilde{\alpha}^i \rangle_{\mathbb{R}^N}}{||\tilde{\alpha}^i||^2_{\mathbb{R}^N}} \Big| \leq \tilde{C}
	\end{equation}  
	for all $1 \leq i\leq k < l \leq r-1$. 

	Our goal is to prove that 
	\begin{equation}
		\Big|\frac{\langle \alpha^l, \tilde{\alpha}^{k+1} \rangle_{\mathbb{R}^N}}{||\tilde{\alpha}^{k+1}||^2_{\mathbb{R}^N}} \Big| \leq \tilde{C}
	\end{equation}
	for any $k+1 < l \leq r-1$.

	Since $\tilde{\alpha}^{k+1} \in \langle \tilde{\alpha}^1, \cdots, \tilde{\alpha}^{k} \rangle^{\perp}$, we have that
	\begin{equation}\label{key bound for coefficient Gr00}
		\Big| \frac{\langle \alpha^l, \tilde{\alpha}^{k+1} \rangle_{\mathbb{R}^N}}{||\tilde{\alpha}^{k+1}||^2_{\mathbb{R}^N}}\Big| ||\tilde{\alpha}^{k+1}||_{\mathbb{R}^N}=||\text{Pr}_{\langle \tilde{\alpha}^{k+1} \rangle} \alpha^l||_{\mathbb{R}^N} \overset{\text{Lemma}~\ref{projection fact}}{\leq} ||\text{Pr}_{\langle \tilde{\alpha}^1, \cdots, \tilde{\alpha}^{k} \rangle^{\perp}} \alpha^l||_{\mathbb{R}^N}.
		\end{equation}
		
		By the assumption of mathematical induction, we have that
		\begin{align}\label{coefficient Gr lemma key 00}
			||\text{Pr}_{\langle \tilde{\alpha}^1, \cdots, \tilde{\alpha}^{k} \rangle^{\perp}} \alpha^l||_{\mathbb{R}^N} &\overset{\text{Lemma}~\ref{projection fact}}{=} ||\alpha^l - \text{Pr}_{\langle \tilde{\alpha}^1, \cdots, \tilde{\alpha}^{k} \rangle}(\alpha^l)||_{\mathbb{R}^N} \notag \\
			&\overset{\text{Lemma}~\ref{projection fact}}{=}||\alpha^l-\sum_{i=1}^{k}\frac{\langle \alpha^l, \tilde{\alpha}^i \rangle_{\mathbb{R}^N}}{||\tilde{\alpha}^i||^2_{\mathbb{R}^N}} \tilde{\alpha}^i||_{\mathbb{R}^N} \notag \\
			\overset{\eqref{coefficient bounded Gr induction assumption}}&{=}\inf_{c_1, \cdots, c_k \in [-\tilde{C}, \tilde{C}]} ||\alpha^l-\sum_{i=1}^{k}c_i \alpha^i||_{\mathbb{R}^N}. 
		\end{align}

 From the above and by Lemma~\ref{coefficients bounded polynomials}, we obtain
		\begin{align}\label{key bound for coefficient Gr0}
		||\text{Pr}_{\langle \tilde{\alpha}^1, \cdots, \tilde{\alpha}^{k} \rangle^{\perp}} \alpha^l||_{\mathbb{R}^N} \overset{\eqref{coefficient Gr lemma key 00}}&{=} \inf_{c_1, \cdots, c_k \in [-\tilde{C}, \tilde{C}]} ||\alpha^l-\sum_{i=1}^{k}c_i \alpha^i||_{\mathbb{R}^N} \notag \\
		&\leq \tilde{C}\inf_{c_1, \cdots, c_k \in [-\tilde{C}, \tilde{C}]} \max_{n=1, \cdots, N} \Big|\int_a^b \eta_n(\tilde{\tilde{f}}_l'-\sum_{i=1}^{k} c_i \tilde{\tilde{f}}_i')\Big| \notag \\
		\overset{\textbf{Case 1}, \eqref{good modulus Gr 2}}&{\leq} \tilde{C} \tilde{\tilde{\alpha}} (b-a)^m \inf_{c_1, \cdots, c_k \in [-\tilde{C}, \tilde{C}]}  \int_a^b |TF_l'-\sum_{i=1}^k c_i TF_i'| \notag \\
		\overset{\eqref{inf reduction}}&{\leq} \tilde{C} \tilde{\tilde{\alpha}} (b-a)^m \inf_{c_1, \cdots, c_k \in \mathbb{R}} \int_a^b |TF_l'-\sum_{i=1}^k c_i TF_i'| \notag \\
		\overset{\eqref{ordering1}, \eqref{ordering2}}&{\leq} \tilde{C} \tilde{\tilde{\alpha}} (b-a)^m \inf_{c_1, \cdots, c_k \in \mathbb{R}} \int_a^b |TF_{k+1}'-\sum_{i=1}^k c_i TF_i'| \notag \\
		&\leq \tilde{C} \tilde{\tilde{\alpha}} (b-a)^m \inf_{c_1, \cdots, c_k \in [-\tilde{C}, \tilde{C}]} \int_a^b |TF_{k+1}'-\sum_{i=1}^k c_i TF_i'|. 
		\end{align}

		Fix $c_1, \cdots, c_k \in [-\tilde{C}, \tilde{C}]$. By the Markov inequality (Lemma~\ref{nice subinterval}), there exists $1 \leq n \leq~N$ such that  $|TF_{k+1}'(x)-\sum_{i=1}^k c_i TF_i'(x)| \geq \max_{y \in [a, b]}|TF_{k+1}'(y)-\sum_{i=1}^k c_i TF_i'(y)|/2$ holds for any $x \in I_{rr-1}(n)$. We remark that we can find such $I_{rr-1}(n)$ thanks to Definition~\ref{a good partition definition}.  Since $TF_{k+1}'-\sum_{i=1}^k c_i TF_i'$ has the same sign on $I_{rr-1}(n)$ and so does $\eta_{n}$, we may assume that $\eta_{n} (TF_{k+1}'-\sum_{i=1}^k c_i TF_i') \geq 0$ on $I_{rr-1}(n)$. We remark that $\eta_{n} \geq \tilde{C}^{-1}\tilde{\tilde{\alpha}} (b-a)^m$ on the middle third of $I_{rr-1}(n)$, which implies that 
		\[|\eta_{n}(TF_{k+1}'-\sum_{i=1}^k c_i TF_i'| \geq \tilde{C}^{-1}\tilde{\tilde{\alpha}} (b-a)^m \max_{x \in [a, b]}|TF_{k+1}'-\sum_{i=1}^k c_i TF_i'|
		\]
		 on the middle third of $I_{rr-1}(n)$. Therefore, calculations similar to \eqref{estimate of eta TF1} and \eqref{estimate of eta TF1 continued} give us that
		\begin{align}\label{coefficient bound Gr lemma key 1}
		\tilde{\tilde{\alpha}} (b&-a )^m \int_a^b |TF_{k+1}'-\sum_{i=1}^k c_i TF_i'| \leq \tilde{C}\Big|\int_{I_{rr-1}(n)} \eta_n(TF_{k+1}'-\sum_{i=1}^k c_i TF_i')\Big| \notag \\
		\overset{\eqref{good modulus Gr 2}}&{\leq} C \tilde{\tilde{\alpha}}^2(b-a)^{2m} +\tilde{C}\Big|\int_{I_{rr-1}(n)}\eta_{n}(\tilde{\tilde{f}}_{k+1}'-\sum_{i=1}^k c_i \tilde{\tilde{f}}_i')\Big|\notag \\
		\overset{\textbf{Case 1}}&{\leq} C \tilde{\tilde{\alpha}}^2(b-a)^m \int_a^b |TF_{k+1}'-\sum_{i=1}^k c_i TF_i'|+\tilde{C}\Big|\int_{I_{rr-1}(n)}\eta_{n}(\tilde{\tilde{f}}_{k+1}'-\sum_{i=1}^k c_i \tilde{\tilde{f}}_i')\Big|.
		\end{align}
		The above estimate together with $\tilde{\tilde{\alpha}} < 1/(2C_1)$, we have that
		\begin{align}
			\tilde{\tilde{\alpha}} (b-a)^m \int_a^b |TF_{k+1}'-\sum_{i=1}^k c_i TF_i'|\overset{\eqref{coefficient bound Gr lemma key 1}}&{\leq} \tilde{C} \Big| \int_{I_{rr-1}(n)}\eta_{n}(\tilde{\tilde{f}}_{k+1}'-\sum_{i=1}^k c_i \tilde{\tilde{f}}_i') \Big|\notag \\
			&\leq \tilde{C}|\alpha^{k+1}_n-\sum_{i=1}^k c_i \alpha^i_n| \notag \\
			&\leq \tilde{C} ||\alpha^{k+1}-\sum_{i=1}^k c_i \alpha^i||_{\mathbb{R}^N}. \notag 
		\end{align}
		This  implies that 
		\begin{align}\label{key bound for coefficient Gr1}
			\tilde{\tilde{\alpha}} (b-a)^m \int_a^b |TF_{k+1}'-\sum_{i=1}^k c_i TF_i'| &\leq \tilde{C}\inf_{c_1, \cdots, c_k \in [-\tilde{C}, \tilde{C}]}||\alpha^{k+1}-\sum_{i=1}^k c_i \alpha^i||_{\mathbb{R}^N}\notag \\
			\overset{\eqref{gram schmidt Gr}}&{=}\tilde{C}||\tilde{\alpha}^{k+1}||_{\mathbb{R}^N},
		\end{align}
		for arbitrary $c_1, \cdots, c_k \in [\tilde{C}, \tilde{C}]$. Combining \eqref{key bound for coefficient Gr00}, \eqref{key bound for coefficient Gr0}, and \eqref{key bound for coefficient Gr1}, we obtain
		\begin{equation}
			\Big| \frac{\langle \alpha^l, \tilde{\alpha}^{k+1} \rangle_{\mathbb{R}^N}}{||\tilde{\alpha}^{k+1}||^2_{\mathbb{R}^N}}\Big|  \leq \tilde{C}.
				\end{equation}
				
\end{proof}
From the above lemma, we verified that the coefficients $|c_{k, l}|$ in \eqref{the choice of coeeficients Gr} are uniformly bounded by some constant $\tilde{C}$. Hence we may assume that $c_{k,l} \in A$. Set 
\begin{equation}
	v=(v_1, \cdots, v_N):=\frac{\text{Pr}_{\langle \alpha^1, \cdots, \alpha^{r-2} \rangle^{\perp}}(\alpha^{r-1})}{||\text{Pr}_{\langle \alpha^1, \cdots, \alpha^{r-2} \rangle^{\perp}}\alpha^{r-1}||_{\mathbb{R}^N}} \ \ \text{and } \ \  \tilde{\xi}_r:=\sum_{i=1}^N v_i \eta_i.
\end{equation}

By the choice of the vector $v$, we have that 
\begin{equation}\label{perturbation orthogonal to functions Gr}
	\int_{P^2_{rr-1}}\tilde{\xi}_r \tilde{\tilde{f}}_k'= \int_{P^2_{rr-1}}\tilde{\xi}_r f_k' = \langle v, \alpha^k \rangle_{\mathbb{R}^N}=0 
\end{equation}
for all $1\leq k \leq r-2$.

Also,
\begin{align}\label{bound of the coefficient of the last perturbation2}
	\int_{P^2_{rr-1}}\tilde{\xi}_r \tilde{\tilde{f}}_{r-1}' &=\int_{P^2_{rr-1}}\tilde{\xi}_r f_{r-1}'  =\langle v, \alpha^{r-1} \rangle_{\mathbb{R}^N} \notag \\
	\overset{\text{Lemma~\ref{projection fact}}}&{=}\langle v, \text{Pr}_{\langle \alpha^1, \cdots, \alpha^{r-2} \rangle^{\perp}}(\alpha^{r-1})+\text{Pr}_{\langle \alpha^1, \cdots, \alpha^{r-2} \rangle}(\alpha^{r-1}) \rangle_{\mathbb{R}^N} \notag \\
	&=||\text{Pr}_{\langle \alpha^1, \cdots, \alpha^{r-2} \rangle^{\perp}}(\alpha^{r-1})||_{\mathbb{R}^N} \notag \\
	&= ||\tilde{\alpha}^{r-1}||_{\mathbb{R}^N}\geq \tilde{C}^{-1}\tilde{\tilde{\alpha}} (b-a)^m \int_a^b|TF_{r-1}'-\sum_{i=1}^{r-2} c_i TF_i'|,
\end{align}
where the last inequality follows by the calculation used to obtain \eqref{key bound for coefficient Gr1}. Finally, we define $\xi_r$ on $P^2_{rr-1}$ by 
\begin{equation}\label{final perturbation Gr}
	\xi_r:=\frac{\mathcal{A}_{rr-1}}{\int_{P^2_{rr-1}}\tilde{\tilde{f}}_{r-1}'\tilde{\xi}_r }\tilde{\xi}_r.
\end{equation}
By \eqref{bound of the coeeficient of the last perturbation1} and \eqref{bound of the coefficient of the last perturbation2}, we have that
\begin{equation}
	\max_{x \in P^2_{rr-1}}|D^{k} \xi_r (x)|\leq C \tilde{\tilde{\alpha}},
\end{equation}
for $k=1, \cdots, m$. It is easy to check from \eqref{perturbation orthogonal to functions Gr} and \eqref{final perturbation Gr} that we have \eqref{Case 1 Gr our goal for the last perturbation}.

We finally construct $\phi_1, \cdots, \phi_r$ on $[a, b]$ by
\begin{equation}\label{case 1 our final perturbation Gr}
	\phi_i  :=
     \begin{cases}
       \psi_i &\quad\text{on $\cup_{1 \leq j <i \leq r} P^1_{ij}$,}\\
       \xi_i &\quad\text{on $\cup_{1\leq j<i\leq r} P^2_{ij}$,} \\
       0 &\quad\text{on $[a, b] \setminus \cup_{1 \leq j<i\leq r} Q_{ij}$,}
     \end{cases}
\end{equation}
for each $1\leq i \leq r$ where we set $\psi_r\equiv 0$ on $[a, b]$ for convenience. One needs to check that these perturbations $\phi_1, \cdots, \phi_r$  satisfy all the properties in the statement of Proposition~\ref{New Perturbation Gr} with the modulus of continuity $(C_1\vee \tilde{C}_1) \tilde{\tilde{\alpha}}$. We only check Property $5$ as it is straightforward to check the other properties. We first note that $ P^1_{ij} \cup P^2_{ij} = Q_{ij}$ for all $1 \leq j<i \leq r$. We fix arbitrary $1\leq k \leq r-1$ and $k<i \leq r$. Since we used the induction hypothesis to construct $\psi_1, \cdots, \psi_r$, we know that for all $1\leq \tilde{k} \leq k-1$
\begin{equation}\label{final perturbation check Gr 1}
	\int_{P^1_{ik}}\psi_i f_{\tilde{k}}' =0.
\end{equation}
We have the following list of properties for $\xi_1, \cdots, \xi_r$:
\begin{enumerate}
	\item $\xi_1 \equiv \cdots \equiv \xi_r \equiv 0$ on $\cup_{1 \leq j < i \leq r-1} P^2_{ij} $,
	\item $\xi_1 \equiv \cdots \equiv \xi_{r-1} \equiv 0$ on $\cup_{1 \leq j  \leq r-1} P^2_{rj}$,
	\item $\xi_r \perp \tilde{f}_{\tilde{k}} \ (\equiv f_{\tilde{k}})$ on $L^2(P^2_{rk})$ for all $1 \leq \tilde{k} \leq k-1$.
\end{enumerate}
Hence for all $1\leq \tilde{k} \leq k-1$,
\begin{equation}
	\int_{P^2_{ik}}\xi_i \tilde{f}_{\tilde{k}}'=\int_{P^2_{ik}}\xi_i f_{\tilde{k}}'=0.
\end{equation}
Collecting all of the above properties, for each $1\leq \tilde{k}\leq k-1$, we have
\begin{align}
	\int_{Q_{ik}}\phi_i f_{\tilde{k}}' &=\int_{P^1_{ik}}\phi_i f_{\tilde{k}}'+ \int_{P^2_{ik}}\phi_i f_{\tilde{k}}' =\int_{P^1_{ik}}\psi_i f_{\tilde{k}}'+ \int_{P^2_{ik}}\xi_i f_{\tilde{k}}'. \notag \\
	&=0
\end{align}
Therefore, $\phi_i \perp f_{\tilde{k}}'$ on $L^2(Q_{ik})$ for all $1 \leq \tilde{k} \leq k-1$. One can easily check that $\phi_j \equiv 0$ on $Q_{ik}$ for $j\neq i$. This completes the proof of \textbf{Case 1}.

\hfill

 \textbf{Case 2}: When $\inf_{c_1, \cdots, c_{k-1} \in \mathbb{R}} \int_{a}^{b}|TF_k'-c_1TF_1'-\cdots-c_{k-1}TF_{k-1}| < (b-a)^m$  for some $1 \leq k \leq r-1$.

 For given uniform $(r, R)$-good subsets $(Q_{ij})_{1 \leq j <i \leq r}$, take uniform $(r, R)$-good subsets $(P^1_{ij})_{1 \leq j<i\leq r}$ with $L_1', L_2'>0$ and $(P^2_{ij})_{1 \leq j<i \leq r}$ with $\tilde{L}_1, \tilde{L}_2>0$ as in Property $2$ of Lemma~\ref{a good partition properties}. By using Property $3$ and then Property $1$ in Lemma~\ref{a good partition properties}, we may assume that $(P^1_{ij})_{1 \leq j<i\leq r-1}$ are uniform $(r-1, R+1)$-good subsets. 
 In order to use the induction hypothesis, we specify the inputs. We choose
	\begin{itemize}
		\item jets to be $(F_i)_{1\leq i \leq r-1}$ and $(F_{ij})_{1\leq j<i \leq r-1}$,
		\item $C^m$ extensions of $F_1, \cdots, F_{r-1}$ to be $f_1 \cdots, f_{r-1}$,
		\item $(r-1, R+1)$-uniform good subsets to be $(P^1_{ij})_{1 \leq j<i \leq r-1}$,
		\item $L^2$ functions to be $\tilde{h}_1=h_1, \cdots, \tilde{h}_{R}:=h_{R}, \tilde{h}_{R+1}:=f_r'$.
	\end{itemize}
	 One can easily see that the above jets $(F_i)_{1\leq i \leq r-1}$ and $(F_{ij})_{1\leq j<i \leq r-1}$ can be viewed as Whitney fields in $\mathbb{G}_{r-1}$ and satisfy the ordering(\eqref{ordering1} and  \eqref{ordering2}) and the generalized $A/V$ condition in $\mathbb{G}_{r-1}$ by plugging $c_{r-1}=\tilde{c}_{r-1}=0$ in \eqref{The generalized A/V in Gr equation}. Using the induction hypothesis with the above inputs, we can find a modulus of continuity $\beta$ and perturbations $\psi_1, \cdots, \psi_{r-1}$ on $[a, b]$  satisfying the properties in the statement of Proposition~\ref{New Perturbation Gr} with the above inputs.

	Set $\tilde{f}_i:=f_i+\psi_i$ and $\tilde{\alpha}:=\alpha+\beta$. By updating Whitney extensions $f_i$ and the modulus of continuity $\alpha$ to $\tilde{f}_i$ and $\tilde{\alpha}$, we may assume that 
	\begin{equation}\label{reduction by mathematical induction case 2}
	 	\mathcal{A}_{ij}=0 \ \ \text{for all $1\leq j<i\leq r-1$}.
	 \end{equation}
	and we can keep \eqref{good modulus Gr 1}, \eqref{good modulus Gr 2}, \eqref{good modulus Gr 3}, and \eqref{estimate after extension}. Note that since $\psi_1 \equiv \cdots \equiv \psi_{r-1}\equiv 0$ on $[a, b]\setminus \cup_{1 \leq j<i \leq r-1} P^1_{ij}$, we have 
	\begin{equation}\label{updating f does not change case 2}
		\tilde{f}_i=f_i \ \ \text{on} \cup_{1 \leq j <i \leq r} P^2_{ij}.
	\end{equation}
	
	Using Property $3$ in Lemma~\ref{a good partition properties} with $k=r-1$, $(\tilde{P}^2_{ij})_{1 \leq j<i \leq r-1}$ defined by 
	  \begin{equation}
	\tilde{P}^2_{ij}  :=
     \begin{cases}
       P^2_{ij} &\quad\text{if $1\leq j<i\leq r-2$,}\\
       P^2_{rj}  &\quad \text{if $i=r$ and $1\leq j \leq r-2$, } 
     \end{cases}
\end{equation}
	are uniform $(r-1, R)$-good subsets. By Property $1$ in Lemma~\ref{a good partition properties} again, we may assume that $(\tilde{P}^2_{ij})_{1 \leq j<i\leq r-1}$ are uniform $(r-1, R+1)$-good subsets. We also define Whitney fields $(G_i)_{1 \leq i \leq r-1}$ and $(G_{ij})_{1 \leq j<i \leq r-1}$ by the same way as in \eqref{def of Gi 1} and \eqref{def of Gi 2}.  Then, again, $(G_i)_{1 \leq i \leq r-1}$ and $(G_{ij})_{1 \leq j<i \leq r-1}$ can be viewed as Whitney fields in $\mathbb{G}_{r-1}$ and satisfy the generalized $A/V$ condition in $\mathbb{G}_{r-1}$ by plugging $c_{r-1}=\tilde{c}_{r-1}=0$ in \eqref{The generalized A/V in Gr equation}. Moreover, they also satisfy the ordering \eqref{ordering1} and \eqref{ordering2}.
	In order to use the induction hypothesis, we choose
	\begin{itemize}
		\item jets to be $(G_i)_{1\leq i \leq r-1}$ and $(G_{ij})_{1\leq j<i \leq r-1}$,
		\item $C^m$ extensions of $G_1, \cdots, G_{r-1}$ to be $g_1:=\tilde{f}_1, \cdots, g_{r-2}:=\tilde{f}_{r-2}, g_{r-1}:=\tilde{f}_{r}$,
		\item $(r-1, R+1)$-uniform good subsets to be $(\tilde{P}^2_{ij})_{1 \leq j<i \leq r-1}$,
		\item $L^2$ functions to be $\tilde{h}_1=h_1, \cdots, \tilde{h}_{R}:=h_{R}, \tilde{h}_{R+1}:=f_{r-1}'$.
	\end{itemize}

	By the induction hypothesis with the above inputs, we can find the modulus of continuity $\tilde{\beta}$ and $C^{\infty}$ perturbations $\tilde{\psi}_1, \cdots, \tilde{\psi}_{r-1}$ satisfying the properties in the statement of Proposition~\ref{New Perturbation Gr} with the above inputs.

	As in \textbf{Case 1}, these perturbations $\tilde{\psi}_1, \cdots, \tilde{\psi}_{r-1}$ may affect the vertical areas for $1 \leq j<i\leq r-1$. Hence we have to check the following.

	 \begin{lemma}\label{correcting everything except the last component case 2}
	 Set $\xi_i:=\tilde{\psi}_i$ for all $1 \leq i \leq r-2$,  $\xi_{r-1}\equiv 0$ and $\xi_r:=\tilde{\psi}_{r-1}$ on $[a, b]$.
	 	The perturbations do not affect the vertical areas for  $1 \leq j<i \leq r-1$ while correcting $(\mathcal{A}_{rj})_{1 \leq k \leq r-2}$, i.e., we have that for all $1 \leq j<i\leq r-1$,
	 	\begin{equation}\label{vertical areas not affected 1}
	 		0=\mathcal{A}_{ij}=\int_a^b \xi_i \tilde{f}_j'-\tilde{f}_i'\xi_j+ \xi_i \xi_j' \ \Big(\overset{\eqref{updating f does not change case 2}}{=}\int_a^b \xi_i f_j'-f_i'\xi_j+ \xi_i \xi_j' \Big),
	 	\end{equation}
	 	and for all  $1\leq j \leq r-2$,
	 	\begin{equation}\label{vertical areas not affected 2}
	 		\mathcal{A}_{rj}=\int_a^b \xi_r \tilde{f}_j'-\tilde{f}_r'\xi_j+ \xi_r \xi_j' \ \Big(\overset{\eqref{updating f does not change case 2}}{=}\int_a^b \xi_i f_j'-f_i'\xi_j+ \xi_i \xi_j' \Big).
	 	\end{equation}
	 \end{lemma}
	 \begin{proof}
	 It is obvious that \eqref{vertical areas not affected 2} holds for all $1\leq j \leq r-2$ due to the induction hypothesis. So we only verify \eqref{vertical areas not affected 1} for all $1 \leq j<i\leq r-1$.

	 We first note that $\xi_1\equiv \cdots \equiv \xi_{r-2}\equiv 0$ on $\cup_{1 \leq j<i \leq r-2} \tilde{P}^2_{ij}$ since $\mathcal{A}_{ij}=0$ for all $1 \leq j < i \leq r-2$, see Property~$4$ in the statement. Hence we check the behavior of perturbations $\xi_1, \cdots, \xi_{r-2}$ on $\cup_{1 \leq j \leq r-2} \tilde{P}^2_{r-1j}$.
	 
	  We first take the minimum number $1\leq k_1 \leq r-1$ for which 
\begin{equation}\label{minimum number k case 2}
	\inf_{c_1, \cdots, c_{k_1-1} \in \mathbb{R}} \int_{a}^{b}|TF_{k_1}'-c_1TF_1'-\cdots-c_{k_1-1}TF_{k_1-1}'| < (b-a)^m
\end{equation}
holds. If $k_1=r-1$, then the proof of Lemma~\ref{correcting everything except the last component case 1} gives \eqref{vertical areas not affected 1} for all $1 \leq j<i \leq r-1$. Hence we may assume that $1\leq k_1 \leq r-2$. By Property $5$ in Proposition~\ref{New Perturbation Gr}, we have $\xi_1\equiv \cdots \equiv \xi_{r-1}\equiv 0$ on $\cup_{1 \leq j \leq k_1-1} \tilde{P}^2_{r-1j}$. We next fix arbitrary $k_1 \leq k< r-1$. By Property~$6$ in Proposition~\ref{New Perturbation Gr}, we have  that $\xi_n\equiv 0$ on $\tilde{P}^2_{r-1k}$ for all $1 \leq n\neq k < r-1$ and $\xi_k \perp f_l'$ on $\tilde{P}^2_{r-1k}$ for every $1 \leq l \leq r$ due to the fact that we have set $\tilde{h}_{R+1}:=f_{r-1}'$. Moreover, we claim that $\xi_i \xi_j' \equiv 0$ on $[a, b]$ for all $1 \leq j < i \leq r-1$. In fact, if $i=r-1$, the claim is obvious. Hence we assume that $i<r-1$. In this case,  $\xi_i \xi_j' \equiv 0$ on $\cup_{1 \leq l \leq r-2}\tilde{P}^2_{r-1l}$ by Properties $5$ and $6$. Hence the claim follows. Collecting these properties we have that for all $1 \leq j<i\leq r-1$,
\begin{align}
	\int_a^b \xi_i f_j'-f_i'\xi_j+ \xi_i \xi_j' &=\int_{\cup_{1 \leq l \leq r-2} \tilde{P}^2_{r-1l}}\xi_i  f_j'-f_i'\xi_j+ \xi_i \xi_j' \notag \\
	&=\int_{\cup_{1 \leq l \leq k_1-1} \tilde{P}^2_{r-1l}}\xi_i f_j'-f_i'\xi_j+ \xi_i \xi_j' +\int_{\cup_{k_1 \leq l \leq r-2} \tilde{P}^2_{r-1l}} \xi_i f_j'-f_i'\xi_j+ \xi_i \xi_j' \notag \\
	&=0.
\end{align}
	This completes the proof.
	 \end{proof}
	 
	 Set $\tilde{\tilde{f}}_i:=\tilde{f}_i+\xi_i$ for each $1\leq i \leq r$ and $\tilde{\tilde{\alpha}}:=\tilde{\alpha}+\tilde{\beta}$. 
	 By the above lemma, we may further assume that $\mathcal{A}_{rj}=0$ for all $1 \leq j \leq r-2$ by updating $\tilde{f}_1, \cdots, \tilde{f}_r$ and the modulus of continuity $\tilde{\alpha}$ to $\tilde{\tilde{f}}_1, \cdots, \tilde{\tilde{f}}_r$ and $\tilde{\tilde{\alpha}}$, respectively. Since $\tilde{\tilde{f}}_j=f_j$ on $[a,b]\setminus \cup_{1 \leq j<i\leq r-1}(P^1_{ij} \cup \tilde{P}^2_{ij})$ for every $1\leq j \leq r$, we construct further perturbations on $P^2_{rr-1}$.  The rest of the proof is devoted to correcting the last $\mathcal{A}_{rr-1}$ component by constructing  perturbations $\xi_{r-1}$ and $ \xi_r$ on $P^2_{rr-1}$.

By the assumption of \textbf{Case 2} (\eqref{minimum number k case 2}) and the ordering(\eqref{ordering1}, \eqref{ordering2}), we know that 
\begin{align}
	\inf_{c_1, \cdots, c_{k_1-1} \in \mathbb{R}}\int_{a}^{b}|TF_{r-1}' &-c_1TF_1'-\cdots-c_{k_1-1}TF_{k_1-1}| \notag \\
	&\leq \inf_{c_1, \cdots, c_{k_1-1} \in \mathbb{R}}\int_{a}^{b}|TF_{k_1}'-c_1TF_1'-\cdots-c_{k_1-1}TF_{k_1-1}| \notag \\
	&< (b-a)^m
\end{align}
 By Lemma~\ref{coefficients bounded polynomials}, take $c_1, \cdots, c_{k_1-1} \in [-\tilde{C}, \tilde{C}]$ such that 
\[
\int_{a}^{b}|TF_{r-1}'-c_1TF_1'-\cdots-c_{k_1-1}TF_{k_1-1}| < \tilde{C}(b-a)^m
\] 
and set $c_l=0$ for all $k_1 \leq l \leq r-2$. We also take $\tilde{c}_1, \cdots, \tilde{c}_{k_1-1} \in [-\tilde{C}, \tilde{C}]$ such that $\int_{a}^{b}|TF_{r}'-\tilde{c}_1TF_1'-\cdots-\tilde{c}_{k_1-1}TF_{k_1-1}| < \tilde{C}(b-a)^m$ and set $\tilde{c}_l=0$ for all $k_1 \leq l \leq r-2$.
In this case we have
\begin{align}\label{remained area case 2 Gr}
	|\mathcal{A}_{rr-1}|&=|\mathcal{E}_{rr-1}(c, \tilde{c})| \notag \\
	\overset{\eqref{estimate after extension}}&{\leq} C \tilde{\tilde{\alpha}}(b-a)^{2m}+ C \tilde{\tilde{\alpha}} (b-a)^m(\Delta_r(\tilde{c})+\Delta_{r-1}(c)) \notag \\
	&\leq  C \tilde{\tilde{\alpha}} (b-a)^{2m}.
\end{align}

From here, we repeat the technique that already appeared in Proposition~\ref{Case1}. Note that $P^2_{rr-1}=\sqcup_{k=1}^{\tilde{L}_1} I_{rr-1}(k)$. For each subinterval $I_{rr-1}(k)$ of $P^2_{rr-1}$, we construct smooth $\tilde{\eta}_k$ and $\eta_k$ supported on $I_{rr-1}(k)$ as in \cite[Lemma~6.6]{PSZ}, i.e., 
\begin{enumerate}
	\item $D^i \tilde{\eta}_k|_{\partial I_{rr-1}(k)}=0$ where $\partial I_{rr-1}(k)$ is the set of endpoints of $I_{rr-1}(k)$,
	\item $|D^i \tilde{\eta}_k| \leq \tilde{C} \sqrt{\tilde{\tilde{\alpha}}}$ on $I_{rr-1}(k)$,
	\item $\tilde{\eta}_k'\geq \tilde{C}^{-1} \sqrt{\tilde{\tilde{\alpha}}}(b-a)^{m-1}$ on the middle third of $I_{rr-1}(k)$,
\end{enumerate}
and 
\begin{enumerate}
	\item $D^i \eta_k|_{\partial I_{rr-1}(k)}=0$ where $\partial I_{rr-1}(k)$ is the set of endpoints of $I_{rr-1}(k)$,
	\item $|D^i \eta_k| \leq \tilde{C} \sqrt{\tilde{\tilde{\alpha}}}$ on $I_{rr-1}(k)$,
	\item $\eta_k \geq \tilde{C}^{-1} \sqrt{\tilde{\tilde{\alpha}}}(b-a)^m$ on the middle third of $I_{rr-1}(k)$.
\end{enumerate}
Since we can construct $\xi_k$ and $\eta_k$ in exactly the same manner on each $I_{rr-1}(k)$ and $|I_{rr-1}|$ is independent of $k$ by Definition~\ref{a good partition definition}, we may assume that $L:=\int_{I_{rr-1}(k)}\eta_k'\xi_k$ is a constant which is independent of $k$. Set $N:=\tilde{L}_1$ and 
\[
\xi_{r-1}:=\sum_{k=1}^{N}v_k\tilde{\eta}_k \ \ \text{and} \ \ \xi_{r}:=\sum_{k=1}^{N}v_k'\eta_k
\]
on $P^2_{rr-1}$ where $v=(v_k)_{k=1}^N, v':=(v_k')_{k=1}^N$. Consider
\[
V_{r-1}:=\Big\{v \in \mathbb{R}^N| \ \xi_{r-1} \perp f_l', h_k \ \text{for all $1 \leq l \leq r$ and all $1\leq k \leq R$ on $L^2(P^2_{rr-1})$} \Big\}
\]
and 
\[
V_{r}:=\Big\{v' \in \mathbb{R}^N| \ \xi_{r} \perp f_l', h_k \ \text{for all $1 \leq l \leq r$ and all $1 \leq k \leq R$ on $L^2(P^2_{rr-1})$} \Big\}.
\]
Recall here that $2(r+R)< N$, which implies that $N/2 < N-(r+R)$. Hence we have that dim$V_{r-1}$, dim$V_r \geq N-(r+R) > N/2 $ by linear algebra. Therefore, there exists $x=(x_1, \cdots, x_N) \in V_1\cap V_2$ such that $||x||_{\mathbb{R}^N}=1$. Choosing $v_k=v_k'=x_k$, we have
\begin{equation}\label{case Gr estimate 3}
	\int_{P^2_{rr-1}}\xi_r'\xi_{r-1}=\sum_{k=1}^{N}x_k^2\int_{I_{rr-1}(k)}\eta_k'\tilde{\eta}_k=\int_{I_{rr-1}(i)}\eta_i'\tilde{\eta}_i \geq \tilde{C}^{-1} \tilde{\tilde{\alpha}} (b-a)^{2m}.
\end{equation}
From \eqref{remained area case 2 Gr}, \eqref{case Gr estimate 3} and by scaling these perturbations $\phi_{r-1}$ and $\phi_r$ by some constant $C$, we have 
\begin{equation}\label{case 2 Gr after scaling}
	\mathcal{A}_{rr-1}= \int_{P^2_{rr-1}}\xi_r \xi_{r-1}'
\end{equation} 
and 
\begin{equation}
	\max_{x \in P^2_{rr-1}}\{|D^{k} \xi_r (x)|, |D^{k} \xi_{r-1} (x)|\} \leq C \tilde{\tilde{\alpha}},
\end{equation}
for $k=1, \cdots, m$.
We finally construct $\phi_1, \cdots, \phi_r$ by
\begin{equation}\label{case 2 our final perturbation Gr}
	\phi_i  :=
     \begin{cases}
       \psi_i &\quad\text{on $\cup_{1 \leq j <i \leq r} P^1_{ij}$,}\\
       \xi_i &\quad\text{on $\cup_{1\leq j<i\leq r} P^2_{ij}$,} \\
       0 &\quad\text{on $[a, b] \setminus \cup_{1 \leq j<i\leq r} Q_{ij}$,}
     \end{cases}
\end{equation}
for each $1\leq i \leq r$ where we set $\psi_r \equiv 0$ on $[a, b]$. One needs to check that these perturbations $\phi_1, \cdots, \phi_r$  satisfy all the properties in the statement of Proposition~\ref{New Perturbation Gr} with the modulus of continuity $(C_1\vee \tilde{C}_1) \tilde{\tilde{\alpha}}$. We only check Properties $5$ and $6$ in the statement of Proposition~\ref{New Perturbation Gr} as it is straightforward to check the other properties. We first note that $ P^1_{ij} \cup P^2_{ij} = Q_{ij}$ for all $1 \leq j<i \leq r$. Let $k_1 \in \mathbb{N}$ be the number taken in \eqref{minimum number k case 2}. 
We first fix arbitrary $1\leq k \leq k_1-1$ and $k<i \leq r$. For this pair of $k$ and $i$, we need to check Property $5$ in the statement of Proposition~\ref{New Perturbation Gr}. Since we used the induction hypothesis to construct $\psi_1, \cdots, \psi_r$, we know that for every  $1\leq \tilde{k} \leq k-1$ and $1\leq n \neq k \leq r$,
\begin{equation}\label{final perturbation check Gr 1}
	\int_{P^1_{ik}}\psi_i f_{\tilde{k}}' =0 \ \ \text{and} \ \ \psi_n \equiv 0 \ \ \text{on $P^1_{ik}$}.
\end{equation}
We have the following list of properties for $\xi_1, \cdots, \xi_r$:
\begin{enumerate}
	\item $\xi_1 \equiv \cdots \equiv \xi_r \equiv 0$ on $\cup_{1 \leq j < i \leq r-1} P^2_{ij} $,
	\item $\xi_1 \equiv \cdots \equiv \xi_{r-1} \equiv 0$ on $\cup_{1 \leq j  \leq k_1-1} P^2_{rj}$,
	\item $\xi_r \perp \tilde{f}_{\tilde{k}} \ (\equiv f_{\tilde{k}})$ on $L^2(P^2_{rk})$ for all $1 \leq \tilde{k} \leq k-1$.
\end{enumerate}
One can check from the above three properties that for all $1\leq \tilde{k} \leq k-1$,
\begin{equation}
	\int_{P^2_{ik}}\xi_i \tilde{f}_{\tilde{k}}'=\int_{P^2_{ik}}\xi_i f_{\tilde{k}}'=0.
\end{equation}
Collecting all of the above properties, for each $1\leq \tilde{k}\leq k-1$, we have
\begin{align}
	\int_{Q_{ik}}\phi_i f_{\tilde{k}}' &=\int_{P^1_{ik}}\phi_i f_{\tilde{k}}'+ \int_{P^2_{ik}}\phi_i f_{\tilde{k}}' =\int_{P^1_{ik}}\psi_i f_{\tilde{k}}'+ \int_{P^2_{ik}}\xi_i f_{\tilde{k}}' \notag \\
	&=0.
\end{align}
Therefore, $\phi_i \perp f_{\tilde{k}}'$ on $L^2(Q_{ik})$ for all $1 \leq \tilde{k} \leq k-1$.

We next fix arbitrary $k_1\leq k \leq r$ and $k<i \leq r$. For this pair of $i$ and $k$, we need to check Property $6$ in the statement of Proposition~\ref{New Perturbation Gr}.  By the induction hypothesis to construct $\psi_1, \cdots, \psi_r$, we know that 
\begin{equation}\label{Remark needed Gr}
	\psi_i, \psi_k \perp f_{\tilde{k}} \ \ \text{on $L^2(P^1_{ik})$} \ \ \text{for all $1 \leq \tilde{k} \leq r$} 
\end{equation}
and
\begin{equation}
	\psi_i, \psi_k \perp h_{\tilde{k}} \ \ \text{on $L^2(P^1_{ik})$} \ \ \text{for all $1 \leq \tilde{k} \leq R$}.
\end{equation}
Note that one can conclude that $\psi_i, \psi_k \perp f_{r}'$ in \eqref{Remark needed Gr} due to the fact that we have set $\tilde{h}_{R+1}:=f_r'$ when we used the induction hypothesis to construct $\psi_1, \cdots, \psi_{r-1}$.
Also, we have the following list of properties for $\xi_1, \cdots, \xi_r$:
\begin{enumerate}
	\item $\xi_1 \equiv \cdots \equiv \xi_r \equiv 0$ on $\cup_{1 \leq j < i \leq r-1} P^2_{ij} $,
	\item $\xi_1 \equiv \cdots \equiv \xi_{r-1} \equiv 0$ on $\cup_{1 \leq j  \leq k_1-1} P^2_{rj}$,
	\item $\xi_r, \xi_{k} \perp \tilde{f}_{\tilde{k}} \ (\equiv f_{\tilde{k}})$ on $L^2(P^2_{rk})$ for all $1 \leq \tilde{k} \leq r$,
	\item $\xi_r, \xi_{k} \perp h_{\tilde{k}}$ on $L^2(P^2_{rk})$ for all $1 \leq \tilde{k} \leq R$.
\end{enumerate}
Note that we have that $\xi_r, \xi_k \perp f_{r-1}'$ on $L^2(P^2_{rk})$ above due to the fact that we have set $\tilde{h}_{R+1}:=f_{r-1}'$ when we used the induction hypothesis to construct $\xi_1, \cdots, \xi_{r}$. The above four properties of $\xi_1, \cdots, \xi_r$ imply that $\xi_i, \xi_k \perp f_{\tilde{k}}$ on $L^2(P^2_{ik})$ for all $1 \leq \tilde{k} \leq r$ and $\xi_i, \xi_{k} \perp h_{\tilde{k}}$ on $L^2(P^2_{ik})$ for all $1 \leq \tilde{k} \leq R$. Therefore, for all $1 \leq \tilde{k} \leq r$ we have
\begin{equation}
	\int_{Q_{ik}}\phi_i f_{\tilde{k}}'=\int_{P^1_{ik}} \psi_i f_{\tilde{k}}'+\int_{P^2_{ik}} \xi_i f_{\tilde{k}}'=0 \ \ \text{and} \ \ \int_{Q_{ik}}\phi_k f_{\tilde{k}}'=\int_{P^1_{ik}} \psi_k f_{\tilde{k}}'+\int_{P^2_{ik}} \xi_k f_{\tilde{k}}'=0,
	\end{equation}
	and for all $1 \leq \tilde{k} \leq R$ we have
	\begin{equation}
	\int_{Q_{ik}}\phi_i h_{\tilde{k}}'=\int_{P^1_{ik}} \psi_i h_{\tilde{k}}'+\int_{P^2_{ik}} \xi_i h_{\tilde{k}}'=0 \ \ \text{and} \ \ \int_{Q_{ik}}\phi_k h_{\tilde{k}}'=\int_{P^1_{ik}} \psi_k h_{\tilde{k}}'+\int_{P^2_{ik}} \xi_k h_{\tilde{k}}'=0.
	\end{equation}
  So we conclude that $\phi_i, \phi_k \perp f_{\tilde{k}}'$ on $L^2(Q_{ik})$ for all $1 \leq \tilde{k} \leq r$ and $\phi_i, \phi_k \perp h_{\tilde{k}}$ on $L^2(Q_{ik})$ for all $1 \leq \tilde{k} \leq R$. This completes the proof of \textbf{Case 2}. So we finished the proof of Proposition~\ref{New Perturbation Gr}.
\end{proof}

The construction of a horizontal curve that extends the jets $(F_i)_{1 \leq i \leq r}$ and $(F_{ij})_{1\leq j<i \leq r}$ follows as in the $\mathbb{G}_3$ case. Namely, for each $[a_l, b_l]$, take $\phi^l_1, \cdots, \phi^l_r$ on each $[a_l, b_l]$ satisfying all the properties in Proposition~\ref{New Perturbation Gr}. Set 
 
\begin{equation}\label{final horizontal extension1 Gr}
	\hat{f}_i(x)  :=
     \begin{cases}
       f_i(x)+\phi^l_i(x) &\quad\text{if $x \in [a_l, b_l]$,}\\
       f_i(x) &\quad\text{if $x \in K$,} 
     \end{cases}
\end{equation}
for $i=1,\cdots, r$ and 
\begin{equation}\label{final horizontal extension2 Gr}
	\hat{f}_{ij}(x)  :=
     \begin{cases}
       F_{ij}(a_l)+ 1/2\int_{a_l}^x \hat{f}_i \hat{f}_j'- \hat{f}_i'\hat{f}_j &\quad\text{if $x \in [a_l, b_l]$,}\\
       F_{ij}(x) &\quad\text{if $x \in K$,} 
     \end{cases}
\end{equation}
for each $1\leq j<i\leq r$. As in \cite[Lemma~6.7 and Proposition~6.8]{PSZ}, one can check that a map $\gamma : [-\text{min}K, \max K] \to \mathbb{G}_r$ defined by $\gamma_{i}:=\hat{f}_i$ for $i=1, \cdots, r$ and $\gamma_{ij}:=\hat{f}_{ij}$ for $1 \leq j<i\leq r$ is a $C^m$ horizontal curve such that for any $0 \leq k \leq m$, $D^k \gamma_i|_{K}=F^k_i$ for all  $1\leq i \leq r$ and $D^k\gamma_{ij}|_{K}=F^k_{ij}$ for all $1 \leq j<i\leq r$.   This completes the proof of Proposition~\ref{Our Ultimate Goal}.
 \end{proof}

}

\end{document}